\documentclass[10pt]{article}

 \usepackage[a4paper, total={39pc,53pc}]{geometry}
\usepackage[utf8]{inputenc}   
\usepackage[T1]{fontenc}      
\usepackage{amssymb,mathrsfs}
\usepackage{amsmath,amsthm} 
\usepackage{amsfonts}
\usepackage{upgreek}
\usepackage{nicefrac}
\usepackage{authblk}
\usepackage{graphicx}
\usepackage{hyperref}
\DeclareMathAlphabet{\mathpzc}{OT1}{pzc}{m}{it}
\hypersetup{
  colorlinks = true,
  citecolor = blue,
}
\usepackage{color}
\usepackage{stmaryrd}
\usepackage{enumitem}
\usepackage{url}
\usepackage{graphicx} 
\usepackage{lmodern} 
\usepackage{xcolor}
\usepackage{bbm}
\usepackage{ifthen}
\usepackage{xargs}
\usepackage[disable]{todonotes}
\usepackage{aliascnt}
\usepackage{cleveref}
\usepackage{autonum}

\newtheorem{theorem}{Theorem}
\crefname{theorem}{theorem}{Theorems}
\Crefname{Theorem}{Theorem}{Theorems}
\newtheorem*{lemma_nonumber*}{Lemma}
\newaliascnt{lemma}{theorem}
\newtheorem{lemma}[lemma]{Lemma}
\aliascntresetthe{lemma}
\crefname{lemma}{lemma}{lemmas}
\Crefname{Lemma}{Lemma}{Lemmas}
\newaliascnt{corollary}{theorem}

\aliascntresetthe{corollary}
\crefname{corollary}{corollary}{corollaries}
\Crefname{Corollary}{Corollary}{Corollaries}
\newaliascnt{proposition}{theorem}
\newtheorem{proposition}[proposition]{Proposition}
\aliascntresetthe{proposition}
\crefname{proposition}{proposition}{propositions}
\Crefname{Proposition}{Proposition}{Propositions}
\newaliascnt{example}{theorem}
\newtheorem{example}[example]{Example}
\aliascntresetthe{example}
\crefname{example}{example}{examples}
\Crefname{Example}{Example}{Examples}
\newaliascnt{remark}{theorem}
\newtheorem{remark}[remark]{Remark}
\aliascntresetthe{remark}
\crefname{remark}{remark}{remarks}
\Crefname{Remark}{Remark}{Remarks}
\newtheorem{assumption}{\textbf{A}\hspace{-3pt}}
\crefformat{assumption}{{\textbf{A}}#2#1#3}

\newtheorem{assumptionD}{\textbf{D}\hspace{-3pt}}
\crefformat{assumptionD}{{\textbf{D}}#2#1#3}
\newtheorem*{assumptionD'}{\textbf{D'}\hspace{-3pt}}

\renewcommand{\leq}{\leqslant}
\renewcommand{\geq}{\geqslant}

\def\scrA{\mathscr{A}}
\def\scrB{\mathscr{B}}
\def\scrC{\mathscr{C}}

\def\varphibf{\pmb{\varphi}}

\newcommandx\sequenceg[3][2=,3=]
{\ifthenelse{\equal{#3}{}}{\ensuremath{( #1_{#2})}}{\ensuremath{( #1_{#2})_{ #2 \geq #3}}}}

\newcommand{\half}{{\nicefrac{1}{2}}}

\newcommand{\ie}{\textit{i.e.}}

\newcommand{\rme}{\mathrm{e}}

\newcommand{\rmV}{\mathrm{V}}

\newcommand{\rmSpan}{\mathrm{Span}}

\newcommand{\rmCov}{\mathrm{Cov}}

\newcommand{\Abf}{\mathbf{A}}

\newcommand{\Bbf}{\mathbf{B}}

\newcommand{\bfX}{\mathbf{X}}
\newcommand{\bfV}{\mathbf{V}}

\def\msk{\mathsf{K}}

\def\msb{\mathsf{B}}

\def\msec{\mathsf{E}_{\mathrm{c}}}

\def\msw{\mathsf{W}}

\def\msy{\mathsf{Y}}

\def\mcbb{\mathcal{B}}  
\newcommand{\mcb}[1]{\mathcal{B}(#1)}

\def\mcy{\mathcal{Y}}

\def\mcw{\mathcal{W}}

\def\mcf{\mathcal{F}}

\def\rset{\mathbb{R}}

\def\nset{\mathbb{N}}
\def\nsets{\mathbb{N}^*}

\def\rmm{\mathrm{m}}

\def\rmd{\mathrm{d}}

\def\rml{\mathrm{L}}

\def\rme{\mathrm{e}}

\def\rmC{\mathrm{C}}

\newcommand{\abs}[1]{\left\vert #1 \right\vert}
\newcommand{\absLigne}[1]{\vert #1 \vert}
\newcommand{\tvnorm}[1]{\| #1 \|_{\mathrm{TV}}}
\newcommand{\tvnormEq}[1]{\left \| #1 \right \|_{\mathrm{TV}}}
\newcommandx{\Vnorm}[2][1=V]{\| #2 \|_{#1}}
\newcommandx{\normpi}[2][2=2]{\left\Vert  #1 \right\Vert_{#2}}
\newcommandx{\normH}[2][2=2]{\left\Vert  #1 \right\Vert}
\newcommandx{\normHLigne}[2][2=2]{\Vert  #1 \Vert}
\newcommandx{\normHLine}[2][2=2]{\Vert  #1 \Vert}
\newcommandx{\normmu}[2][2=2]{\left\Vert  #1 \right\Vert_{#2}}
\newcommandx{\normopmu}[2][2=2]{\left\Vert  #1 \right\Vert_{#2}}
\newcommandx{\normopH}[2][2=2]{\left\Vert  #1 \right\Vert_{\mathrm{op}}}
\newcommand{\normop}[1]{\left\Vert  #1 \right\Vert_{\mathrm{op}}}
\newcommand{\ps}[2]{\left\langle#1,#2 \right\rangle}

\newcommandx{\normpiLine}[2][2=2]{\Vert  #1 \Vert_{#2}}
\newcommandx{\normmuLine}[2][2=2]{\Vert  #1 \Vert_{#2}}
\newcommandx{\normopmuLine}[2][2=2]{\Vert  #1 \Vert_{#2}}
\newcommandx{\normopHLine}[2][2=2]{\Vert  #1 \Vert_{\mathrm{op}}}
\newcommandx{\normopLine}[2][2=2]{\Vert  #1 \Vert_{\mathrm{op}}}
\newcommandx{\normopLigne}[2][2=2]{\Vert  #1 \Vert_{\mathrm{op}}}

\newcommandx{\VnormEq}[2][1=V]{\left\| #2 \right\|_{#1}}
\newcommandx{\norm}[2][1=]{\ifthenelse{\equal{#1}{}}{\left\Vert #2 \right\Vert}{\left\Vert #2 \right\Vert^{#1}}}
\newcommandx{\normLigne}[2][1=]{\ifthenelse{\equal{#1}{}}{\Vert #2 \Vert}{\Vert #2\Vert^{#1}}}

\newcommand{\parenthese}[1]{\left(#1 \right)}

\newcommand{\parentheseDeux}[1]{\left[ #1 \right]}

\newcommand{\defEns}[1]{\left\lbrace #1 \right\rbrace }

\newcommandx\probaMarkovTilde[2][2=]
{\ifthenelse{\equal{#2}{}}{{\widetilde{\mathbb{P}}_{#1}}}{\widetilde{\mathbb{P}}_{#1}\left[ #2\right]}}

\newcommand{\PE}{\mathbb{E}}

\newcommand{\expe}[1]{\PE \left[ #1 \right]}

\newcommandx{\expeE}[2][2=]{\mathbb{E}^{#2}\left[ #1 \right]}

\newcommandx{\expeI}[2][1=]{\PE_{#1} \left[ #2 \right]}

\newcommand{\plusinfty}{+\infty}
\def\ie{\textit{i.e.}}

\def\eqsp{\;}
\renewcommand{\iint}[2]{\{ #1,\ldots,#2\}}
\newcommand{\coint}[1]{\left[#1\right)}
\newcommand{\ocint}[1]{\left(#1\right]}
\newcommand{\ooint}[1]{\left(#1\right)}
\newcommand{\ccint}[1]{\left[#1\right]}

\newcommand{\ocintLigne}[1]{(#1]}

\newcommand{\cballdim}[3]{\overline{\operatorname{B}}_{#1}(#2,#3)}

\newcommandx\sequence[3][2=,3=]
{\ifthenelse{\equal{#3}{}}{\ensuremath{\{ #1_{#2}\}}}{\ensuremath{\{ #1_{#2}, \eqsp #2 \in #3 \}}}}
\newcommandx\sequenceD[3][2=,3=]
{\ifthenelse{\equal{#3}{}}{\ensuremath{\{ #1_{#2}\}}}{\ensuremath{( #1)_{ #2 \in #3} }}}

\newcommandx{\sequencen}[2][2=n\in\N]{\ensuremath{\{ #1_n, \eqsp #2 \}}}
\newcommandx\sequenceDouble[4][3=,4=]
{\ifthenelse{\equal{#3}{}}{\ensuremath{\{ (#1_{#3},#2_{#3}) \}}}{\ensuremath{\{  (#1_{#3},#2_{#3}), \eqsp #3 \in #4 \}}}}
\newcommandx{\sequencenDouble}[3][3=n\in\N]{\ensuremath{\{ (#1_{n},#2_{n}), \eqsp #3 \}}}

\def\iid{i.i.d.}

\def\eg{e.g.}
\def\Id{\operatorname{Id}}
\def\Idd{\operatorname{I}_d}

\def\Zd{\bfZero_{d}}

\def\generator{\mathscr{L}}

\def\transpose{\operatorname{T}}

\newcommand{\1}{\mathbbm{1}}

\newcommandx{\CPE}[4][1=,4=]{{\mathbb E}^{#4}_{#1}\left[#2 \, \middle | #3 \right]} 
\newcommandx{\CPELigne}[4][1=,4=]{{\mathbb E}^{#4}_{#1}[#2\, | #3 ]} 
\newcommandx{\CPVar}[3][1=]{\mathrm{Var}^{#3}_{#1}\left\{ #2 \right\}}
\newcommand{\CPP}[3][]
{\ifthenelse{\equal{#1}{}}{{\mathbb P}\left(\left. #2 \, \right| #3 \right)}{{\mathbb P}_{#1}\left(\left. #2 \, \right | #3 \right)}}

\def\distance{\ell}
\newcommandx{\wasserstein}[3][1=\distance,3=]{\mathbf{W}_{#1}^{#3}\left(#2\right)}
\newcommandx{\wassersteinLigne}[3][1=\distance,3=]{\mathbf{W}_{#1}^{#3}(#2)}
\newcommandx{\wassersteinD}[1][1=\distance]{\mathbf{W}_{#1}}
\newcommandx{\wassersteinDLigne}[1][1=\distance]{\mathbf{W}_{#1}}

\def\bgamma{\bar{\gamma}}
\def\bsigma{\bar{\sigma}}

\newcommand\ceil[1]{\lceil #1 \rceil}
\newcommand\floor[1]{\lfloor #1 \rfloor}

\newcommandx{\gperthmc}[2][1=,2=]{\ifthenelse{\equal{#1}{}}{\Xi}{\ifthenelse{\equal{#2}{}}{\Xi_{h,#1}}{\Xi_{#2,#1}}}}
\newcommandx{\Phiverlet}[2][1=,2=]{\ifthenelse{\equal{#1}{}}{\Phi}{\Phi_{#1}^{\circ (#2)}}}
\newcommandx{\gpertub}[2][1=,2=]{\ifthenelse{\equal{#1}{}}{g}{g_{#1}^{#2}}}
\newcommandx{\Phiverletq}[2][1=,2=]{\ifthenelse{\equal{#1}{}}{\widetilde{\Phi}}{\widetilde{\Phi}_{#1}^{\circ (#2)}}}
\newcommandx{\Phiverletqi}[2][1=,2=]{\ifthenelse{\equal{#1}{}}{\bar{\Psi}}{\bar{\Psi}_{#1}^{\circ (#2)}}}
\newcommandx{\Pkerhmc}[2][1=,2=]{\ifthenelse{\equal{#1}{}}{\mathrm{P}}{\mathrm{P}_{#1, #2}}}
\newcommandx{\Qkerhmc}[2][1=,2=]{\ifthenelse{\equal{#1}{}}{\mathrm{Q}}{\mathrm{Q}_{#1, #2}}}
\newcommandx{\tPkerhmc}[2][1=,2=]{\ifthenelse{\equal{#1}{}}{\tilde{\mathrm{P}}}{\tilde{\mathrm{P}}_{#1, #2}}}
\newcommandx{\PkerhmcD}[2][1=,2=]{\ifthenelse{\equal{#1}{}}{\mathrm{K}}{\mathrm{K}_{#1, #2}}}

\def\tZ{\tilde{Z}}

\def\varphibf{\boldsymbol{\varphi}}

\newcommandx{\functionspace}[2][1=+]{\mathbb{F}_{#1}(#2)}

\newcommandx\EntU[1][1=]{\ifthenelse{\equal{#1}{}}{\mathrm{Ent}}{\mathrm{Ent}_{#1}}}

\def\Id{\operatorname{Id}}

\def\Gammabf{\mathbf{\Gamma}}

\newcommand{\argmin}{\operatorname*{arg\,min}}

\newcommandx{\VarDeux}[3][3=]{\operatorname{Var}^{#3}_{#1}\left\{#2 \right\}}

\newcommandx{\osc}[2][1=]{\mathrm{osc}_{#1}(#2)}

\def\Sigmabf{\boldsymbol{\Sigma}}
\newcommandx{\mgamma}[1][1=\gamma]{\mathbf{m}_{#1}}
\newcommandx{\Sgamma}[1][1=\gamma]{\Sigmabf_{#1}}
\newcommand{\Sgamman}[1]{\Sigmabf_{#1, \gamma}}

\def\PP{\mathbb{P}}

\def\LyaW{\mathpzc{W}}
\def\LyaV{\mathpzc{V}}
\def\LyapV{\mathpzc{V}}
\def\Fpzc{\mathpzc{F}}
\newcommand{\Lya}[1][\gamma]{\LyaW_{#1}}

\def\LyaZmain{\LyaV}
\newcommand{\Lyaexp}[1][\gamma]{\overline{\LyaW}_{#1,\varpi}}
\newcommand{\LyaexpU}[1][\gamma]{\overline{\LyaW}_{#1,\varpi_U}}
\newcommand{\LyaZmainexp}[1]{\overline{\LyaV}_{#1}}

\def\cWb{\underline{c}_{\LyaW}}
\def\cWu{\overline{c}_{\LyaW}}

\newcommand{\Lip}{\mathrm{Lip}}
\newcommand\tx{\tilde{x}}

\newcommandx\carreduchampadjoint[1][1=]{\Gamma_{#1}^{\star}}
\newcommandx\carreduchampadjointbf[1][1=]{\Gammabf_{#1}^{\star}}

\def\scrC{\mathscr{C}}

\def\rsetdd{\rset^{2d}}

\newcommandx{\normCarreSGamma}[2][1=]{\norm{#2}^2_{\Sgamman{#1}^{-1/2}}}

\def\0bf{\boldsymbol{0}}

\def\tz{\tilde{z}}

\def\Gminor{\mathrm{G}}
\newcommand{\Gminorc}[1]{\mathrm{G}^{(#1)}}

\def\cminor{\mathrm{c}}

\newcommand{\cminorc}[1]{\cminor^{(#1)}}

\def\Phiminor{\Phi}
\def\Gammaminor{\Gamma}
\def\GammaminorF{\Upsilon}
\newcommandx\Phiminorc[2][2=]
{\ifthenelse{\equal{#2}{}}{\ensuremath{\Phiminor_{#1}}}{\ensuremath{\Phiminor_{#1}^{(#2)}}}}
\newcommandx\Gammaminorc[2][2=]
{\ifthenelse{\equal{#2}{}}{\ensuremath{\Gammaminor_{#1}}}{\ensuremath{\Gammaminor_{#1}^{(#2)}}}}

\newcommandx\GammaminorcF[2][2=]
{\ifthenelse{\equal{#2}{}}{\ensuremath{\GammaminorF_{#1}}}{\ensuremath{\GammaminorF_{#1}^{(#2)}}}}

\def\tPhiminor{\tilde{\Phi}}
\newcommandx\tPhiminorc[2][2=]
{\ifthenelse{\equal{#2}{}}{\ensuremath{\tPhiminor_{#1}}}{\ensuremath{\tPhiminor_{#1}^{(#2)}}}}
\def\tGammaminor{\widetilde{\Gamma}}
\def\tGammaminorF{\widetilde{\GammaminorF}}
\newcommandx\tGammaminorc[2][2=]
{\ifthenelse{\equal{#2}{}}{\ensuremath{\tGammaminor_{#1}}}{\ensuremath{\tGammaminor_{#1}^{(#2)}}}}
\newcommandx\tGammaminorcF[2][2=]
{\ifthenelse{\equal{#2}{}}{\ensuremath{\tGammaminorF_{#1}}}{\ensuremath{\tGammaminorF_{#1}^{(#2)}}}}

\def\bPhiminor{\bar{\Phi}}
\newcommandx\bPhiminorc[2][2=]
{\ifthenelse{\equal{#2}{}}{\ensuremath{\bPhiminor_{#1}}}{\ensuremath{\bPhiminor_{#1}^{(#2)}}}}

\def\bGammaminor{\overline{\Gamma}}
\newcommandx\bGammaminorc[2][2=]
{\ifthenelse{\equal{#2}{}}{\ensuremath{\bGammaminor_{#1}}}{\ensuremath{\bGammaminor_{#1}^{(#2)}}}}

\def\bGammaminorF{\overline{\GammaminorF}}
\newcommandx\bGammaminorcF[2][2=]
{\ifthenelse{\equal{#2}{}}{\ensuremath{\bGammaminorF_{#1}}}{\ensuremath{\bGammaminorF_{#1}^{(#2)}}}}

\newcommandx\bPhiminorinvc[2][2=]
{\ifthenelse{\equal{#2}{}}{\ensuremath{\bPhiminor_{#1}}}{\ensuremath{\bPhiminor_{#1}^{(\leftarrow,#2)}}}}
\newcommandx\bGammaminorinvc[2][2=]
{\ifthenelse{\equal{#2}{}}{\ensuremath{\bGammaminor_{#1}}}{\ensuremath{\bGammaminor_{#1}^{(\leftarrow,#2)}}}}

\newcommandx\bGammaminorinvcF[2][2=]
{\ifthenelse{\equal{#2}{}}{\ensuremath{\bGammaminorF_{#1}}}{\ensuremath{\bGammaminorF_{#1}^{(\leftarrow,#2)}}}}

\def\LipXi{\mathrm{L}}
\newcommandx\LipXic[2][2=]
{\ifthenelse{\equal{#2}{}}{\ensuremath{\LipXi_{\Xi,#1}}}{\ensuremath{\LipXi_{\Xi,#1}^{(#2)}}}}

\def\Psiminor{\Psi}
\newcommandx\Psiminorc[2][2=]
{\ifthenelse{\equal{#2}{}}{\ensuremath{\Psiminor_{#1}}}{\ensuremath{\Psiminor_{#1}^{(#2)}}}}

\def\Ximinor{\Xi}
\newcommandx\Ximinorc[2][2=]
{\ifthenelse{\equal{#2}{}}{\ensuremath{\Ximinor_{#1}}}{\ensuremath{\Ximinor_{#1}^{(#2)}}}}

\def\Mminor{\mathrm{M}}
\newcommandx\Mminorc[2][2=]
{\ifthenelse{\equal{#2}{}}{\ensuremath{\Mminor_{#1}}}{\ensuremath{\Mminor_{#1}^{#2}}}}

\def\Bminor{\Theta}

\def\alphaminor{\alpha}
\def\betaminor{\beta}
\def\alphaminorbf{\boldsymbol{\alpha}}
\def\betaminorbf{\boldsymbol{\beta}}

\def\XminorTer{X}
\def\VminorTer{V}

\def\Ximinor{\Xi}

\def\Onebf{\boldsymbol{1}}

\def\Unbf{\Onebf}

\def\Zerobf{\boldsymbol{0}}
\def\bfZero{\boldsymbol{0}}

\def\Nminor{N}

\def\LipMGammaminor{\mathrm{M}_{\Gamma,\gamma}}
\def\LipGammaminor{\mathrm{L}_{\Gamma,\gamma}}

\newcommand{\tZc}[1]{\widetilde{Z}^{(#1)}}

\newcommand{\tZbfc}[1]{\widetilde{\mathbf{Z}}^{(#1)}}
\newcommand{\tZbfcs}{\widetilde{\mathbf{Z}}}
\newcommand{\Zbfc}[1]{\mathbf{Z}^{(#1)}}
\newcommand{\bfWc}[1]{\mathbf{W}^{(#1)}}

\def\vartZ{\mathbf{S}}
\newcommand{\vartZc}[1]{\vartZ^{(#1)}}

\def\bminorg{g_{\gamma}}

\def\LipT{\mathtt{L}}

\def\Rkerminor{R}
\newcommand{\Rkerminorc}[1]{\Rkerminor_{#1}}

\def\Qkerminor{Q}
\newcommand{\Qkerminorc}[1]{\Qkerminor_{#1}}

\def\gminor{g}
\newcommand{\gminorc}{\tau_{\gamma}}
\def\gminorbf{\mathbf{g}}
\newcommand\gminorbfc[1]{\gminorbf^{(#1)}}

\def\Pminor{\mathbf{P}}
\newcommand\Pminorc[1]{\Pminor^{(#1)}}

\def\bart{\bar{t}}

\def\tzbf{\mathbf{\tz}}

\def\JacD{\mathrm{J}}
\def\JacM{\mathrm{Jac}}

\def\tbound{\bar{t}_0}

\newcommand{\coeffContinueMatrix}[2]{\mathbf{\Sigma}_{#1}^{(#2)}}
\newcommand{\continueMatrix}[1]{\mathbf{\Sigma}^{(#1)}}

\def\btZ{\bar{t}_0}
\def\brho{\bar{\rho}}
\newcommand{\bfw}{\mathbf{w}}
\newcommand{\wdensity}[1]{\mu_W^{#1}}
\newcommand{\rsetddd}{\rset^{3d}}

\def\Dbf{\mathbf{D}}

\def\Ltt{\mathtt{L}}
\newcommand{\txts}[1]{\textstyle #1}
\def\tLambda{\tilde{\Lambda}}

\def\muw{\mu_{\mathrm{w}}}
\def\muwOne{\mu_{\mathrm{w_1}}}
\def\muwTwo{\mu_{\mathrm{w_2}}}
\def\rmX{\mathrm{X}}
\def\rmV{\mathrm{V}}

\def\XF{X}

\def\BM{\mathrm{B}}

\def\rmy{\mathrm{y}}

\def\parametreGradU{\alpha_U}
\def\constgGradU{\mathrm{C}_U}
\def\LipGradU{L}

\newcommand{\constasssuperexpone}{\zeta_U}
\newcommand{\fonctionControle}{\Fpzc}
\newcommand{\fonctionControleBis}{\widetilde{\Fpzc}}

\def\bvartheta{\bar{\vartheta}}
\newcommand{\LSfonction}[1][\gamma]{\phi_{#1}}
\newcommand{\normLip}[1]{\norm{#1}_\Lip}

\def\tsigma{\widetilde{\sigma}}
\newcommand{\cLya}{c_{\LyaW}}
\newcommand{\cbgamma}{\bgamma_{\LyaW}}
\newcommand{\cPhi}{\mathfrak{C}_\phi}
\newcommand{\LSconstant}{C_{\mathrm{S}}}
\newcommand{\LLS}{\tilde{\mathtt{L}}}

\newcommand{\deltau}{{\delta_{U}}}

\def\supplementname{supplement}

\def\supplementname{appendix}

\title{Uniform minorization condition and convergence bounds for discretizations of kinetic Langevin dynamics}

\author[1]{Alain Durmus}
\author[2]{Aurélien Enfroy}
\author[1]{\'Eric Moulines}
\author[3]{Gabriel Stoltz}
\affil[1]{\small{CMAP, CNRS, Ecole Polytechnique, Institut Polytechnique de Paris, 91120 Palaiseau, France}}
\affil[2]{\small{Laboratoire de mathématiques d’Orsay, Université Paris-Saclay, Orsay, France}}
\affil[3]{\small{CERMICS, Ecole des Ponts, Marne-la-Vallée, France \& MATHERIALS team-project, Inria Paris, France}}

\begin{document}
\footnotetext[1]{Email: alain.durmus@polytechnique.edu, eric.moulines@polytechnique.edu}
\footnotetext[2]{Email: aurelien.enfroy@ens-paris-saclay.fr}
\footnotetext[3]{Email: gabriel.stoltz@enpc.fr}

\maketitle

\begin{abstract}
  We study the convergence in total variation and
  $V$-norm of discretization schemes of the underdamped Langevin
  dynamics. Such algorithms are very popular and commonly used in
  molecular dynamics and computational statistics to  approximatively sample from a target
  distribution of interest. We show first that, for a
  very large class of schemes, a minorization condition uniform in the
  stepsize holds. This class encompasses popular methods such as
  the Euler-Maruyama scheme and the schemes based on splitting
  strategies. Second, we provide mild conditions ensuring that the
  class of schemes that we consider satisfies a geometric Foster--Lyapunov drift
  condition, again uniform in the stepsize. This allows us to derive geometric convergence bounds, with a
  convergence rate scaling linearly with the stepsize. This kind of
  result is of prime interest to obtain estimates on norms of
  solutions to Poisson equations associated with a given numerical method. 
\end{abstract}



\section{Introduction}

Langevin dynamics are nowadays one of the
default dynamics to sample configurations of molecular systems in
computational statistical physics; see for
instance~\cite{FrenkelSmit,Tuckerman,AT17} for reference textbooks on
molecular dynamics, as well as the more mathematically oriented
works~\cite{PavliotisBook,LM15,LS16}. 
They  are also gaining
increasing popularity in 
Bayesian statistics and machine learning
\cite{robert:2007:choice,barber_2012} to obtain approximate samples
from the a posteriori distribution of a statistical model \cite{cheng:et:al:2018,dalalyan2020sampling}. 
In this paper, we are interested in the Langevin dynamics, sometimes
coined underdamped or kinetic, which describes the evolution of the position $(X_t)_{t \geq 0}$ and the velocity $(V_t)_{t \geq 0}$ of a
system by the $2d$-dimensional stochastic
differential equation (SDE):
\begin{equation}
  \label{eq:general_SDE}
  \begin{aligned}
    \rmd \rmX_t &= \rmV_t \, \rmd t, \\
    \rmd \rmV_t &= \left[b(\rmX_t) - \kappa \rmV_t \right] \rmd t  + \sigma\,  \rmd \BM_t \eqsp.
  \end{aligned} 
\end{equation}
Here, $\kappa,\sigma>0$ are  some friction and diffusion coefficients respectively, and $(\BM_t)_{t \geq 0}$ is a standard $d$-dimensional Brownian motion defined on the filtered probability space~$(\Omega,\mcf,\PP,(\mcf_t)_{t \geq 0})$ satisfying the usual conditions. Note that, for notational simplicity, the mass matrix of the system is set to be the identity matrix and the friction coefficient is a scalar, independent of the  position. It would nonetheless be possible to consider symmetric definite positive mass matrices, and position dependent friction matrices. A typical choice for the vector field $b \in \rmC^1(\rset^d,\rset^d)$ is $b = -\nabla U$ for some potential energy~$U$, in which case the unique invariant probability measure of~\eqref{eq:general_SDE} is the Boltzmann--Gibbs probability measure, whose density with respect to the Lebesgue measure is proportional to~$(x,v) \mapsto \exp(-\kappa [2U(x)+\norm{v}^2] / \sigma^2)$. However, we are also interested in situations where the drift does not arise from a gradient, as in nonequilibrium molecular dynamics simulations~\cite{CKS05,evans-moriss} (see for instance~\cite[Section~5]{LS16} for a mathematical introduction to this field).

There are various techniques to prove the convergence of the
continuous dynamics~\eqref{eq:general_SDE}, for instance
hypocoercivity~\cite{Talay02,Villani09}, Lyapunov estimates~\cite{Wu01,mattingly:stuart:higham:2002,rey-bellet} and
coupling methods~\cite{EGZ19}. Moreover,
although the dynamics is degenerate, \ie~the covariance matrix
associated with \eqref{eq:general_SDE} is not invertible, it can be
shown using a combination
of controllability arguments and
hypoellipticity~\cite{rey-bellet}, that, for any $t_0 >0$ and initial condition
$(X_0,V_0) = (x_0,v_0)\in\rset^{2d}$, the random variable~$(X_{t_0},V_{t_0})$ has a
distribution with a positive density $(x_{t_0},v_{t_0}) \mapsto p_{t_0}((x_0,v_0),(x_{t_0},v_{t_0}))$ with respect to the Lebesgue
measure on $\rset^{2d}$, implying
that the Markov semigroup associated with \eqref{eq:general_SDE} is
irreducible. Finally, quantitative bounds on $p_{t_0}$ for any $t_0 >0$ can be established by various techniques, such as
Malliavin calculus~\cite{BC13}, representation formulas for pinned
diffusions together with comparison principles~\cite{QZ02,QZ04}, or
through Gaussians bounds obtained with the so-called parametrix
method~\cite{KMM10,LRR20}.

Obtaining quantitative convergence rates at the discrete level for discretization scheme associated with~\eqref{eq:general_SDE}  is more difficult.
In particular, it is of prime concern to establish convergence rates $\rho_{\gamma}$ for discretization schemes with timestep $\gamma >0$ which are similar to their continuous counterpart, \ie~such that $\log(\rho_{\gamma})$ scales linearly in the timestep. More precisely, if we denote by $R_{\gamma}$ the Markov kernel associated with a given discretization scheme, it is sensible to expect that, under appropriate conditions and for any $\gamma >0$ small enough, this kernel admits an invariant distribution~$\pi_{\gamma}$ and $\mathbf{d}(\mu_0 R_{\gamma}^{k},\pi_{\gamma}) \leq C \overline{\rho}^{\gamma k}$ for an initial distribution $\mu_0$, where~$C \geq0$ and $\overline{\rho}\in \coint{0,1}$, while~$\mathbf{d}$ is some distance on the set of probability measures on $\rset^{2d}$. 
Currently, one of the  main options to obtain such a convergence result is to prove Lyapunov estimates and minorization conditions which are uniform in the timestep, \ie~that such conditions holds for $R_{\gamma}^{\ceil{t_0/\gamma}}$ for some $t_0>0$ for constants which do not depend on~$\gamma$. This strategy was used for overdamped Langevin dynamics in~\cite{BH13} and~\cite{dBD19}. While Lyapunov conditions are based on direct algebraic computations, and may require to consider implicit schemes~\cite{mattingly:stuart:higham:2002,KopecLangevin}, showing a minorization condition uniform in the timestep is the main bootleneck of this approach. For non-degenerate stochastic dynamics, this type of results can be established relying on Malliavin calculus~\cite{BT96}, 
but the resulting proof is rather involved. 
 It may also be possible to write a direct proof as in~\cite{BH13} from the Girsanov theorem, although this is however possible only for dynamics with additive noise. Finally, the coupling approach developed in \cite{durmus:moulines:2019,dBD19,eberle2018quantitative} cannot be applied to discretizations of the degenerate SDE \eqref{eq:general_SDE}. 

 In this paper, we consider another approach to obtaining minorization conditions uniform in the timestep: the idea is to consider the numerical scheme over small physical times as a perturbation of a given Gaussian process.
This approach was first advocated by one of the authors to easily present the main results and rationale from \cite{durmus:moulines:2019,dBD19} for discretization schemes of the overdamped Langevin dynamics; see \Cref{sec:idea_ULA}. 
 It turns out that this approach can be extended to various discretization schemes for the underdamped Langevin dynamics, in particular the splittings schemes proposed in~\cite{BO10,LM12}, which are becoming increasingly popular in molecular dynamics (see~\cite[Section~12.2]{AT17}). This contribution allows to amend and correct the proof of~\cite[Lemma~2.8]{leimkuhler:matthews:stoltz:2016} and extend  this result to unbounded spaces.

\paragraph{Outline of the work.}
The present document is organized as follows. We present the main results we obtain in Section~\ref{sec:main_result}; see in particular \Cref{theo:minor_main_gen} for the minorization condition, and \Cref{theo:drift_exp} for the Lyapunov condition. Both results are stated so that the dependence on the timestep is explicit. The general structure of the numerical schemes we consider is motivated in Section~\ref{sec:exampl-numer-schem}, where we present various algorithms to discretize Langevin dynamics. The remaining sections are devoted to the proofs of these results. The proof of the minorization condition is written in \Cref{sec:proof_minorization}, with, for pedagogical purposes, a sketch of the proof for nondegenerate dynamics in \Cref{sec:idea_ULA}; while the proof of the Lyapunov condition can be read in \Cref{sec:proof-crefth_drift}. For completeness, some proofs and derivations are deferred to the \supplementname.

\paragraph{Notation.} In order to present more concisely our results, we use the following notation throughout this work. For $m,n,p,q\in \mathbb{N}^*$, the Kronecker product $\mathbf{A} \otimes \mathbf{B}$ of a $m\times n$ matrix $\mathbf{A}=(a_{i,j})_{(i,j) \in \{1,\ldots,m\}\times\{1,\ldots,n\}}$ and a $p \times q$ matrix $\mathbf{B}=(b_{i,j})_{(i,j) \in \{1,\ldots,p\}\times\{1,\ldots,q\}}$ is the $pm \times qn$ dimensional matrix with entries
\begin{equation}
\forall (i,j) \in \{1,\ldots,pm\}\times\{1,\ldots,qn\}\eqsp,
\qquad
(\mathbf{A}\otimes \mathbf{B})_{i,j}=a_{\lceil i/p \rceil, \lceil j/q \rceil }b_{i-\lfloor  (i-1)/p \rfloor p, j-\lfloor  (j-1)/q \rfloor q } \eqsp. 
\end{equation}
Equivalently, and more explicitly,
\begin{equation}
  \label{eq:1}
  \Abf \otimes \Bbf = \begin{pmatrix} a_{11} \Bbf & \cdots & a_{1n}\Bbf \\ \vdots & \ddots & \vdots \\ a_{m1} \Bbf & \cdots & a_{mn} \Bbf \end{pmatrix} \eqsp.
\end{equation}
For two symmetric matrices $\mathbf{A}$ and $\mathbf{B}$, we say that $\mathbf{A}\succeq  \mathbf{B}$ if $\mathbf{A}-\mathbf{B}$ is positive semi-definite.
We denote by $\Zerobf_d$ and $\Unbf_d$ the $d$-dimensional vectors with all components equal to $0$ and $1$ respectively.

The set~$\mathcal{B}(\rset^d)$ denotes the Borel $\sigma$-field of $\rset^d$. The Euclidean scalar product of vectors~$x$ and $y$ in~$\rset^d$ is denoted by~$\ps{x}{y} = x^{\transpose} y$, the Euclidean norm of~$x$ being~$\norm{x}$. For any $n\in \mathbb{N}^*$ and for any matrix $\mathbf{A}$ of size $n\times n$, the notation~$\normop{\mathbf{A}}$ stands for the induced norm defined by $\normop{\mathbf{A}}=\sup \{\norm{\mathbf{A}x}\,:\, x\in \rset^n$ with $\norm{x}=1\}$. For $k,n,m\in \nsets$, the set of $k$-times continuously differentiable functions $f : \rset^n \to \rset^m$ is denoted by~$\rmC^k(\rset^n,\rset^m)$. For $f \in \rmC^1(\rset^{d},\rset)$, $\nabla f$ is the gradient of $f$, and for $f \in \rmC^2(\rset^{d},\rset)$ we denote by $\Delta f$ the Laplacian of $f$. When $f \in \rmC^1(\rset^{2d},\rset)$, $\nabla_x f$ is the gradient of $f$ restricted to the first~$d$ components and $\nabla_v f$ the gradient of $f$ restricted to the last~$d$ components. When $f \in \rmC^2(\rset^{2d},\rset)$, $\Delta_x f$ is the Laplacian of $f$ restricted to the first~$d$ components and $\Delta_v f$ the Laplacian of $f$  restricted to the last~$d$ components.
The closed ball centered at~$\tx_0  \in  \rset^d$ (with~$d \geq 1$) with radius~$M \geq 0$ is $\cballdim{d}{x_0}{M} = \left\{ x \in \rset^{d} \,: \, \norm{x-x_0} \leq M \right\}$. Finally, the $d$-dimensional standard normal distribution is denoted $\varphibf$  and by abuse of notation  its density with respect to the Lebesgue measure is also denoted by  $z \mapsto \varphibf(z)$.
 For some measurable functions $\LyaV: \rset^{n} \to \coint{1,\infty}$ and $g: \mathbb{R}^{n} \to \rset$, we define
$\Vnorm[\LyaV]{g}= \sup_{x \in \mathbb{R}^{n}} \{| g(x) | / \LyaV(x) \} < \infty$.
The $\LyaV$-norm of a signed measure $\xi$ on $(\mathbb{R}^{n},\mcbb(\rset^{n}))$ is defined as
$\Vnorm[\LyaV]{\xi} = \int_{\mathbb{R}^{n}} \LyaV(x) \rmd \abs{\xi}(x)$, where $\abs{\xi}$ is the absolute
value of $\xi$. In the case $\LyaV \equiv 1$, the $\LyaV$-norm is the total variation norm and it is denoted by $\tvnorm{\cdot}$. Equivalently (see \cite[Theorem D.3.2]{douc:moulines:priouret:2018} for details), $\Vnorm[\LyaV]{\xi}$ can be defined as $\Vnorm[\LyaV]{\xi} = \sup\{ \xi(g) \, :\, \Vnorm[\LyaV]{g}\leq 1\}$. 

\section{Setting and main results}
\label{sec:main_result}

We first discuss in Section~\ref{sec:structural} the general structure of the discretization schemes we consider (the relevance of the structural assumptions we make is illustrated later on by various examples in Section~\ref{sec:exampl-numer-schem}). We then state the main results of this work in Section~\ref{sec:main_results}, namely minoration and drift conditions uniform in the discretization timestep, from which we immediately deduce a geometric convergence with a rate uniform in the timestep as well. 

\subsection{Structural assumptions on the numerical schemes}
\label{sec:structural}

Discretization schemes for~\eqref{eq:general_SDE} are obtained in practice by introducing a positive timestep~$\gamma >0$.  They correspond to a Markov chain~$\{ (\XminorTer_k,\VminorTer_k) \}_{k\in \mathbb{N}}$, where $(\XminorTer_k,\VminorTer_k)$ approximates~$(\rmX_{k\gamma},\rmV_{k\gamma})$, the solution of~\eqref{eq:general_SDE} at time~$k\gamma$. More precisely, we consider the following general structure on the induction defining the numerical schemes: for $k \in\nset$,
\begin{equation}
  \label{eq:def_xminor_v_minor_gen}
 \begin{aligned}
   \XminorTer_{k+1} & = \XminorTer_k + \gamma\VminorTer_k + \gamma f_\gamma\left(\XminorTer_k,\gamma^{\delta}\VminorTer_k,\gamma^{\delta+\half} \sigma_{\gamma} Z_{k+1},W_{k+1}\right) + \gamma^{\delta+\half} \sigma_{\gamma} \Dbf_{\gamma} Z_{k+1} \eqsp,\\ 
    \VminorTer_{k+1} & = \gminorc\VminorTer_k + \gamma g_\gamma\left(\XminorTer_k,\gamma^{\delta}\VminorTer_k, \gamma^{\delta+\half} \sigma_{\gamma} Z_{k+1},W_{k+1}\right)+ \sqrt{\gamma}\sigma_{\gamma} Z_{k+1} \eqsp ,
  \end{aligned} 
\end{equation}
where $\delta >0$ is a positive parameter (equal to~1 in all the examples we consider), the family~$(W_{k+1})_{k \in\nset}$ is a sequence of independent and identically distributed (\iid)~random variables with common distribution~$\muw$ on a measurable space $(\msw,\mcw)$, independent of the family~$(Z_{k+1})_{k\in\nset}$ of \iid~$d$-dimensional standard Gaussian random variables. In many cases of interest (see \Cref{sec:exampl-numer-schem} below), $(W_{k+1})_{k \in\nset}$ is a family of \iid~standard Gaussian random variables.

The actual numerical schemes under consideration are encoded by the  measurable functions~$f_\gamma,g_\gamma: \rsetddd \times\msw \to \rset^d$, as well as by $\sigma_{\gamma},\tau_{\gamma} >0$ and $\Dbf_{\gamma}\in \rset^{d \times d}$. We illustrate the choice of the form for the recursion~\eqref{eq:def_xminor_v_minor_gen} by several discretization schemes for~\eqref{eq:general_SDE} in \Cref{sec:exampl-numer-schem}. One formally expects in the limit~$\gamma \to 0$ that
\[
\tau_\gamma = 1 - \kappa \gamma+ \mathrm{O}(\gamma^2), \qquad \sigma_\gamma \to \sigma,
\]
and for any $x,v \in \rset^d,w\in\msw$,
\begin{equation}
  \label{eq:beharibour_f_g}
g_\gamma(x,\gamma^\delta v,\gamma^{\delta+\half}z,w) \to b(x)\eqsp,
\qquad
f_\gamma(x,\gamma^\delta v,\gamma^{\delta+\half}z,w) \to 0\eqsp.
\end{equation}
These limits are in fact  equalities for simple numerical schemes such as the Euler--Maruyama method (see~\eqref{eq:EM_scheme} below). However, we  need to consider a general framework, and introduce additional noise variables~$(W_{k+1})_{k \in\nset}$ and drift functions such as~$f_\gamma$ in order to analyze more complicated numerical schemes,  as~\eqref{eq:second_order_splitting_complicated} below for instance. Let us also emphasize that for \eqref{eq:beharibour_f_g} to hold, the arguments in the functions~$f_\gamma,g_\gamma$ need to be scaled by powers of~$\gamma$. In addition, these rescaled versions of $f_{\gamma},g_{\gamma}$ are Lipschitz with constants independent of the timestep in all our examples in \Cref{sec:exampl-numer-schem}, which motivates Assumption~\Cref{ass:lip_minor_gen} below.

We consider the following assumptions on the coefficients and functions entering~\eqref{eq:def_xminor_v_minor_gen}. We always assume that $\gamma \in \ocintLigne{0,\bgamma}$, for some fixed $\bgamma >0$. Typically, $\bgamma$ represents a threshold which ensures that the scheme under consideration is stable. Here, for ease of presentation, we assume in \Cref{ass:gminorc} below that $\bgamma$ is even smaller than a specific constant.

The first assumption \Cref{ass:gminorc} expresses some form of consistency of the coefficients~$\sigma_\gamma,\gminorc$ in~\eqref{eq:def_xminor_v_minor_gen}, which are related to the coefficients~$\sigma,\kappa$ in~\eqref{eq:general_SDE}. In addition, we also impose some upper bound on $\bgamma$ to simplify the derivation of our main results but it could be easily relaxed. 

\begin{assumption}
  \label{ass:gminorc}
    \begin{enumerate}[wide, labelwidth=!, labelindent=0pt,label=\arabic*)]
    \item 
There exists $C_{\kappa}\geq 0$ such that  $\bar{\gamma}\leq (\kappa+2C_{\kappa}/\kappa)^{-1}$, and for any $\gamma\in \ocint{0,\bar{\gamma}}$, it holds $\abs{\gminorc-\rme^{-\kappa\gamma}}\leq C_{\kappa} \gamma^2$ and~$\tau_{\gamma} \in\ooint{0,1}$.
\item  There exist $\bar{\sigma},\mathscr{D} \in \mathbb{R}_+$ such that $\displaystyle \sup_{\gamma\in \ocint{0,\bar{\gamma}}}\sigma_\gamma\leq \bar{\sigma}$ and $\displaystyle \sup_{\gamma \in \ocint{0,\bgamma}}\normop{\Dbf_{\gamma}} \leq \mathscr{D} $. Finally, $\displaystyle \lim_{\gamma\downarrow 0}\sigma_\gamma=\sigma$.
\end{enumerate}
\end{assumption}

The conditions in~\Cref{ass:lip_minor_gen} express some form of Lipschitz stability with respect to scaled variables, and quantifies the fact that perturbations arising from~$f_\gamma,g_\gamma$ can be of order~$\gamma$ with respect to positions, while they are restricted to be of order~$\gamma^{1+\delta}$ with respect to momenta. 
\begin{assumption}
  \label{ass:lip_minor_gen}
  For all $w \in \msw$, the functions~$(x,v,z) \mapsto (f_\gamma(x,v,z,w), g_\gamma(x,v,z,w))$ are~$\rmC^1$. In addition, there exists $\LipT \geq 0$ such that, for any $\gamma\in \ocint{0,\bar{\gamma}}$, $w \in \msw$ and $(x,v,z),(x',v',z')\in \rset^{3d}$,
  \begin{align}
    \norm{f_\gamma(x,v,z,w)-f_\gamma(x',v',z',w)}&\leq \LipT\parentheseDeux{\norm{(x,v)-(x',v')}+\norm{z-z'}}  \eqsp, \\
    \norm{g_\gamma(x,v,z,w)-g_\gamma(x',v',z',w)}&\leq \LipT\parentheseDeux{\norm{(x,v)-(x',v')}+\norm{z-z'}} \eqsp.
  \end{align}
\end{assumption}
In fact, we could assume that~\Cref{ass:lip_minor_gen} only holds for $\muw$-almost all $w\in\msw$, but strenghten this condition for ease of presentation.

For concreteness of the discussion and results to follow, we illustrate our choice of the framework specified by \eqref{eq:def_xminor_v_minor_gen} by two simple examples. In order not to distract the reader and to not further postpone our main results, we show that most common discretization schemes of \eqref{eq:general_SDE} fit into the framework~\eqref{eq:def_xminor_v_minor_gen} in \Cref{sec:exampl-numer-schem}.

\begin{example}[Euler--Maruyama scheme]
  \label{exam:em_verlet_main}
 The Euler--Maruyama discretization of~\eqref{eq:general_SDE} reads
\begin{equation}
  \label{eq:EM_scheme}
  \XminorTer_{k+1} = \XminorTer_k+\gamma\VminorTer_k \eqsp, \qquad 
  \VminorTer_{k+1} = (1-\kappa \gamma) \VminorTer_k + \gamma b(\XminorTer_k) + \sqrt{\gamma} \sigma Z_k \eqsp. 
\end{equation}
This numerical scheme can be written in the form~\eqref{eq:def_xminor_v_minor_gen} upon taking
\begin{equation}
\tau_{\gamma} = 1-\kappa \gamma \eqsp, \qquad \sigma_{\gamma} = \sigma\eqsp, \qquad \Dbf_{\gamma} = 0\eqsp, \qquad f_\gamma(x,v,z,w) = 0\eqsp, \qquad g_{\gamma}(x,v,z,w) = b(x)\eqsp.
\end{equation}
The parameter~$\delta$ is irrelevant. Note that \Cref{ass:gminorc} and \Cref{ass:lip_minor_gen} hold under the condition that $b$ is Lipschitz.
\end{example}
\begin{example}[Simple Verlet scheme]
  \label{exam:verlet}
  The general recursion  also  includes the simplest Verlet-type scheme (see \Cref{sec:splitting-schemes}) which reads
\begin{equation}
  \XminorTer_{k+1} = \XminorTer_k+\gamma\VminorTer_k+\gamma^2 b(\XminorTer_k)\eqsp, \qquad \VminorTer_{k+1}  = \rme^{-\kappa \gamma} \VminorTer_k + \gamma \rme^{-\kappa \gamma} b(\XminorTer_k) + \sqrt{\frac{1-\rme^{-2\kappa \gamma}}{2\kappa}} \sigma Z_{k+1}\eqsp.
\end{equation}
 This corresponds
to~\eqref{eq:def_xminor_v_minor_gen} with
$\gminorc = \rme^{-\kappa \gamma}$,
$\sigma_{\gamma} = \sigma\sqrt{(1-\rme^{-2\kappa \gamma})/(2\kappa)}$, $\Dbf_\gamma = 0$ ($\delta$ is
irrelevant), and $f_\gamma(x,v,z,w) = \gamma b(x)$,
$ g_\gamma(x,v,z,w) = \rme^{-\kappa \gamma}b(x)$.
\end{example}

\subsection{Minorization, drift and convergence uniform in the timestep}
\label{sec:main_results}

To state our results, we introduce the Markov kernel $\Rkerminorc{\gamma}$ associated with~\eqref{eq:def_xminor_v_minor_gen}, \ie, for any~$(x,v) \in \rset^{2d}$ and~$\msb \in \mcb{\rset^{2d}}$,
\begin{equation}
  \label{eq:def_R_gamma}
  \Rkerminorc{\gamma}((x,v),\msb) = \int_{\rset^{d+m}} \1_{\msb}\left(\Gammaminorc{\gamma}(x,v,(\gamma\sigma_{\gamma}^2)^{\half}z,w)\right) \varphibf(z) \, \rmd z \, \mu_W( \rmd w) \eqsp,
\end{equation}
where $\varphibf(z)$ is the density of the $d$-dimensional standard normal distribution, and
\begin{equation}
  \label{minor_def_Gamma_b_gamma}
  \begin{aligned}
    \Gammaminor_\gamma(x,v,z,w) & = \Big( x+\gamma v+ \gamma f_\gamma\left(x, \gamma^{\delta} v,\gamma^{\delta} z,w\right)+\gamma^\delta \Dbf_{\gamma} z, \gminorc v+\gamma g_\gamma\left(x, \gamma^{\delta}v,\gamma^\delta z,w\right)+z \Big) \eqsp.
  \end{aligned}
\end{equation}

We present in the following our main results which allow us to
conclude to the uniform $\LyapV$-geometric ergodicity for
$ \Rkerminorc{\gamma}$ of the form
$\Vnorm[\LyapV]{\updelta_{(x,v)} \Rkerminorc{\gamma}^k - \pi_{\gamma}}
\leq C \LyapV(x,v) \rho^{\gamma k}$ for any $k \in\nset$,
$\gamma \in\ocint{0,\bgamma}$ small enough, and where $C\geq 0$ and
$\rho \in \coint{0,1}$ are independent of $\gamma$. The proof of this result standardly follows from quantitative minorization and Lyapunov conditions.


\paragraph{Minorization condition uniform in the timestep.}
\label{sec:minorization_unif_timestep}

Our first main result shows that $\Rkerminorc{\gamma}^{\ceil{t_0/\gamma}}$ satisfies a minorization condition with a constant which depends only on the physical time~$t_0>0$ (considered sufficiently small) and not on~$\gamma$. 

\begin{theorem}
  \label{theo:minor_main_gen}
  Assume that \Cref{ass:gminorc} and \Cref{ass:lip_minor_gen} hold. Then there exists $\bart_0 >0$ such that, for any $t_0 \in \ocint{0,\bart_0}$ and $M \geq 0$, there are $\varepsilon_{t_0,M}>0$ and $\bgamma_{t_0} >0$ for which, for any $\gamma \in \ocint{0,\bgamma_{t_0}}$ and $(x,v),(x',v') \in \cballdim{2d}{\Zerobf_{2d}}{M}$, 
  \begin{equation}
    \label{eq:theo:minor_main_gen}
 \tvnormEq{\Rkerminorc{\gamma}^{\floor{t_0/\gamma}+1}((x,v),\cdot)  - \Rkerminorc{\gamma}^{\floor{t_0/\gamma}+1}((x',v'),\cdot)} \leq 2\left( 1-\varepsilon_{t_0,M}\right). 
  \end{equation}
\end{theorem}
In words, \Cref{theo:minor_main_gen} ensures that for any compact set $\msk \subset \rsetdd$, there exists $t_0,\gamma_{t_0} >0$ for which, for any $\gamma \in \ocint{0,\bgamma_{t_0}}$, $\msk$ is $1$-small for $R_{\gamma}^{\ceil{t_0/\gamma}}$.
The proof of this result can be read in Section~\ref{sec:proof_minorization}. The main steps are sketched out in the simple case of nondegenerate Langevin dynamics in Section~\ref{sec:idea_ULA}.

\paragraph{Lyapunov condition uniform in the timestep.}
\label{sec:Lyapunov_unif_timestep}

To ensure the existence of a unique stationary distribution $\mu_{\gamma}$ for $R_{\gamma}$, and obtain a rate of convergence to stationarity, we need to consider additional assumptions. We focus on conditions which allow to ensure the uniform $\LyapV$-geometric ergodicity of~$R_{\gamma}$ with a Lyapunov function built upon a function $U : \rset^d \to \rset$ satisfying the following condition. 

\begin{assumptionD}
  \label{ass:lip}
  The function $U : \rset^d \to \rset$ is~$\rmC^1$,  $U(x)\geq 0$ for any~$x \in \rset^d$, and $U(0) = 0$, $\nabla U(0)=0$. In addition, there exists $\LipGradU \geq 0$ such that for any $x,y \in\rset^d$, $\norm{\nabla U(x)-\nabla U(y)} \leq \LipGradU \norm{x-y}$.
\end{assumptionD}

In the sequel, we call this assumption \Cref{ass:lip}(U) in order to highlight that it is related to the existence of an appropriate function~$U$. The condition~$\inf_{ \rset^d} U \geq 0$ is not restrictive since any
function bounded from below can be shifted to be nonnegative. The
conditions~$U(0) = 0$ and~$\nabla U(0)=0$ could be relaxed but allow
to simplify some computations and are therefore considered for
ease of presentation. When $b$ comes from the gradient of a potential (up to some perturbation)
function, as discussed around~\eqref{eq:19}, then this potential is a
natural candidate for the function $U$. In this case, \Cref{ass:lip}(U) is necessary to ensure the stability of most of the schemes.

The Lyapunov function~$\Lyaexp:\rset^{2d}\to \rset$ we consider is
parametrized by a constant~$\varpi >0$, and is of the following
exponential form for any $\gamma\in \ocint{0,\bar{\gamma}}$ and
$x,v\in \rset^d$:
\begin{equation}
  \label{eq:16}
   \Lyaexp(x,v)=\exp\parenthese{\varpi\sqrt{1+\Lya(x,v)}}\eqsp,
\end{equation}
where~$\Lya : \rset^{2d} \to \rset_+$, defined as 
\begin{equation}
\Lya(x,v)= \frac{\kappa^2}{2} \norm{x}^2 + \norm{v}^2 + \frac{\kappa^2\gamma (1+\gamma^\delta \vartheta_\gamma)}{1-\gminorc} \ps{x}{v} + 2\parametreGradU U(x) \eqsp,
\label{eq:def_function_lyap_super_quad_init}
\end{equation}
is the sum of~$U(x)$ and a positive definite quadratic form in~$(x,v)$. The parameters~$\vartheta_\gamma,\parametreGradU$ are introduced in \Cref{ass:contoleGAndF} below. As in previous studies on discretization of Langevin dynamics such as~\cite[Equation (8.3)]{mattingly:stuart:higham:2002}, the Lyapunov function $\Lyaexp$ depends on the stepsize~$\gamma$ through~$\Lya$ (in fact, $\Lya$ converges as~$\gamma \to 0$ to a Lyapunov function for the continuous dynamics~\eqref{eq:general_SDE}, similarly to the family of Lyapunov functions considered in \cite{mattingly:stuart:higham:2002}; see \cite[Equation~(3.6)]{mattingly:stuart:higham:2002}).

By the estimates in \Cref{sec:proof-crefth_drift} (see~\Cref{lem:minoration_lya} and~\eqref{eq:upper_bound_QF_Lya}), the design of $\Lya$ ensures that there exist~$\cWb,\cWu \in \rset_+$ and~$\cbgamma > 0$ such that, for any~$\gamma \in \ocint{0,\cbgamma}$,
\begin{equation}
  \label{eq:17}
  \cWb \LyaZmain \leq \Lya \leq \cWu \LyaZmain \eqsp, \qquad \LyaZmain(x,v) = \norm{x}^2 + \norm{v}^2 + U(x) \eqsp.
\end{equation}
Therefore, 
there exist $\varpi_1,\varpi_2 >0$ such that for any~$\varpi >0$, $\gamma \in \ocint{0,\cbgamma}$ and $x,v \in \rset^{d}$,
\begin{equation}
  \label{eq:18}
   \left[ \LyaZmainexp{\varpi}(x,v) \right]^{\varpi_1} \leq \Lyaexp(x,v) \leq \left[ \LyaZmainexp{\varpi}(x,v) \right]^{\varpi_2} \eqsp, \qquad \text{ where } \LyaZmainexp{\varpi}(x,v)   =\exp\parenthese{\varpi\sqrt{1+\LyaZmain(x,v)}}\eqsp. 
\end{equation}
In particular, convergence bounds on~$\Vnorm[ \Lyaexp]{\updelta_{(x,v)} R_{\gamma}^k - \mu_{\gamma}}$ for $\varpi >0$ (where~$\mu_{\gamma}$ is the unique invariant probability measure associated with~$R_{\gamma}$) thus imply convergence bounds on~$\Vnorm[\LyaZmainexp{\varpi}^{\varpi_1}]{\updelta_{(x,v)} R_{\gamma}^k - \mu_{\gamma}}$.

In view of the minorization condition provided by~\Cref{theo:minor_main_gen} and using for example~\cite[Theorem 19.4.1]{douc:moulines:priouret:2018}, it is sufficient to establish a Lyapunov condition for~$R_{\gamma}^{\ceil{t_0/\gamma}}$ uniform in the stepsize~$\gamma>0$ in order to obtain exponential convergence bounds for~$R_\gamma$. To this end, we first establish a Lyapunov condition for~$R_{\gamma}$, under additional technical conditions on the family of functions~$f_{\gamma},g_{\gamma}$ for~$\gamma \in \ocint{0,\bgamma}$, in relation with the potential~$U$ considered in~\Cref{ass:lip}(U). To state these conditions, we introduce the function $\fonctionControle : \rset^{3d} \times \rset^{m} \to \rset_+$ defined for $x,v,z \in\rset^{d}$ and $w \in\rset^{m}$ as
\begin{equation}
  \label{eq:5}
   \fonctionControle\left(x,v,z,w\right) = \frac{\norm{\nabla U(x)}^2}{L^2} + \norm{v}^2+\norm{z}^2+\norm{w}^2+\norm{x} \eqsp.
\end{equation}
Note that the position~$x$ appears through the two terms~$\norm{\nabla U(x)}^2$ and~$\norm{x}$ (mind the fact that the latter norm is not squared).
  
\begin{assumptionD}
  \label{ass:contoleGAndF}
  $\msw = \rset^{m}$, $\mcw = \mcbb(\rset^{m})$ and there exist $\parametreGradU,\constasssuperexpone>0$, $\deltau\in \ocint{0,1}$ and $\constgGradU\geq 0$ for which, for any $\gamma\in \ocint{0,\bar{\gamma}}$, there is~$\vartheta_{\gamma} \in\rset$ with
  \[
  \sup_{\upgamma\in\ocint{0,\bgamma}} \abs{\vartheta_{\upgamma}} \leq \bvartheta\eqsp , 
  \]
  such that the following estimates hold: for any $x,v,z\in \rset^d$ and for $\muw$-almost every $w \in \rset^{m}$,
  \begin{equation}
    \label{eq:ass:contoleGAndF_1}
    \begin{aligned}
      \norm{f_\gamma\left(x,\gamma^{\delta}v, \gamma^{\delta+\half} \sigma_{\gamma} z,w\right)}^2 & + \norm{g_\gamma\left(x,\gamma^{\delta}v, \gamma^{\delta+\half} \sigma_{\gamma} z,w\right)+\parametreGradU\nabla U(x)}^2 \\
      & \qquad \qquad \leq \constgGradU\parentheseDeux{1+\gamma^{\deltau}\fonctionControle\left(x,v, \gamma^{\half} \sigma_{\gamma} z,w\right)} \eqsp ,
    \end{aligned}
  \end{equation}
  and
  \begin{align}
     \ps{x}{f_\gamma\left(x,\gamma^{\delta}v, \gamma^{\delta+\half} \sigma_{\gamma} z,w\right)}&\leq \gamma^\delta\vartheta_\gamma\ps{x}{v}+\gamma^\deltau\constgGradU\norm{x}\norm{w_1}+\constgGradU\parentheseDeux{1+\gamma^{\deltau}\fonctionControle\left(x,v, \gamma^{\half} \sigma_{\gamma} z,w\right)} \eqsp ,\\
    \ps{x}{g_\gamma\left(x,\gamma^{\delta}v, \gamma^{\delta+\half} \sigma_{\gamma} z,w\right)}&\leq -\constasssuperexpone\parentheseDeux{\frac{\norm{\nabla U(x)}^2}{L^2}+\norm{x}} + \constgGradU\gamma^{\deltau}\parentheseDeux{1+\fonctionControle\left(x,v, \gamma^{\half} \sigma_{\gamma} z,w\right)}\eqsp.
  \end{align}
\end{assumptionD}

As for \Cref{ass:lip}(U), we refer to this assumption as \Cref{ass:contoleGAndF}(U) in the sequel.
The condition \eqref{eq:ass:contoleGAndF_1} means that, at first order in $\gamma$, $g_\gamma$ is some bounded perturbation of $\parametreGradU \nabla U$, which holds in many applications. 
The parameter~$\vartheta_{\gamma}$ is the prefactor of the term linear
in~$v$ in the expression of~$f_\gamma$ (see the expressions of the functions~$f_\gamma$ for the examples presented in Section~\ref{sec:exampl-numer-schem}). This explains why an unsigned
term~$\gamma^\delta\vartheta_\gamma\ps{x}{v}$ appears on the right
hand side of the third inequality above.  Since $f_{\gamma}$ and
$g_{\gamma} + \parametreGradU \nabla U$ corresponds to some form of second order expansion in the timestep~$\gamma$ for the schemes we consider, the fact that
$\deltau >0$ in \Cref{ass:contoleGAndF}(U) is relatively easy to
verify. Finally, note that if \Cref{ass:contoleGAndF}(U) is satisfied for
$\deltau >1$, we can replace this parameter by $1\wedge \deltau$ upon
modifying the constants appearing in this assumption.

Note that \Cref{ass:contoleGAndF}(U) is satisfied by the Euler-Maruyama scheme \eqref{eq:EM_scheme} and the ones described in \Cref{sec:exampl-numer-schem} in the case $b = -\nabla U$ with $U$ satisfying \Cref{ass:lip}(U) and
\begin{equation}
  \label{eq:19}
  \liminf_{\norm{x} \to \plusinfty}\frac{ \ps{\nabla U(x)}{x}}{\norm{x} + \norm[2]{\nabla U(x)}} >0 \eqsp. 
\end{equation}
The latter condition is satisfied by potentials which are convex (see \cite[Lemma 2.2]{bakry:barthe:cattiaux:guillin:2008}) or behave at infinity as~$\norm{x}^a$ with~$1 \leq a \leq 2$. In order to illustrate that \Cref{ass:contoleGAndF}(U) is a natural assumption, we check in the \supplementname~\Cref{sec:check_D3} the following result for the splitting scheme leading to~\eqref{eq:second_order_splitting_complicated} (the other cases considered in \Cref{sec:exampl-numer-schem} being easier). 

\begin{proposition}
  \label{propo:verif_drift_condition_splitting_scheme}
  Assume that $b = -\nabla U$ with $U:\rset^d \to \rset$ satisfying \Cref{ass:lip}(U) and \eqref{eq:19}. Then the condition~\Cref{ass:contoleGAndF}(U) holds for the second order splitting scheme leading to~\eqref{eq:second_order_splitting_complicated}.
\end{proposition}

The condition \Cref{ass:contoleGAndF}(U) also holds when~$b$ is an appropriate perturbation of~$-\nabla U$. Let us emphasize that we consider a weaker assumption than in~\cite{mattingly:stuart:higham:2002} even when~$b$ derives from a potential. More precisely, in the case $b = -\nabla U$, the condition \eqref{eq:19} is strengthened in \cite[Corollary~7.4]{mattingly:stuart:higham:2002} to $\liminf_{\norm{x} \to \plusinfty} \ps{\nabla U(x)}{x}/(\norm[2]{x} + \norm[2]{\nabla U(x)}) >0$.

To ease the presentation of the main results, we consider the following simple assumption on the additional noise $(W_{k+1})_{k\in\nset}$. 

\begin{assumptionD'}
\label{ass:Dprime}
  The random variables~$(W_{k+1})_{k \in \mathbb{N}}$ are i.i.d. $d$-dimensional standard Gaussian random variables, and there exists $\LLS\geq 0$ such that for any $x,v,z,z',w,w'\in\rset^d$, 
  \[
  \norm{f_\gamma(x,v,z,w)-f_\gamma(x,v,z',w')}+\norm{g_\gamma(x,v,z,w)-g_\gamma(x,v,z',w')} \leq \LLS\norm{(z,w)-(z',w')} \eqsp.
  \]
\end{assumptionD'}

In fact, we consider the more general condition \Cref{ass:w} on $(W_{k+1})_{k\in\nset}$  in \Cref{sec:general_assumptions_noise}, which includes \Cref{ass:Dprime} as a special instance. We postpone its presentation since it is mainly motivated by the use of stochastic gradients in discretization schemes for \eqref{eq:general_SDE}, described in \Cref{sec:numer-schem-with}. Therefore, \Cref{ass:w} will be more transparent once this family of discretization schemes is introduced. Nevertheless,  we state and prove our next result under this general assumption. 




\begin{theorem}
  \label{theo:drift_exp}
  Consider a potential~$U$ satisfying \Cref{ass:lip}(U), \Cref{ass:contoleGAndF}(U), and assume that \Cref{ass:gminorc} and \Cref{ass:w} (or~\textbf{D'}) hold. Then there exist $\varpi,\bar{\gamma}>0$, $\lambda \in\ooint{0,1}$ and $K,b\geq0$ (which all depend on~$U$) such that, for any $\gamma\in \ocint{0,\bar{\gamma}}$,
  \begin{equation}
  \forall x,v\in \rset^d, \qquad \Rkerminorc{\gamma}\Lyaexp(x,v) \leq \lambda^{\gamma}\Lyaexp(x,v)+\gamma b\1_{\ccint{0,K}}(\norm{x}+\norm{v}) \eqsp,
  \end{equation}
  where $\Lyaexp$ is defined in \eqref{eq:16}.
\end{theorem}

The proof of this result is postponed to
\Cref{sec:proof-crefth_drift}.

\paragraph{Exponential convergence uniform in the timestep.}
\label{sec:expo_cv_unif_timestep}

When the statement of \Cref{theo:drift_exp} holds, \cite[Lemma 1]{durmus:moulines:2016} implies that, for any $k \in \nsets$ and~$\gamma\in \ocint{0,\bar{\gamma}}$,
\begin{equation}
  \label{eq:15}
  \forall x,v\in \rset^d, \qquad \Rkerminorc{\gamma}^k \Lyaexp(x,v) \leq \lambda^{k\gamma}\Lyaexp(x,v)+ b \frac{1}{\lambda^{\bgamma}|\log \lambda|} \eqsp. 
\end{equation}
Therefore, in view of~\cite[Theorem~19.4.1]{douc:moulines:priouret:2018} applied
to~$\Rkerminorc{\gamma}^k$ with $k = \ceil{t_0/\gamma}$ (where~$t_0$ is
such that the conclusions of \Cref{theo:minor_main_gen} hold), we  obtain the following result.

\begin{theorem}
  \label{theo:geo_convergence}
  Consider a potential~$U$ satisfying \Cref{ass:lip}(U) and \Cref{ass:contoleGAndF}(U), and assume that~\Cref{ass:gminorc}, \Cref{ass:lip_minor_gen} and \Cref{ass:w} hold. Then there exists $\bgamma >0$ such that, for any $\gamma \in\ocint{0,\bgamma}$, $R_{\gamma}$ admits a unique invariant probability measure~$\mu_{\gamma}$. Moreover, there exist $\varpi>0$, $A\geq 0$ and $\rho \in\ooint{0,1}$ such that, for any $\gamma\in \ocint{0,\bar{\gamma}}$,
  \begin{equation}
  \label{eq:Lyap_estimate_unif}
  \forall x,v\in \rset^d, \qquad \Vnorm[\Lyaexp]{\updelta_{(x,v)} R_{\gamma}^k - \mu_{\gamma}} \leq A\rho^{k\gamma} \Lyaexp(x,v)\eqsp.
  \end{equation}
\end{theorem}

Then, combining this result and \eqref{eq:18} we get that there exist $\varpi_1,\varpi_2>0$ such that, for any $\gamma\in \ocint{0,\bar{\gamma}}$,
\begin{equation}
  \forall x,v\in \rset^d, \qquad \VnormEq[\LyaZmainexp{\varpi}^{\varpi_1}]{\updelta_{(x,v)} R_{\gamma}^k - \mu_{\gamma}}  \leq A\rho^{k\gamma}  \LyaZmainexp{\varpi}^{\varpi_2}(x,v)\eqsp,
\end{equation}
where $\LyaZmainexp{\varpi}$ is defined in \eqref{eq:18}.

Another corollary of~\eqref{eq:Lyap_estimate_unif} is the following: there exists~$\mathcal{K} \geq 0$ (which can be computed explictly in terms of~$A,\rho$) such that, for any measurable function~$\upphi : \rset^{2d} \to \rset$ with $\int_{\rset^d} \upphi \, \rmd \pi_{\gamma} = 0$ and $\Vnorm[\Lyaexp]{\upphi} < \plusinfty$, and for any~$\gamma \in (0,\bar{\gamma}]$, the function
  \[
  \uppsi_{\gamma} = \gamma \sum_{k=0}^{\plusinfty} R_{\gamma}^k \upphi
  \]
  is well-defined, satisfies
  \[
  \Vnorm[\Lyaexp]{\uppsi_{\gamma}} \leq \mathcal{K} \Vnorm[\Lyaexp]{\upphi},
  \]
  and is the unique solution to the following Poisson equation associated with $R_{\gamma}$ and $\upphi$ in the Banach space of measurable functions with finite $\Vnorm[\Lyaexp]{\cdot}$-norm and average~0 with respect to~$\pi_\gamma$:
  \[
  \frac{\Id-R_{\gamma}}{\gamma} \uppsi_{\gamma} = \upphi.
  \]

\section{Strategy of proof of \Cref{theo:minor_main_gen} in a simple case}
\label{sec:idea_ULA}

We present in this section the main ideas behind the proof of Theorem~\ref{theo:minor_main_gen}. We illustrate the strategy in the simplest case, namely (overdamped) Langevin dynamics
\[
\rmd \rmX_t = b(\rmX_t) \, \rmd t  + \BM_t \eqsp,
\]
discretized by a Euler--Maruyama scheme
\begin{equation}
  \label{eq:def_ULA}
  \XF_{k+1} = \XF_k + \gamma b(\XF_k) + \sqrt{\gamma} Z_{k+1} \eqsp, 
\end{equation}
where $b :\rset^d \to \rset^d$ and $(Z_{k+1})_{k \in\nset}$ is a sequence of \iid~standard $d$-dimensional Gaussian random variables.
Denoting by $\Qkerminorc{\gamma}$ the transition kernel of~\eqref{eq:def_ULA} (defined analogously to~\eqref{eq:def_R_gamma}), the counterpart of \Cref{theo:minor_main_gen} reads as follows. 

\begin{theorem}
  \label{theo:minor_EM}
  Assume that~$b \in \rmC^1(\rset^d,\rset^d)$ is globally Lipschitz. Then there exists $\bart_0 >0$ such that, for any $t_0 \in \ocint{0,\bart_0}$ and $M \geq 0$, there are $\varepsilon_{t_0,M}>0$ and $\bgamma_{t_0} >0$ for which, for any $\gamma \in \ocint{0,\bgamma_{t_0}}$ and $x,x' \in \cballdim{d}{\Zerobf_{d}}{M}$,
  \begin{equation}
    \label{eq:propo:minor_main_EM}
 \tvnormEq{\Qkerminorc{\gamma}^{\ceil{t_0/\gamma}}(x,\cdot)  - \Qkerminorc{\gamma}^{\ceil{t_0/\gamma}}(x',\cdot)} \leq 2( 1-\varepsilon_{t_0,M}). 
  \end{equation}
\end{theorem}

To prove this result, we need to study the distribution of~$\XF_{k}$ defined by \eqref{eq:def_ULA} for $k$ of order~$t_0/\gamma$ for some $t_0 >0$. The main idea is that the recursion \eqref{eq:def_ULA} can be seen as a perturbation of the evolution with~$b\equiv 0$ provided~$t_0$ is sufficiently small. To this end, we first notice that a straightforward induction gives, for any $k \in \nset$,
\begin{equation}
  \label{eq:ula_G_k}
\XF_{k+1} = \XF_0 + G^{(k+1)} + \gamma \sum_{i=0}^k b(\XF_i) \eqsp, \qquad G^{(k+1)} = \sqrt{\gamma} \sum_{i=1}^{k+1} Z_i \eqsp.
\end{equation}
If $t_0>0$ is fixed,   the variance of $G^{(k+1)}$ is of order $t_0$ for $k \approx t_0/\gamma$, from which the proof of \Cref{theo:minor_EM} easily follows when~$b \equiv 0$. To treat the case $b \not \equiv 0$, we rewrite~\eqref{eq:ula_G_k} as
\begin{equation}
  \label{eq:ula_G_2_k}
  \XF_{k+1} =  G^{(k+1)} +  \Phi_{\tZbfcs}^{(k+1)}(G^{(k+1)}) ,
\end{equation}
for some application~$\Phi_{\tZbfcs}^{(k+1)} : \rset^{d} \to \rset^d$, where $ \tZbfcs$ stands for~$(\tZc{k+1}_1,\ldots,\tZc{k+1}_{k})$, which are some \iid~Gaussian random variables constructed from~$(Z_1,\dots,Z_k)$ (see the precise definition~\eqref{def_tZc_ULA} below) and independent of~$G^{(k+1)}$. We show in addition that $g \mapsto \Phi_{\tZbfcs}^{(k+1)}(g)$ is Lipschitz with Lipschitz constant strictly smaller than~$1$ for~$t_0$ sufficiently small. As a result, the mapping $g \mapsto g + \Phi_{\tZbfcs}^{(k+1)}(g)$ is a $\rmC^1$ diffeomorphism, so that, by a change a variable, the random variable~$\XF_{k+1}$ admits a density with respect to the Lebesgue measure, and this density can even be bounded from below.

In order to rigorously formalize the above discussion, we first need to provide expressions for the random variables~$(\tZc{k+1}_1,\ldots,\tZc{k+1}_{k})$. Relying on Cochran's theorem, natural candidates are obtained by a linear combination of the first~$k$ original Gaussian increments~$(Z_1,\dots,Z_{k})$ and~$G^{(k+1)}$, as
\begin{equation}
  \label{def_tZc_ULA}
  \tZc{k+1}_i = Z_i - \frac{1}{\sqrt{\gamma}(k+1)}G^{(k+1)}\eqsp, \qquad i \in \iint{1}{k} \eqsp. 
\end{equation}
The random vector~$(\tZc{k+1}_1,\ldots,\tZc{k+1}_{k},G^{(k+1)})$ is still a Gaussian vector, with~$(\tZc{k+1}_1,\ldots,\tZc{k+1}_{k})$ independent of~$G^{(k+1)}$ by construction. This allows to express the iterated transition kernel as follows, upon introducing $n_0 = t_0/\gamma$ (assuming for simplicity that~$t_0/\gamma \in \mathbb{N}$ and $t_0/\gamma>1$): 
\begin{equation}
  \label{eq:R_gamma_iterated_EM}
  \Qkerminorc{\gamma}^{n_0}(x,\msb) = \int_{\rset^{n_0 d}} \1_{\msb}\left\{\gminor + \Phi_{\tzbf}^{(n_0)}(\gminor)\right\} \mathbf{N}_{0,t_0 \Id_d}(\gminor) \, \mathbf{N}_{0,\Sigma_{n_0}}(\tzbf) \, \rmd \tzbf \, \rmd\gminor \eqsp,
\end{equation}
where we denote by~$\mathbf{N}_{0,\Sigma}$ the density of Gaussian random variables with mean~0 and covariance matrix~${\Sigma}$, and $\Sigma_{n_0}$ is the covariance matrix of $(\tZc{n_0}_1,\ldots,\tZc{n_0}_{n_0-1})$. We can show similarly to \Cref{lem:indep_z_tilde_gen} that $\Sigma_{n_0}$ is positive definite.

Now that the output of the Markov chain has been rewritten as a perturbation of the output obtained with~$b=0$, we can proceed with a quantitative analysis to obtain lower bounds on~\eqref{eq:R_gamma_iterated_EM}. By a reasoning similar to the one leading to \Cref{propo:inverse_lip_Gamma} below, there exists~$\bar{t}_0>0$ (sufficiently small) such that the function~$g\mapsto \Phi_{\tzbf}^{(n_0)}(g)$ is Lipschitz with Lipschitz constant strictly smaller than~$1$ for any $t_0 \in \ocint{0,\bar{t}_0}$ (provided $\gamma>0$ is sufficiently small). Therefore, $\Id + \Phi_{\tzbf}^{(n_0)}$ is a perturbation of the identity and hence a $\rmC^1(\rset^d,\rset^d)$-diffeomorphism. Denoting by~$\GammaminorcF{\tzbf}[n_0]$ its inverse, we obtain, by a change of variable, 
\begin{equation}
  \label{eq:ula_G_4_k}
  \Qkerminorc{\gamma}^{n_0}(x,\msb) = \int_{\rset^{n_0 d}}  \1_{\msb}(u) \, \mathbf{N}_{0,t_0 \Id_d}\left( \GammaminorcF{\tzbf}[n_0](u) \right) \eqsp \JacD_{\tzbf}^{(n_0)}(u) \, \mathbf{N}_{0,\Sigma_{n_0}}(\rmd \tzbf) \, \rmd \tzbf \, \rmd u \eqsp,
\end{equation}
where $\JacD_{\tzbf}^{(n_0)}(u)$ is the absolute value of the determinant of the Jacobian matrix associated with~$\GammaminorcF{\tzbf}[n_0]$.

The next step is to construct some minorization probability measure to provide a lower bound on $\Qkerminorc{\gamma}^{n_0}(x,\msb)$ based on~\eqref{eq:ula_G_4_k}. We need to this end to explicitly indicate the dependence of~$\GammaminorcF{\tzbf}[n_0]$ on the initial condition~$x$, as the point~$x=\Zd$ will serve as a reference initial condition. We then consider
\begin{equation}
  \mu_{t_0}(\msb) = \int_{\rset^{n_0 d}} \1_{\msb}\left\{\frac{\gminor}{\sqrt{2}} + \Phi_{\Zd,\tzbf}^{(n_0)}\left(\frac{\gminor}{\sqrt{2}}\right)\right\} \mathbf{N}_{0,t_0 \Id_d}(\gminor) \, \mathbf{N}_{0,\Sigma_{n_0}}(\tzbf) \, \rmd \tzbf \, \rmd\gminor \eqsp,
\end{equation}
where the subscript~$\Zd$ in~$\Phi_{\Zd,\tzbf}^{(n_0)}$ indicates that this corresponds to~\eqref{eq:ula_G_2_k} starting from~$\XF_0 = \Zd$. This expression is similar to~\eqref{eq:R_gamma_iterated_EM}, except that the initial condition is set to~$\Zd$ and~$g$ is replaced by~$g/\sqrt{2}$ (in order to make use of the inequality~\eqref{eq:ineq_Gaussian_A_2} below). Similarly to~\eqref{eq:ula_G_4_k}, a change of variable leads to
\begin{equation}
  \mu_{t_0}(\msb) = 2^{d/2} \int_{\rset^{n_0 d}}  \1_{\msb}(u) \, \mathbf{N}_{0,t_0 \Id_d} \left( \sqrt{2}\GammaminorcF{\Zd,\tzbf}[n_0](u) \right) \eqsp \JacD_{\Zd,\tzbf}^{(n_0)}(u) \, \mathbf{N}_{0,\Sigma_{n_0}}( \tzbf)  \, \rmd \tzbf \, \rmd u \eqsp.
\end{equation}
From this expression and \eqref{eq:ula_G_4_k}, we get 
\begin{equation}
\Qkerminorc{\gamma}^{n_0}(x,\msb) \geq 2^{-d/2} \int_{\rset^{d}} \1_{\msb}(u)  \inf_{\tzbf\in \rset^{(n_0-1) d}} \defEns{A_{x,\tzbf}(u)} \mu_{t_0}(\rmd u) \eqsp,
\end{equation}
where~$A_{x,\tzbf}(u) = A^{(1)}_{x,\tzbf}(u)A^{(2)}_{x,\tzbf}(u)$ with
\begin{equation}
A^{(1)}_{x,\tzbf}(u) = \frac{J_{x,\tzbf}^{(n_0)}(u)}{J_{\Zd,\tzbf}^{(n_0)}(u)} \eqsp,
\qquad \qquad
A^{(2)}_{x,\tzbf}(u) = \frac{\mathbf{N}_{0,t_0 \Id_d} \left( \GammaminorcF{x,\tzbf}[n_0](u) \right)}{\mathbf{N}_{0,t_0 \Id_d} \left( \sqrt{2}\GammaminorcF{\Zd,\tzbf}[n_0](u) \right)}.
\end{equation}
When $A_{x,\tzbf}$ is lower bounded by a positive quantity~$a_{t_0,M}$ for~$x \in \cballdim{d}{\Zerobf_d}{M}$, we obtain $\Qkerminorc{\gamma}^{n_0}(x,\msb) \geq  a_{t_0,M}2^{-d/2} \mu_{t_0}(\msb)$,
which immediately implies~\eqref{eq:propo:minor_main_EM}. The lower bound on~$A_{x,\tzbf}(u)$ is proved in two steps:
\begin{enumerate}[label=(\alph*)]
\item by obtaining upper and lower bounds on the Jacobians~$J_{x,\tzbf}^{(n_0)}(u)$ to lower bound~$A^{(1)}_{x,\tzbf}(u)$, leveraging the fact that the mapping~$g \mapsto \GammaminorcF{\tzbf}[n_0](g)$ and its inverse are Lipschitz (as in \Cref{propo:inverse_lip_Gamma} below) and  making use of Hadamard's inequality (see \Cref{propo:bound_det_Lipsc_map});
\item by making use of the following inequality, which motivates the factor~$\sqrt{2}$ in the argument of the denominator of~$A^{(2)}_{\Zd,\tzbf}(u)$ and the definition of $\mu_{t_0}$:
  \begin{equation}
    \begin{aligned}
t_0 \log A^{(2)}_{x,\tzbf}(u) & = \norm{\GammaminorcF{\Zd,\tzbf}[n_0](u)}^2 -\frac12 \norm{\GammaminorcF{x,\tzbf}[n_0](u)}^2  \geq - \norm{\GammaminorcF{x,\tzbf}[n_0](u)-\GammaminorcF{\Zd,\tzbf}[n_0](u)}^2 \eqsp.     
    \end{aligned}
  \end{equation}
  Here we have used the expression of the density of the Gaussian random variable~$\Gminorc{n_0}$ defined in
\eqref{eq:ula_G_k} and the Young inequality on $\ps{a}{b}$ for $a,b \in \rset^{d}$, which implies that 
  \begin{equation}
  \label{eq:ineq_Gaussian_A_2}
 \|a -b\|^2 \geq \|a\|^2/2 - \|b\|^2 \eqsp. 
  \end{equation}
  When $x \mapsto \GammaminorcF{x,\tzbf}[n_0](u)$ is Lipschitz on~$\cballdim{d}{\Zerobf_d}{M}$ with a constant~$K_{t_0}$, uniformly in~$\gamma \in \ocint{0,\gamma_{t_0}}$ (as in \Cref{propo:inverse_lip_gamma_initial_condition} below), one finds $A^{(2)}_{x,\tzbf}(u) \geq \rme^{-K_{t_0}^2 M^2/t_0}$ for all~$x \in \cballdim{d}{\Zerobf_d}{M}$.
\end{enumerate}

\section{Examples of admissible numerical schemes}
\label{sec:exampl-numer-schem}

We show in this section that most popular discretization schemes associated to \eqref{eq:general_SDE} can be cast into the framework considered in~\eqref{eq:def_xminor_v_minor_gen}, and that they satisfy \Cref{ass:gminorc} and \Cref{ass:lip_minor_gen}. All these numerical schemes can be obtained by decomposing the generator~$\generator$ of~\eqref{eq:general_SDE} as
\begin{equation}
  \label{eq:L=A+B+C}
  \generator = \scrA + \scrB + \scrC \eqsp,
\end{equation}
where, for any $g \in \rmC^2(\rset^{2d})$ and~$(x,v) \in \rset^{2d}$, the elementary operators~$\scrA,\scrB,\scrC$ act as
\begin{align}
  \label{eq:2}
  \scrA g(x,v) &= v^{\transpose} \nabla_x g(x,v)\eqsp, \quad \scrB g(x,v) = b(x)^{\transpose} \nabla_v g(x,v) \eqsp, \quad
  \scrC g(x,v)  = -\kappa v^{\transpose} \nabla_{v} g(x,v) + \frac{\sigma^2}{2} \Delta_v g(x,v) \eqsp.
\end{align}
Note that the dynamics associated with $\scrA$ and $\scrB$ simply correspond to the deterministic flows $t \mapsto (x+vt,v)$ and $t \mapsto (x,v+b(x)t)$, while the dynamics associated with $\scrC$ is the Ornstein--Uhlenbeck process $(x,\rme^{-\kappa t}v+\sigma \int_{0}^t \rme^{-\kappa(t-s)} \rmd \BM_s)_{t \geq 0}$, where, for any $t > 0$, the random variable~$\rme^{-\kappa t}v+\sigma \int_{0}^t \rme^{-\kappa(t-s)} \rmd \BM_s$ is Gaussian with mean~$\rme^{-\kappa t}v$ and covariance matrix~$t \tsigma_t^2 \Idd$, with
\begin{equation}
  \label{eq:def_sigma_t_example}
  \tsigma_t^2 = \sigma^2 \frac{1-\rme^{-2\kappa t}}{2\kappa t } \eqsp. 
\end{equation}
We consider three classes of schemes: stochastic exponential Euler schemes in Section~\ref{sec:A+C_B_splitting}, splitting schemes in Section~\ref{sec:splitting-schemes}, and numerical discretizations relying on stochastic gradients in Section~\ref{sec:numer-schem-with}.

\subsection{Stochastic exponential Euler scheme}
\label{sec:A+C_B_splitting}

The elementary stochastic dynamics with generator $\scrA+\scrC$ is also analytically integrable and corresponds to an Ornstein--Uhlenbeck process (see \Cref{lem:OU_equiv}). One obtains the following numerical scheme by this analytic integration, upon fixing the drift to the current value~$b(\XminorTer_k)$:
\begin{equation}
  \label{eq:def_discretization_semi_euler}
\begin{aligned}
    \XminorTer_{k+1}  & = \XminorTer_k  +\frac{1-\rme^{-\kappa \gamma}}{\kappa} \VminorTer_k + \frac{\kappa \gamma + \rme^{-\kappa \gamma}-1}{\kappa^2} b (\XminorTer_k) + \eta_{k+1} \eqsp ,\\
    \VminorTer_{k+1}  & = \rme^{-\kappa \gamma} \VminorTer_k+ \frac{1-\rme^{-\kappa \gamma}}{\kappa} b (\XminorTer_k) + \xi_{k+1} \eqsp ,
\end{aligned} 
\end{equation}
with
\begin{equation}
  \label{eq:def_xi_eta}
\xi_{k+1} = \sigma \int_0^\gamma \rme^{-\kappa (\gamma-s)} \rmd \BM_{k\gamma+s} \eqsp,
\qquad
\eta_{k+1} = \sigma \int_0^\gamma \frac{1-\rme^{-\kappa (\gamma-s)}}{\kappa} \rmd \BM_{k\gamma+s} \eqsp.
\end{equation}
The scheme~\eqref{eq:def_discretization_semi_euler} corresponds to a stochastic exponential Euler integrator, see~\cite{shi2012convergence,komori:burrage:2014} and references therein in a general framework. It has also been recently considered and studied in the machine learning community, starting with~\cite{cheng:et:al:2018}. Similar schemes were in fact developed in the molecular dynamics communities in the late~70s based on the analytical solution of Langevin dynamics for~$b=0$ provided in~\cite{C43}, see for instance~\cite{EB80}. 

Note that the random variable~$(\eta_{k+1},\xi_{k+1})_{k \in \mathbb{N}}$, given in \eqref{eq:def_xi_eta}, is a family of i.i.d. Gaussian random vectors with covariance matrix
\begin{equation}
  \label{eq:def_continue_matrix}
\coeffContinueMatrix{}{\gamma} \otimes \Idd =  \begin{pmatrix} \coeffContinueMatrix{1}{\gamma} & \coeffContinueMatrix{2}{\gamma} \\
    \coeffContinueMatrix{2}{\gamma} & \coeffContinueMatrix{3}{\gamma} \\
  \end{pmatrix} \otimes \Idd \eqsp ,
\end{equation}
where $\coeffContinueMatrix{}{\gamma}$ has entries
\begin{align}
  \coeffContinueMatrix{1}{\gamma} & = \sigma^2 \int_0^\gamma \left( \frac{1-\rme^{-\kappa (\gamma-s)}}{\kappa} \right)^2 \rmd s = \frac{\sigma^2}{2\kappa^2}\left[2\gamma - \frac{3-4\rme^{-\kappa \gamma}+\rme^{-2\kappa \gamma}}{\kappa}\right] \eqsp, \label{eq:first_coeff_Sigma} \\ 
  \coeffContinueMatrix{2}{\gamma} & = \sigma^2\int_0^\gamma \rme^{-\kappa (\gamma-s)} \frac{1-\rme^{-\kappa (\gamma-s)}}{\kappa} \, \rmd s = \frac{\sigma^2(1-\rme^{-\kappa \gamma})^2}{2\kappa^2}  \eqsp, \label{eq:second_coeff_Sigma} \\
  \coeffContinueMatrix{3}{\gamma} & = \sigma^2\int_0^\gamma \rme^{-2\kappa (\gamma-s)} \, \rmd s = \frac{\sigma^2(1-\rme^{-2\kappa \gamma})}{2\kappa} \eqsp . \label{eq:third_coeff_Sigma}
\end{align}
Note that $\coeffContinueMatrix{3}{\gamma}/\gamma = \tsigma_\gamma^2$ (recall~\eqref{eq:def_sigma_t_example}). Introduce also $Z_{k+1} = \xi_{k+1} / \sqrt{\coeffContinueMatrix{3}{\gamma}}$,
and $W_{k+1}$ satisfying 
\begin{equation}
\eta_{k+1} = \sqrt{\coeffContinueMatrix{1}{\gamma}}\left(\alpha_\gamma Z_{k+1} + \sqrt{1-\alpha_\gamma^2}W_{k+1}\right) \eqsp, 
\qquad
\alpha_\gamma = \coeffContinueMatrix{2}{\gamma}\left/\sqrt{\coeffContinueMatrix{1}{\gamma}\coeffContinueMatrix{3}{\gamma}}\right.  \eqsp. 
\end{equation}
It can be verified (see \Cref{prop:inv}) that $\coeffContinueMatrix{}{\gamma}$ is invertible and~$\alpha_{\gamma} < 1$. Then, an easy computation shows that $(Z_{k+1})_{k \in \mathbb{N}}$ and~$(W_{k+1})_{k \in \mathbb{N}}$ are independent families of i.i.d. Gaussian random vectors with identity covariance matrix. With this notation, the numerical scheme~\eqref{eq:def_discretization_semi_euler} can be rewritten as
\begin{equation}
 \begin{aligned}
  \XminorTer_{k+1}  & = \XminorTer_k + \frac{1-\rme^{-\kappa \gamma}}{\kappa} \VminorTer_k + \frac{\kappa \gamma + \rme^{-\kappa \gamma}-1}{\kappa^2} b (\XminorTer_k) + \sqrt{\coeffContinueMatrix{1}{\gamma}(1-\alpha_\gamma^2)} W_{k+1} + \gamma^{3/2}\tsigma_\gamma \Dbf_\gamma Z_{k+1} \eqsp ,\\
  \VminorTer_{k+1}  & = \rme^{-\kappa \gamma} \VminorTer_k+ \frac{1-\rme^{-\kappa \gamma}}{\kappa} b (\XminorTer_k) +\sqrt{\gamma} \tsigma_\gamma Z_{k+1} \eqsp ,
\end{aligned} 
\end{equation}
where  $  \Dbf_\gamma =\coeffContinueMatrix{2}{\gamma}\left/ \parentheseDeux{\tsigma_\gamma \sqrt{\gamma^{3}\coeffContinueMatrix{3}{\gamma}}}\right.$.
This fits into the framework~\eqref{eq:def_xminor_v_minor_gen} upon setting $\delta=1$, $\gminorc = \rme^{-\kappa \gamma}$, $\sigma_{\gamma} = \tsigma_{\gamma}$,
\begin{equation}
g_\gamma(x,v,z,w) = \frac{1-\rme^{-\kappa \gamma}}{\kappa\gamma} b (x)\eqsp, \quad f_\gamma(x,v,z,w) = \frac{1-\kappa\gamma-\rme^{-\kappa \gamma}}{\kappa \gamma^2} \left( v - \frac{\gamma}{\kappa} b(x) \right) + \frac{1}{\gamma}\sqrt{\coeffContinueMatrix{1}{\gamma}(1-\alpha_\gamma^2)} w \eqsp.
\end{equation}
The conditions in~\Cref{ass:gminorc} are easily seen to be satisfied since $\tsigma_\gamma/\sigma \to 1$ as $\gamma \to 0$, while $\Dbf_\gamma \to 1/2$. Finally, the conditions in~\Cref{ass:lip_minor_gen} hold true when $b$ is Lipschitz. 

\subsection{Splitting schemes}
\label{sec:splitting-schemes}

A systematic way of constructing numerical schemes for Langevin dynamics is to rely on splitting procedures based on the decomposition~\eqref{eq:L=A+B+C}, as systematically studied in~\cite{LM12,LM15,leimkuhler:matthews:stoltz:2016} for kinetic Langevin dynamics. The interest of these schemes is that they reduce to symplectic integrators of the Hamiltonian dynamics as $\kappa \to 0$ when $\sigma^2 = \mathrm{O}(\kappa)$. We consider here first and second order numerical discretizations based on Lie-Trotter or Strang splittings built upon the operators $\scrA,\scrB,\scrC$ introduced in~\eqref{eq:L=A+B+C}. These operators are the generators of elementary SDEs which can be analytically integrated, as discussed after~\eqref{eq:2}.

\paragraph{First order schemes.} Schemes of weak order~1 are obtained, up to cyclic permutations, by composing the elementary dynamics associated with the operators in the order $\scrA,\scrB,\scrC$ or~$\scrA,\scrC,\scrB$. There are therefore 6 possible first order splitting schemes, which can all be cast in the framework~\eqref{eq:def_xminor_v_minor_gen}. Some general comments can be formulated on the 6 splittings schemes.
\begin{itemize}
\item There is no dependence on~$w$ for the functions~$f_\gamma,g_\gamma$, and no dependence on~$z$ for~$g_\gamma$.
\item Functions~$f_\gamma$ which genuinely depend on~$z$ are obtained when~$\scrC$ appears before~$\scrA$. This corresponds for example to the scheme  $\scrC \scrA \scrB$.
\item Functions~$g_\gamma$ which genuinely depend on~$v$ are obtained when~$\scrA$ appears before~$\scrB$. This corresponds for example to the scheme  $\scrC \scrA \scrB$.
\item The functions $f_\gamma,g_\gamma$ are quite similar for schemes were consecutive applications of~$\scrB,\scrC$ are exchanged (as for the evolutions~$\scrA\scrB\scrC$ and~$\scrA\scrC\scrB$; or for~$\scrB\scrC\scrA$ and~$\scrC\scrB\scrA$).
\end{itemize}
In view of these remarks, the simplest scheme, from a structural viewpoint, is associated with~$\scrB\scrA\scrC$ (although the schemes associated with~$\scrA\scrB\scrC$ and~$\scrA\scrC\scrB$ are also quite simple), while the most complicated is the one associated with~$\scrC\scrA\scrB$. We therefore consider more precisely these two schemes, as paradigmatic examples of first order splittings.

The numerical scheme associated with $\scrB\scrA\scrC$ is the one presented in \Cref{exam:verlet}.
The numerical scheme associated with $\scrC\scrA\scrB$ reads
\begin{equation}
\XminorTer_{k+1}  = \XminorTer_k+\gamma\rme^{-\kappa \gamma}\VminorTer_k + \gamma^{3/2} \tsigma_\gamma Z_{k+1}\eqsp, \quad 
\VminorTer_{k+1}  = \rme^{-\kappa \gamma} \VminorTer_k + \gamma b\left(\XminorTer_k+\gamma\rme^{-\kappa \gamma}\VminorTer_k + \gamma^{3/2} \tsigma_\gamma Z_{k+1}\right) + \sqrt{\gamma} \tsigma_\gamma Z_{k+1}\eqsp, 
\end{equation}
which corresponds to~\eqref{eq:def_xminor_v_minor_gen} with $\delta=1$, $\gminorc = \rme^{-\kappa \gamma}$, 
$\sigma_{\gamma} = \tsigma_{\gamma}$, $\Dbf_\gamma = \Idd$, and $f_\gamma(x,v,z,w) = \gamma^{-1}(\rme^{-\kappa \gamma}-1) v$, $g_\gamma(x,v,z,w) = b(x+\rme^{-\kappa \gamma}v + z)$.
For both schemes, the conditions \Cref{ass:gminorc} and \Cref{ass:lip_minor_gen} hold true when $b$ is Lipschitz.

\paragraph{Second order schemes.}
Schemes of weak order~2 are obtained, up to cyclic permutations of the operators, by a Strang splitting based on the operators~$\scrA,\scrB,\scrC$. The scheme~$\scrA\scrB\scrC\scrB\scrA$, for instance, corresponds to integrating the elementary dynamics associated with~$\scrA$ for a time~$\gamma/2$, then the elementary dynamics associated with~$\scrB$ for a time~$\gamma/2$, then the elementary dynamics associated with~$\scrC$ for a time~$\gamma$, then again the elementary dynamics associated with~$\scrB$ for a time~$\gamma/2$, and finally the elementary dynamics associated with~$\scrA$ for a time~$\gamma/2$.

As for first order splitting schemes, there are 6 possible schemes, which can all be cast in the framework~\eqref{eq:def_xminor_v_minor_gen}. Some general comments can be formulated on the 6 splittings schemes.
\begin{itemize}
\item There is no dependence on~$w$ for the functions~$f_\gamma,g_\gamma$ when~$\scrC$ is between instances of~$\scrA$ (as for the schemes~$\scrB\scrA\scrC\scrA\scrB$, $\scrA\scrB\scrC\scrB\scrA$ and~$\scrA\scrC\scrB\scrC\scrA$). This dependence is linear for~$f_\gamma$ when the operator~$\scrA$ is in the central place (as for the schemes~$\scrC\scrB\scrA\scrB\scrC$ and~$\scrB\scrC\scrA\scrC\scrB$).
\item The function~$g_\gamma$ does not depend on~$z,w$ for schemes with~$\scrA$ at the first and last places (as for the schemes~$\scrA\scrB\scrC\scrB\scrA$ and~$\scrA\scrC\scrB\scrC\scrA$).
\item The functions $f_\gamma,g_\gamma$ are quite similar when consecutive operators~$\scrB,\scrC$ are exchanged (as for the schemes~$\scrA\scrB\scrC\scrB\scrA$ and~$\scrA\scrC\scrB\scrC\scrA$; as well as for~$\scrB\scrC\scrA\scrC\scrB$ and~$\scrC\scrB\scrA\scrB\scrC$).
\end{itemize}
In view of these remarks, the simplest scheme, from a structural viewpoint, is associated with~$\scrA\scrB\scrC\scrB\scrA$ (or with~$\scrA\scrC\scrB\scrC\scrB$), while the most complicated is the one associated with~$\scrC\scrA\scrB\scrA\scrC$. We next write out more precisely these two schemes, as paradigmatic examples of second order splittings.

The numerical scheme associated with~$\scrA\scrB\scrC\scrB\scrA$ reads
\begin{equation}
 \begin{aligned}
\XminorTer_{k+1} & = \XminorTer_k+\frac{\gamma(1+\rme^{-\kappa \gamma})}{2} \VminorTer_k+ \frac{\gamma^2(1+\rme^{-\kappa \gamma})}{4} b\left(\XminorTer_k+\frac\gamma2 \VminorTer_k\right) + \frac{\gamma^{3/2}}{2} \tsigma_\gamma Z_{k+1} \eqsp, \\
\VminorTer_{k+1} & = \rme^{-\kappa \gamma} \VminorTer_k + \frac{\gamma(1+\rme^{-\kappa \gamma})}{2} b\left(\XminorTer_k+\frac\gamma2 \VminorTer_k\right) + \sqrt{\gamma} \tsigma_\gamma Z_{k+1}\eqsp,
\end{aligned} 
\end{equation}
where $\tsigma_{\gamma}$ is given by \eqref{eq:def_sigma_t_example}. 
It can indeed be written as \eqref{eq:def_xminor_v_minor_gen} with $\delta=1$, $\gminorc = \rme^{-\kappa \gamma}$, $\sigma_{\gamma} = \tsigma_{\gamma}$, $\Dbf_\gamma = \Idd/2$, and
\begin{equation}
f_\gamma(x,v,z,w) = \frac{\rme^{-\kappa \gamma}-1}{2 \gamma} v + \frac{\gamma(1+\rme^{-\kappa \gamma})}{4} b\left(x+\frac12 v\right),
\qquad 
g_\gamma(x,v,z,w) = \frac{1+\rme^{-\kappa \gamma}}{2} b\left(x+\frac12 v\right).
\end{equation}
The numerical scheme associated with~$\scrC\scrA\scrB\scrA\scrC$ reads
\begin{equation}
 \begin{aligned}
\XminorTer_{k+1} & = \XminorTer_k+\gamma \rme^{-\kappa \gamma/2} \VminorTer_k+\frac{ \gamma^2}{2} b\left(\XminorTer_k+\frac{\gamma \rme^{-\kappa \gamma/2}}{2} \VminorTer_k + \frac{\gamma^{3/2}}{2\sqrt{2}} \tsigma_{\gamma/2} \xi^1_{k+1}\right) + \frac{\gamma^{3/2}}{\sqrt{2}} \tsigma_{\gamma/2} \xi^1_{k+1}\eqsp, \\
\VminorTer_{k+1} & = \rme^{-\kappa \gamma} \VminorTer_k + \gamma \rme^{-\kappa \gamma/2} b\left(\XminorTer_k+\frac{\gamma \rme^{-\kappa \gamma/2}}{2} \VminorTer_k + \frac{\gamma^{3/2}}{2\sqrt{2}} \tsigma_{\gamma/2} \xi^1_{k+1}\right) + \sqrt{\frac\gamma2} \tsigma_{\gamma/2} \left(\rme^{-\kappa \gamma/2}\xi^1_{k+1}+\xi^2_{k+1}\right) \eqsp, 
\end{aligned} 
\end{equation}
where~$(\xi^1_{k+1})_{k \geq 0}$ and~$(\xi^2_{k+1})_{k \geq 0}$ are two independent families of \iid~standard $d$-dimensional Gaussian random variables. In order to write this scheme under the form \eqref{eq:def_xminor_v_minor_gen}, similarly to \Cref{sec:A+C_B_splitting}, we introduce the sequence of random variables~$(Z_{k+1})_{k \in \nset}$ defined as
\begin{equation}
  \sqrt{\frac\gamma2} \tsigma_{\gamma/2} \left(\rme^{-\kappa \gamma/2}\xi^1_{k+1}+\xi^2_{k+1}\right) = \sqrt{\gamma}\sigma_\gamma Z_{k+1} \eqsp,
  \qquad 
  \sigma_\gamma^2 = \tsigma_{\gamma/2}^2 \frac{1+\rme^{-\kappa\gamma}}{2}\eqsp.
\end{equation}
By construction, $(Z_{k+1})_{k \geq 0}$ are \iid~standard $d$-dimensional Gaussian random variables. We next consider the family of \iid~standard $d$-dimensional Gaussian random variables   $(W_{k+1})_{k \geq 0}$ independent of~$(Z_{k+1})_{k \geq 0}$ defined through the relation 
\begin{equation}
\xi^1_{k+1} = \alpha_\gamma Z_{k+1} + \sqrt{1-\alpha_\gamma^2} W_{k+1},
\qquad
\alpha_\gamma = \mathbb{E}\left[\xi^1_{k+1} Z_{k+1} \right] = \frac{\rme^{-\kappa \gamma/2}}{\sqrt{1+\rme^{-\kappa \gamma}}} \in \ooint{0,1} \eqsp.
\end{equation}
In terms of the random variables~$Z_{k+1},W_{k+1}$, the scheme~$\scrC\scrA\scrB\scrA\scrC$ can be reformulated as
\eqref{eq:def_xminor_v_minor_gen} with~$\delta = 1$, $\gminorc = \rme^{-\kappa \gamma}$, $    \Dbf_\gamma = \rme^{-\kappa \gamma/2}/(1+\rme^{-\kappa \gamma}) \Idd$ and 
\begin{equation}
  \label{eq:second_order_splitting_complicated}
  \begin{aligned}
    f_\gamma(x,v,z,w) & = \frac{\rme^{-\kappa \gamma/2}-1}{\gamma} v + \frac{\gamma}{2} b\left(x+\frac{\rme^{-\kappa \gamma/2}}{2} v + \frac{\rme^{-\kappa\gamma/2}}{2(1+\rme^{-\kappa\gamma})} z + \sqrt{\frac{\gamma^3 \tsigma_{\gamma/2}^2}{8(1+\rme^{-\kappa\gamma})}} w\right) +\sqrt{\frac{\gamma^3 \tsigma_{\gamma/2}^2}{2(1+\rme^{-\kappa\gamma})}} w\eqsp,\\
    g_\gamma(x,v,z,w) & = \rme^{-\kappa \gamma/2} b\left(x+\frac{\rme^{-\kappa \gamma/2}}{2} v + \frac{\rme^{-\kappa\gamma/2}}{2(1+\rme^{-\kappa\gamma})} z + \sqrt{\frac{\gamma^3 \tsigma_{\gamma/2}^2}{8(1+\rme^{-\kappa\gamma})}} w\right) \eqsp.
  \end{aligned}
\end{equation}
The conditions \Cref{ass:gminorc} and \Cref{ass:lip_minor_gen} are then easily seen to hold true for the two schemes~$\scrC\scrA\scrB\scrA\scrC$ and~$\scrC\scrA\scrB\scrA\scrC$ when $b$ is Lipschitz.

\subsection{Numerical schemes with stochastic gradients}
\label{sec:numer-schem-with}

We consider in this section discretizations of~\eqref{eq:general_SDE} using stochastic approximation strategies. Such a methodology is particularly appealing in statistics and machine learning where the field $b$ can be very expensive to evaluate or cannot even be accessed~\cite{duflo:2013,welling:the:2011}. In these contexts, the Langevin dynamics have been primarily considered for either  performing Bayesian inference \cite{ding:et:al:2014,ahn:korattikara:welling:2012,ma:chen:fox:2015} or optimizing an objective function \cite{pelletier:1998,raginsky:rakhlin:telgarsky:2017,barkhagen:et:al:2021,de2020efficient}.
In the first case, $b = \nabla \log \pi $ where $\pi :\rset^d \to \rset_+$ is the a posteriori distribution of a statistical model, which can generally be written as~$-\log \pi  = \sum_{k=1}^N U_k$, with~$N$ the number of observations and~$U_k:\rset^d \to \rset$.
It has also been proposed to use Langevin dynamics to find an element of $\argmin_{\rset^d} f$ for some function $f : \rset^d \to \rset$ by setting $b= -\nabla f$ and taking $\sigma$ small. In particular, we are interested in the situation where $\nabla f$ can only be estimated through estimators which can potentially be biased. 

The use of stochastic approximation for $b$ in these two settings can be formalized as follows. We suppose that there exist a probability measure $\mu_{\msy}$ on a measurable space $(\msy,\mcy)$, and a measurable function~$H : \rset^d \times \msy  \to \rset^d$ such that for any $  x \in \rset^d$,
\begin{equation}
  \label{eq:3}
 \widetilde{b}(x) = \int_{\msy}H_x(\rmy) \, \mu_{\msy}(\rmd\rmy) \eqsp, \qquad \sup_{\rset^d} \normLigne{\widetilde{b}-b} < \plusinfty \eqsp. 
\end{equation}
In addition, we suppose that we can generate \iid~samples $(Y_{k+1})_{k\in\nset}$ from $\mu_{\msy}$. Then, a stochastic approximation of the discretization of~\eqref{eq:general_SDE} essentially consists in replacing the evaluation of $b$ at each iteration~$k$ by~$H_{X_k}(Y_{k+1})$. For example, following~\cite{chen:fox:guestrin:2014}, the Euler--Maruyama discretization~\eqref{eq:EM_scheme} can be generalized as 
\begin{equation}
  \label{eq:sto_grad_rec}
  \XminorTer_{k+1} = \XminorTer_k+\gamma\VminorTer_k \eqsp,
  \qquad 
  \VminorTer_{k+1} = (1-\kappa \gamma) \VminorTer_k + \gamma H_{\XminorTer_k}(Y_{k+1}) + \sqrt{\gamma} \sigma Z_k \eqsp. 
\end{equation}
This numerical scheme fits into the framework~\eqref{eq:def_xminor_v_minor_gen} upon taking $W_k = Y_k$ for any~$k\geq 1$,
\begin{equation}
  \tau_{\gamma} = 1-\kappa \gamma \eqsp, \qquad \sigma_{\gamma} = \sigma\eqsp, \qquad \Dbf_{\gamma} = 0\eqsp, \qquad f_\gamma(x,v,z,w) = 0\eqsp, \qquad g_{\gamma}(x,v,z,w) = H_x(w)\eqsp.
\end{equation}
The parameter~$\delta$ is irrelevant. Note that \Cref{ass:gminorc} and~\Cref{ass:lip_minor_gen} hold under the condition that for any $y \in \msy$, the function~$x \mapsto H_{x}(y)$ is Lipschitz with a Lipschitz constant independent of $y$. The stochastic exponential Euler and splitting schemes presented in Sections~\ref{sec:A+C_B_splitting} and~\ref{sec:splitting-schemes} can also be adapted to take into account a stochastic approximation of~$b$, for instance by relying on the methodology developed in~\cite{chen:ding:carin:2015}.

\subsection{General assumptions on the noise}
\label{sec:general_assumptions_noise}

Now that we have presented numerical methods with stochastic gradients, we can specify the assumptions we need for the additional noise variables to prove \Cref{theo:drift_exp} and~\Cref{theo:geo_convergence}. The idea is that the random variables~$W$ can be decomposed into two parts $(W_1,W_2)$. The first random variable $W_1$ is a Gaussian noise resulting from a higher order integration in time of the Brownian motion. Therefore, the random variables $W_1$ should be considered as Gaussian random variables. The second random variable $W_2$ is related to the use of stochastic gradients. For example, it corresponds to the sequence $(Y_{k+1})_{k\in\nset}$ in \eqref{eq:sto_grad_rec}.

\begin{assumptionD}
  \label{ass:w}
  It holds $\msw = \rset^{m_1}\times \rset^{m_2}$ and $\mcw = \mcbb(\rset^{m_1})\otimes \mcbb(\rset^{m_2})$. The probability measure~$\muw$ can be decomposed as $\muw=\muwOne\otimes\muwTwo$, where $\muwOne$ and $\muwTwo$ are probability mesures on $\rset^{m_1}$ and $\rset^{m_2}$ respectively, such that the following two conditions are satisfied. 
  \begin{enumerate}[wide, labelwidth=!, labelindent=0pt,label=\arabic*)]
  \item
\label{ass:w_logsob_spec_1}
    \begin{enumerate}[label=(\alph*)]
      \item \label{ass:w_logsob_spec} There exists $\LSconstant>0$ such that, for any Lipschitz continuous function~$h:\rset^{d+m_1}\to \rset$, with Lipschitz constant $\normLip{h}$, and any~$s\in \rset$, it holds~$\int_{\rset^{d+m_1}} |h| \, \rmd (\varphibf\otimes\muwOne)<\plusinfty$ and
      \begin{equation}
        \int_{\rset^{d+m_1}} \exp\parenthese{sh}\rmd (\varphibf\otimes\muwOne)\leq \exp\parenthese{s\int_{\rset^{d+m_1}} h \, \rmd (\varphibf\otimes\muwOne)+\LSconstant \normLip{h}^2\frac{s^2}{2}} \eqsp;
      \end{equation}
      \item       \label{ass:w_lip} There exists $\LLS\geq 0$ such that for any $x,v,z,z'\in\rset^d$, $w_1,w'_1\in \rset^{m_1}$ and $w_2\in \rset^{m_2}$,
      \begin{multline}
        \norm{f_\gamma(x,v,z,(w_1,w_2))-f_\gamma(x,v,z',(w'_1,w_2))}+
        \norm{g_\gamma(x,v,z,(w_1,w_2))-g_\gamma(x,v,z',(w'_1,w_2))}\\\leq \LLS\norm{(z,w_1)-(z',w'_1)} \eqsp.
      \end{multline}
    \end{enumerate}
    \label{ass:w_logsob}
    \item There exists $\bgamma_W\in\ocint{0,\bgamma}$ such that $\displaystyle \int_{\rset^{m_2}} \rme^{\bgamma_W\norm{w_2}^2}\muwTwo(\rmd w_2)<\plusinfty$.
    \label{ass:w_compact}
  \end{enumerate}
\end{assumptionD}

\begin{remark}
  \label{rem:logsob}
  If $\muwOne$ admits a first moment and satisfies a log-Sobolev inequality with constant $\widetilde{\LSconstant}>0$, \ie~for any continuously differentiable function $h:\rset^{m_1}\to\rset_+$ such that $\int_{\rset^{m_1}} h(w_1) \muwOne(\rmd w_1)=1$,
  \begin{equation}
    \int_{\rset^{m_1}} h(w_1)\log(h(w_1))\, \muwOne(\rmd w_1)\leq 2\widetilde{\LSconstant}\int_{\rset^{m_1}}\norm{\nabla h}^2\muwOne(\rmd w_1)\eqsp ,
  \end{equation}
  then by \cite[Propositions~5.5.1 and~5.2.7]{bakry:gentil:ledoux:2014}, the probability measure~$\varphibf\otimes\muwOne$ satisfies a log-Sobolev inequality with constant $\max(1, \widetilde{\LSconstant})$, which is the case for all the examples we consider in Sections~\ref{sec:A+C_B_splitting} and~\ref{sec:splitting-schemes}. Therefore by Herbst’s argument (see \cite[Proposition 5.4.1]{bakry:gentil:ledoux:2014} or \cite[Theorem 5.5]{boucheron:lugosi:massart:2013}), condition~\Cref{ass:w}-\ref{ass:w_logsob_spec_1}-\ref{ass:w_logsob_spec} holds. 
\end{remark}

Note that \Cref{ass:w}-\ref{ass:w_compact} is not restrictive and
covers the case where~$\muwTwo$ is (sub-)Gaussian, which is frequently
the case for the schemes presented in \Cref{sec:numer-schem-with}. 

\section{Proof of the minorization condition}
\label{sec:proof_minorization}

The proof is organized in several steps. We start by rewriting the end point of the Markov chain in terms of the initial conditions and noise increments in \Cref{sec:algebraic_structure}, in a form amenable to perturbative treatments. The dominant part of the evolution is given by the outcome of discretization schemes corresponding to~\eqref{eq:def_xminor_v_minor_gen} in the case~$b=0$. We therefore carefully study this case in \Cref{sec:noise_estimation}, where we condition solutions by the sum of the random increments characterizing the endpoints, and write out a decomposition into intermediate increments independent of the sum. We next quantify, by stability estimates, how numerical solutions depend on the random increments used to generate them (see \Cref{sec:stability}). This finally allows us to prove Theorem~\ref{theo:minor_main_gen} in \Cref{sec:proof_main_resukt_gen} by considering the actual solutions of the numerical method as perturbations of discretizations of~\eqref{eq:def_xminor_v_minor_gen} in the case~$b=0$, provided the physical time~$t_0$ is sufficiently small. 

\subsection{Algebraic structure of the equations}
\label{sec:algebraic_structure}

It is convenient, in particular for the stability estimates of \Cref{sec:stability}, to rewrite $(\XminorTer_{k+1},\VminorTer_{k+1})$ as a function of $(\XminorTer_{0},\VminorTer_{0})$ and the realizations of the random variables $\{(Z_{i+1},W_{i+1})\}_{i=0}^k$ needed to define the $(k+1)$-th iterate.
First, in view of~\eqref{eq:def_xminor_v_minor_gen} and \eqref{minor_def_Gamma_b_gamma}, the iterates of the Markov chain can be written as $(\XminorTer_{k+1},\VminorTer_{k+1}) = \Gammaminorc{\gamma}\left(\XminorTer_{k},\VminorTer_{k},\sqrt{\gamma}\sigma_{\gamma} Z_{k+1}, W_{k+1}\right)$. This allows to write 
\begin{equation}
  \label{eq:def_Gammaminorc_k+1}
  (\XminorTer_{k+1},\VminorTer_{k+1}) = \Gammaminorc{\gamma}[k+1]\left(\XminorTer_{0},\VminorTer_{0}, \left\{\sqrt{\gamma}\sigma_{\gamma} Z_{i}\right\}_{i=1}^{k+1}, \{W_i \}_{i=1}^{k+1}\right) \eqsp,
\end{equation}
where the applications $\Gammaminorc{\gamma}[i] : \rset^{2d} \times \rset^{i \times d }\times \msw^i \to \rset^{2d}$ are recursively defined as follows:
\begin{equation}
  \label{eq:def_Phiminor_k}
  \begin{aligned}
    & \Gammaminorc{\gamma}[0](x,v) = (x,v) \eqsp, \\
    & \Gammaminorc{\gamma}[i]\left(x,v,\{z_j\}_{j=1}^{i},\{w_j\}_{j=1}^{i}\right) = \Gammaminorc{\gamma}\left(\Gammaminorc{\gamma}[i-1]\left(x,v,\{z_j\}_{j=1}^{i-1},\{w_j\}_{j=1}^{i-1}\right),z_{i},w_{i}\right)\eqsp, \qquad i \geq 1 \eqsp.
    \end{aligned} 
\end{equation}
The next result provides a more explicit and constructive expression for~$\Gammaminorc{\gamma}[k+1]$. It is stated for the sequence $(x_{k+1},v_{k+1}) = \Gammaminorc{\gamma}(x_k,v_k,z_{k+1},w_{k+1})$, with $(x_{i},v_{i}) = \Gammaminorc{\gamma}[i](x,v,\{z_j\}_{j=1}^{i},\{w_j\}_{j=1}^{i})$ for $i\in\iint{1}{k+1}$ and a given initial condition~$(x,v) \in \rset^{2d}$. 

\begin{lemma}
  \label{lem:recu_x_k_v_k_gen}
  Define, for $k \geq 0$, the vectors $\gminorbfc{k+1}_1 = (\gminorbfc{k+1}_{1,1},\dots,\gminorbfc{k+1}_{1,k+1})$ and $\gminorbfc{k+1}_2=(\gminorbfc{k+1}_{2,1},\dots,\gminorbfc{k+1}_{2,k+1})$ with components
\begin{equation}
  \label{eq:def_gminorbfc}
  \gminorbfc{k+1}_{1,i} = \frac{\gamma(1-\gminorc^{k-i+1})}{1-\gminorc}\eqsp,
  \qquad 
  \gminorbfc{k+1}_{2,i} = \gminorc^{k-i+1}\eqsp,
  \qquad
  i\in\iint{1}{k+1}\eqsp. 
 \end{equation}
Then,
\begin{equation}
\begin{aligned}
    x_{k+1} & =x+\gamma \frac{1-\gminorc^{k+1}}{1-\gminorc} v + \sum_{i=0}^{k-1} \gminorbfc{k+1}_{1,i+1}\parentheseDeux{\gamma\bminorg(x_i,\gamma^{\delta}v_i,\gamma^{\delta}z_{i+1},w_{i+1})+z_{i+1}}\\
    & \qquad\qquad\qquad\qquad\qquad\qquad+\gamma\sum_{i=0}^{k}f_\gamma\left(x_i,\gamma^{\delta}v_i,\gamma^{\delta}z_{i+1},w_{i+1}\right)+\gamma^\delta\sum_{i=0}^{k}\Dbf_{\gamma} z_{i+1} \eqsp,\\
    v_{k+1} & = \gminorc^{k+1}v+\sum_{i=0}^{k} \gminorbfc{k+1}_{2,i+1}\parentheseDeux{\gamma\bminorg\left(x_i,\gamma^{\delta}v_i,\gamma^{\delta}z_{i+1},w_{i+1}\right)+z_{i+1}} \eqsp .
  \end{aligned} 
\end{equation}
\end{lemma}

The proof is obtained by a simple induction and the equalities~$\gminorc \gminorbfc{k}_{2,j} = \gminorbfc{k+1}_{2,j}$ and $\gminorbfc{k}_{1,i+1} + \gamma \gminorbfc{k}_{2,i+1} = \gminorbfc{k+1}_{1,i+1}$ for $j \in\iint{0}{k}$ and $i \in\iint{0}{k-1}$. \Cref{lem:recu_x_k_v_k_gen} allows to rewrite~\eqref{eq:def_Gammaminorc_k+1} as:
\begin{equation}
  \label{eq:Xminor_Vminor_decomp_start_gen}
  \begin{pmatrix} \XminorTer_{k+1} \\ \VminorTer_{k+1} \end{pmatrix}
  = \Mminorc{\gamma}[k+1]\begin{pmatrix} \XminorTer_{0} \\ \VminorTer_{0} \end{pmatrix} + \gamma \sum_{i=0}^{k}  \Bminor^{(i)}\left(\XminorTer_i,\VminorTer_i, \sqrt{\gamma}\sigma_{\gamma} Z_{i+1},W_{i+1} \right) + \sigma_{\gamma} \left[\Gminorc{k+1} + \gamma^{\delta} \begin{pmatrix}\Gminorc{k+1}_3 \\ \Zd \end{pmatrix} \right],
\end{equation}
where the matrix which multiplies the initial condition reads
\begin{equation}
  \label{eq:def_matrix_Mminor_gen}
  \Mminorc{\gamma}[k+1] = \parentheseDeux{\begin{pmatrix}
    1 & \gamma(1-\gminorc^{k+1})/(1-\gminorc) \\
    0 & \gminorc^{k+1}
  \end{pmatrix}} \otimes \Idd \eqsp,
\end{equation}
the drift part of the dynamics is encoded by
\begin{equation}
  \label{eq:def_Bminor}
  \Bminor^{(i)}(x,v,z,w) = \begin{pmatrix} \gminorbfc{k+1}_{1,i+1}\\ \gminorbfc{k+1}_{2,i+1} \end{pmatrix} \otimes \bminorg\left(x,\gamma^\delta v,\gamma^{\delta}z,w\right) +\begin{pmatrix} f_{\gamma}\left(x,\gamma^\delta v,\gamma^{\delta}z,w\right)\\ \Zd \end{pmatrix} \eqsp, \quad i \in\{0,\ldots,k\} \eqsp,
\end{equation}
while the actual noise obtained at the end of the iterations is given by~$  \Gminorc{k+1}_3$ and 
\begin{equation}
  \label{eq:def_G_minor}
  \Gminorc{k+1} = \left(\Gminorc{k+1}_1,\Gminorc{k+1}_2\right) ,
\end{equation}
with 
\begin{equation}
  \label{eq:def_Gminor_gen}
  \Gminorc{k+1}_1 = \sqrt{\gamma}\sum_{i=0}^{k-1} \gminorbfc{k+1}_{1,i+1} Z_{i+1}\eqsp , 
  \qquad
  \Gminorc{k+1}_2 = \sqrt{\gamma}\sum_{i=0}^{k} \gminorbfc{k+1}_{2,i+1} Z_{i+1} \eqsp,
  \qquad  
  \Gminorc{k+1}_3 = \sqrt{\gamma}\Dbf_{\gamma} \sum_{i=0}^{k} Z_{i+1} \eqsp.
\end{equation}
 When~\Cref{ass:lip_minor_gen} is satisfied, it holds, for any $(x,v,z),(x',v',z') \in\rset^{3d}$, $w \in \msw$ and $i \in\iint{0}{k}$, 
\begin{equation}
  \label{eq:lip_Bminor}
  \norm{\Bminor_\gamma^{(i)}(x, v, z,w)-\Bminor_\gamma^{(i)}(x',v',z',w)}\leq (2+k\gamma)\LipT \left( \norm{x-x'} + \gamma^\delta \norm{v-v'}+\gamma^{\delta}\norm{z-z'} \right) \eqsp.
\end{equation}
This result easily follows from the bounds 
\begin{equation}
  \label{lem:estimate_gminorbfc}
  \sup_{i\in \{0,\ldots,k\}}\abs{\gminorbfc{k+1}_{1,i+1}} \leq k\gamma \eqsp ,\qquad \sup_{i\in \{0,\ldots,k\}}\abs{\gminorbfc{k+1}_{2,i+1}} \leq 1 \eqsp .
\end{equation}
These bounds are in turn a consequence of the fact that~$\tau_{\gamma} \in \ooint{0,1}$ and, for the first inequality, $\gminorbfc{k+1}_{1,i} = \gamma\sum_{j=0}^{k-1}\tau_{\gamma}^{j}$.

\subsection{Structure and properties of the noise}
\label{sec:noise_estimation}

We study in this section the structure of the Gaussian noise~\eqref{eq:def_G_minor} in~\eqref{eq:Xminor_Vminor_decomp_start_gen}. We present in \Cref{sec:ppties_continuous_noise} some estimates on the covariance matrix $\coeffContinueMatrix{}{t_0}$, for $t_0 >0$ small enough, of the continuous process~\eqref{eq:general_SDE} with $b \equiv 0$. Then, we relate in \Cref{sec:ppties_Gaussian_noise} the statistics of the Gaussian noise $\Gminorc{k+1}$ in~\eqref{eq:Xminor_Vminor_decomp_start_gen} to $\coeffContinueMatrix{}{t_0}$ for $k \approx t_0/\gamma$. We finally provide a decomposition into a final effective Gaussian increment and independent intermediate increments, by an orthogonal decomposition (see \Cref{sec:increment_decomposition}). More precisely, we explicitly write out a linear transformation~$\mathbf{A}$ such that $(\tZbfc{k+1},\Gminorc{k+1}) = \mathbf{A}\Zbfc{k}$ with~$\Gminorc{k+1}$ independent of~$\tZbfc{k+1}$ and~$\tZbfc{k+1}$ a Gaussian vector whose components are \iid~$d$-dimensional standard Gaussian random variables. 

\subsubsection{Properties of the noise covariance of the continuous process}
\label{sec:ppties_continuous_noise}

Recall the expression~\eqref{eq:def_continue_matrix} of the covariance matrix $\continueMatrix{t}$ of the Gaussian process corresponding to~\eqref{eq:general_SDE} when $b\equiv 0$. Our first technical result provides some bounds on the covariance matrix~$\continueMatrix{t_0}$ associated with the underlying reference Gaussian process over times~$t_0>0$.

\begin{lemma}
  \label{lem:estimate_coeff_Sigma}
  There exist $\btZ >0$ and positive constants $\underline{\varrho}_1,\underline{\varrho}_2,\underline{\varrho}_3,\overline{\varrho}_1,\overline{\varrho}_2,\overline{\varrho}_3$ with $\underline{\varrho}_1\underline{\varrho}_3-\bar{\varrho}_2^2>0$, such that, for any $t_0\in \ocintLigne{0,\btZ}$,
  \begin{equation}
    \label{eq:estimate_coeff_Sigma}
    \underline{\varrho}_1 \leq \frac{\coeffContinueMatrix{1}{t_0}}{\sigma^2 t_0^3} \leq \bar{\varrho}_1 \eqsp,
    \qquad
    \underline{\varrho}_2 \leq \frac{\coeffContinueMatrix{2}{t_0}}{\sigma^2 t_0^2} \leq \bar{\varrho}_2 \eqsp,
    \qquad
    \underline{\varrho}_3 \leq \frac{\coeffContinueMatrix{3}{t_0}}{\sigma^2 t_0} \leq \bar{\varrho}_3 \eqsp. 
  \end{equation}
\end{lemma}

The result is an immediate consequence of the limit $\lim_{t_0 \downarrow 0} \coeffContinueMatrix{1}{t_0} t_0^{-3} \sigma^2/3$ (which can be seen from~\eqref{eq:first_coeff_Sigma} by approximating the integrand in the integral appearing on the right hand side of the first equality), as well as the limits $\lim_{t_0 \downarrow 0} \coeffContinueMatrix{2}{t_0} t_0^{-2} = \sigma^2/2$ and $\lim_{t_0 \downarrow 0} \coeffContinueMatrix{3}{t_0} t_0^{-1} = \sigma^2$ which are obtained in a similar way. A detailed version of the proof is provided in the \supplementname~(see \Cref{sec:proof-crefl_coeef_sigma}).

\begin{lemma}
  \label{lem:matrix_minor_gen_continue}
  Assume that~\Cref{ass:gminorc} holds. Then there exist $\brho_c,\btZ>0$ such that the following inequality holds in the sense of $2 \times 2$ symmetric  matrices: for any $t_0\in\ocintLigne{0,\btZ}$,
  \begin{equation} 
    t_0\brho_c^{-1}\begin{pmatrix} t_0^2 &0 \\ 0 &  1 \end{pmatrix}
    \preceq \continueMatrix{t_0} \preceq
    t_0\brho_c\begin{pmatrix} t_0^2 &0\\0 &  1 \end{pmatrix} \eqsp .
  \end{equation}
\end{lemma}

\begin{proof}
  Introduce, for $t_0,\rho_c>0$,
  \begin{equation}
    \overline{\Bbf}^{(t_0,\rho_c)}=t_0\rho_c\begin{pmatrix}
    t_0^2 &0\\
    0 &  1\\
    \end{pmatrix}-\begin{pmatrix}
    \coeffContinueMatrix{1}{t_0} & \coeffContinueMatrix{2}{t_0}\\
    \coeffContinueMatrix{2}{t_0} & \coeffContinueMatrix{3}{t_0}\\
    \end{pmatrix} \eqsp ,
    \qquad
    \underline{\Bbf}^{(t_0,\rho_c)}=\begin{pmatrix}
    \coeffContinueMatrix{1}{t_0} & \coeffContinueMatrix{2}{t_0}\\
    \coeffContinueMatrix{2}{t_0} & \coeffContinueMatrix{3}{t_0}\\
    \end{pmatrix}-t_0\rho_c^{-1}\begin{pmatrix}
      t_0^2 &0\\
      0 &  1\\
    \end{pmatrix} \eqsp .
  \end{equation}
  The aim is to choose~$\bar{t}_0$ and~$\bar{\rho}_c$ such that $\overline{\Bbf}^{(t_0,\bar{\rho}_c)}$ and $\underline{\Bbf}^{(t_0,\bar{\rho}_c)}$ are both positive for any $t \in\ocint{0,\bar{t}_0}$. By Sylvester's criterion~\cite[Theorem 7.2.5]{horn:johnson:2012}, the result of the lemma is  implied by the following statement: there exist $\rho_c,\btZ>0$ such that, for any $t_0\in\ocintLigne{0,\btZ}$,
  \begin{equation}
    t_0\rho_c -\coeffContinueMatrix{3}{t_0}>0 \eqsp ,
    \qquad
    \coeffContinueMatrix{3}{t_0}-t_0\rho_c^{-1}>0 \eqsp ,
    \qquad
    \det\left(\overline{\Bbf}^{(t_0,\rho_c)}\right)>0 \eqsp ,
    \qquad
    \det\left(\underline{\Bbf}^{(t_0,\rho_c)}\right)>0 \eqsp .
  \end{equation}
  \sloppy The first two conditions are satisfied for the value of~$\btZ >0$ given by \Cref{lem:estimate_coeff_Sigma} and $\rho_c \geq \rho_c^{(1)} = \sigma^2 \max(\overline{\varrho}_3, \underline{\varrho}_3^{-1})$. Moreover, by \Cref{lem:estimate_coeff_Sigma}, there exist $\rho_c^{(2)} >0$ such that, for any $t_0\in \ocintLigne{0,\btZ}$,
  \begin{align}
    &\det\left(\overline{\Bbf}^{(t_0,\rho_c^{(2)})}\right) = \left(t_0^3\rho_c^{(2)}-\coeffContinueMatrix{1}{t_0}\right)\left(t_0\rho_c^{(2)}-\coeffContinueMatrix{3}{t_0}\right)-\left(\coeffContinueMatrix{2}{t_0}\right)^2\\
    & \geq \left(t_0^3\rho_c^{(2)}-\overline{\varrho}_1\sigma^2 t_0^3 \right)\left( t_0\rho_c^{(2)}-\overline{\varrho}_3\sigma^2 t_0 \right)-\overline{\varrho}_2^2\sigma^4 t_0^4 = t_0^4\parentheseDeux{\left(\rho_c^{(2)}-\overline{\varrho}_1\sigma^2\right)\left(\rho_c^{(2)}-\overline{\varrho}_3\sigma^2 \right)-\overline{\varrho}_2^2\sigma^4} >0 \eqsp.
  \end{align}
  Similar computations show that there exists $\rho_c^{(3)} >0$ for which $\det(\underline{\Bbf}^{(t_0,\rho_c^{(3)})})>0$ for any $t_0\in \ocintLigne{0,\btZ}$. This completes the proof upon setting $\brho_c=\max(\rho_c^{(1)},\rho_c^{(2)},\rho_c^{(3)})$. 
\end{proof}

\subsubsection{Properties of the final Gaussian noise~$\Gminorc{k+1}$}
\label{sec:ppties_Gaussian_noise}

We specify in this section the limit of the covariance of the noise~$\Gminorc{k+1}$ defined in \eqref{eq:def_G_minor} as~$\gamma \to 0$ and $k\gamma \approx t_0$. The following result, whose proof can be read in \supplementname-\Cref{sec:proof-lem:lemma_exp_gminorc}, gathers useful estimates which allow to compare the covariance of~$\Gminorc{k+1}$ to the covariance $\coeffContinueMatrix{}{t_0}$ of the underlying continuous process with~$b=0$ (see Lemma~\ref{lem:covariance_matrix_minorization_conditions_simple_gen} below).

\begin{lemma}
  \label{lem:lemma_exp_gminorc}
  Assume that~\Cref{ass:gminorc} holds. Then, for any $\gamma \in \ocint{0,\bgamma}$ and $\ell \geq 1$, $    \absLigne{\gminorc^\ell-\rme^{-\kappa\gamma \ell}}\leq C_{\kappa}\ell\gamma^2 $.
  Moreover,  for any  $\gamma \in \ocint{0,\bgamma}$,
  \begin{equation}
    \label{eq:lem:lemma_exp_gminorc_2}
|\gminorc - 1| \leq (\kappa+C_\kappa\gamma)\gamma \eqsp, \qquad    \left| \frac{\gamma}{1-\gminorc} - \frac1\kappa\right| \leq \left(\frac{2C_{\kappa}}{\kappa^2}+1\right) \gamma \eqsp .
  \end{equation}
\end{lemma}

\begin{lemma}
  \label{lem:covariance_matrix_minorization_conditions_simple_gen}
  The random variable~$\Gminorc{k+1}$ defined in~\eqref{eq:def_G_minor} is a $2d$-dimensional Gaussian random variable with mean zero and covariance matrix
  \begin{equation}
    \label{eq:def_C_k}
    \cminorc{k+1} = \begin{pmatrix}
      \cminorc{k+1}_1 &\cminorc{k+1}_2\\
      \cminorc{k+1}_2 &  \cminorc{k+1}_3\\
    \end{pmatrix} \otimes \Idd  \eqsp,
  \end{equation}
  with 
  \begin{equation}
    \label{eq:def_cminor_gen}
  \left\{ \begin{aligned}
    \cminorc{k+1}_1 & = \gamma\norm{\gminorbfc{k+1}_1}^2 = \frac{\gamma^2}{(1-\gminorc)^{2}} \frac{1}{1+\gminorc}\parentheseDeux{(1+\gminorc)k\gamma- \frac{\gminorc(2+\gminorc)-2\gminorc(1+\gminorc)\gminorc^k + \gminorc^{2(k+1)}}{(1-\gminorc)/\gamma}}, \\
  \label{eq:def_cminor_gen_2}
  \cminorc{k+1}_2 & = \gamma\ps{\gminorbfc{k+1}_1}{\gminorbfc{k+1}_2}= \frac{\gamma^2}{(1-\gminorc)^{2}} \frac{\gminorc}{1+\gminorc}\parentheseDeux{1-(1+\gminorc)\gminorc^k + \gminorc^{2k+1}},\\
  \cminorc{k+1}_3 & = \gamma\norm{\gminorbfc{k+1}_2}^2 = \frac{\gamma}{1-\gminorc} \frac{1}{1+\gminorc}\left( 1-\gminorc^{2(k+1)}\right) \eqsp ,
  \end{aligned} \right. 
\end{equation}
where $\gminorbfc{k+1}_1,\gminorbfc{k+1}_2$ are defined in~\eqref{eq:def_gminorbfc}.
  In addition, when~\Cref{ass:gminorc} holds, there exists for any $\btZ>0$ a constant~$C_{\btZ} \geq 0$ such that, for any $i \in \{1,2,3\}$ and $\gamma \in \ocint{0,\bgamma}$,
  \begin{equation}
    \label{eq:convergence_to_continue_matrix}
    \sup_{t_0\in \ocintLigne{0,\btZ}}\abs{ \cminorc{\lfloor t_0/\gamma \rfloor+1}_i- \frac{\coeffContinueMatrix{i}{t_0}}{\sigma^2}} \leq C_{\btZ} \gamma \eqsp.
  \end{equation}
\end{lemma}

\begin{proof}
  Note that since $(Z_i)_{i \in \{1,\ldots,k+1\}}$ are \iid~zero-mean Gaussian random variables and $\Gminorc{k+1}_1$ and $\Gminorc{k+1}_2$ are linear combination of elements of this family, $\Gminorc{k+1}$ is a $2d$-dimensional zero-mean Gaussian random variable with covariance matrix
  \begin{equation}
    \begin{pmatrix}
      \rmCov\left(\Gminorc{k+1}_1\right) &\rmCov\left(\Gminorc{k+1}_1,\Gminorc{k+1}_2\right)\\
      \rmCov\left(\Gminorc{k+1}_1,\Gminorc{k+1}_2\right) & \rmCov\left(\Gminorc{k+1}_2\right)
    \end{pmatrix} \eqsp.
  \end{equation}
  By \eqref{eq:def_Gminor_gen} and \eqref{eq:def_gminorbfc}, straightforward computations give
  \begin{align}
    \rmCov\left(\Gminorc{k+1}_1\right) &= \frac{\gamma^3}{(1-\gminorc)^{2}} \sum_{i=0}^{k-1}  \parenthese{1-\gminorc^{k-i}}^2 \Idd = \frac{\gamma^3}{(1-\gminorc)^{2}}\left[k-2\gminorc\frac{1-\gminorc^{k}}{1-\gminorc}+ \gminorc^2\frac{1-\gminorc^{2k}}{1-\gminorc^2}\right] \Idd \eqsp,\\
    \rmCov\left(\Gminorc{k+1}_1,\Gminorc{k+1}_2\right) &= \frac{\gamma^2}{1-\gminorc}\sum_{i=0}^{k-1} \left[ \parenthese{1-\gminorc^{k-i}} \gminorc^{k-i}\right]\Idd = \frac{\gamma^2}{1-\gminorc}\left[\gminorc\frac{1-\gminorc^{k}}{1-\gminorc}- \gminorc^2\frac{1-\gminorc^{2k}}{1-\gminorc^2}\right]\Idd\eqsp,\\
    \rmCov\left(\Gminorc{k+1}_2\right) &= \gamma \sum_{i=0}^{k}  \gminorc^{2(k-i)} \Idd = \frac{\gamma}{1-\gminorc^2}\left( 1-\gminorc^{2(k+1)} \right) \Idd \eqsp,
  \end{align}
  from which we obtain the expressions of the coefficients~$\cminorc{k+1}_i$ for $1 \leq i \leq 3$ after some reorganization allowing to compare the resulting expressions more easily with~\eqref{eq:first_coeff_Sigma}-\eqref{eq:third_coeff_Sigma} (upon replacing~$\gamma$ by~$k\gamma$ in the latter equations).

  We finally show \eqref{eq:convergence_to_continue_matrix}. By \Cref{lem:lemma_exp_gminorc}, we have, for any~$\btZ>0$ and~$\gamma\in\ocint{0,\bgamma}$, 
  \begin{align}
    \begin{aligned}
      \label{eq:calc_limit_exp_gminor}
      \sup_{t_0\in \ocintLigne{0,\btZ}}\abs{\gminorc^{\lfloor t_0/\gamma \rfloor+1}-\rme^{-\kappa t_0}}&\leq C_{\kappa} (\btZ+1) \gamma \eqsp ,
      \qquad
      \sup_{t_0\in \ocintLigne{0,\btZ}}\abs{\gminorc^{2\lfloor t_0/\gamma \rfloor+1}-\rme^{-2\kappa t_0}}&\leq 2C_{\kappa}(\btZ+1) \gamma \eqsp .
    \end{aligned}
  \end{align}
  The inequality~\eqref{eq:convergence_to_continue_matrix} is then a simple consequence of the definition~\eqref{eq:def_continue_matrix} of $\continueMatrix{t_0}$,\eqref{eq:def_cminor_gen} and the estimates~\eqref{eq:lem:lemma_exp_gminorc_2} in \Cref{lem:lemma_exp_gminorc}.
\end{proof}

\subsubsection{Decomposition into final noise and independent increments}
\label{sec:increment_decomposition}

After analyzing the covariance of $\Gminorc{k+1}$ defined by  \eqref{eq:def_G_minor} in the limit $\gamma \to 0$ and $k\gamma \approx t_0$, we aim at providing the conditional distribution of~$(X_{k+1},V_{k+1})$ given~$\Gminorc{k+1}$. This will be a crucial step to use a perturbation argument similar to the one presented for overdamped Langevin dynamics in \Cref{sec:idea_ULA}. To this end, we introduce in this section the linear transformation $\tZbfc{k+1}$ of $\Zbfc{k}$ which plays the same role as \eqref{def_tZc_ULA}, \ie, $\tZbfc{k+1}$ is independent of~$\Gminorc{k+1}$ and its components are \iid~standard Gaussian random variables. For underdamped Langevin dynamics, we introduce for $k \geq 2$ and $\gamma \in\ocint{0,\bgamma}$ the vectors~$\alphaminorbf(k+1,\gamma) = (\alphaminor_1(k+1,\gamma),\dots,\alphaminor_{k+1}(k+1,\gamma))$ and $\betaminorbf(k+1,\gamma) = (\betaminor_1(k+1,\gamma),\dots,\betaminor_{k+1}(k+1,\gamma))$, and define, for any $i \in \{1,\dots,k-1\}$,
\begin{equation}
  \label{eq:def_tZ_gen}
  \tZc{k+1}_i  =  Z_i -  \sqrt{\gamma} \beta_i(k,\gamma)\Gminorc{k+1}_2 - \sqrt{\gamma} \alpha_i(k,\gamma)\Gminorc{k+1}_1.
\end{equation}
The components~$k$ and~$k+1$ of the vectors~$\alphaminorbf(k+1,\gamma),\betaminorbf(k+1,\gamma)$ are not needed at this stage, but they will turn out to be useful later on. As specified below in \Cref{lem:minor_alpha_beta_gen}, the scaling factor~$\sqrt{\gamma}$ in front of the coefficients~$\alpha_i(k,\gamma),\beta_i(k,\gamma)$ ensures that the latter coefficients are of order~1 at most as $\gamma \to 0$ and $k\gamma \approx t_0$.

\begin{remark}
  In order to further motivate the chosen scalings in~\eqref{eq:def_tZ_gen}, note that the linear transform~\eqref{def_tZc_ULA} for overdamped Langevin dynamics can be rewritten as
  \[
  \tZc{k+1}_i = Z_i - \sqrt{\gamma} \frac{G^{(k+1)}}{(k+1)\gamma} \eqsp, \qquad i \in \iint{1}{k} \eqsp.
  \]
  The counterpart of the coefficients~$\alphaminor_i(k+1,\gamma),\betaminor_i(k+1,\gamma)$ in this context is simply~$1/[(k+1)\gamma]$, which is indeed of order~1 when~$k\gamma$ is of order~1.
\end{remark}

The choice of $\alphaminorbf(k+1,\gamma), \betaminorbf(k+1,\gamma)$ corresponds to the orthogonal projection of the Gaussian variables~$\{Z_{i+1}\}_{i=0}^{k-2}$ onto the the orthogonal of the vector space spanned by~$\gminorbfc{k+1}_1$ and~$\gminorbfc{k+1}_2$. This will ensure that the covariance of these projected vectors and~$\Gminorc{k+1}$ vanishes, and hence that the projected vectors $\tZbfc{k+1} = (\tZc{k+1}_1,\dots,\tZc{k+1}_{k-1})$ are independent of~$\Gminorc{k+1}$.

We start by constructing the orthogonal projector onto $\mathrm{Span}(\gminorbfc{k+1}_1,\gminorbfc{k+1}_2)$. We need to restrict the discussion to iteration indices and time steps in the set
\begin{equation}
  \label{eq:def_msec}
  \msec=\defEns{(k,\gamma)\in \nset^*\times (0,+\infty) \, \middle| \, \cminorc{k+1}_1\cminorc{k+1}_3-\left(\cminorc{k+1}_2\right)^2 \neq 0}\eqsp ,
\end{equation}
where the coefficients $(\cminorc{k+1}_i)_{1 \leq i \leq 3}$ are defined in \Cref{lem:covariance_matrix_minorization_conditions_simple_gen}. Note that the condition to be satisfied in~\eqref{eq:def_msec} is in fact that the determinant of the matrix~$\cminorc{k+1}$ in~\eqref{eq:def_C_k} is positive (it is always nonnegative). This condition is not restrictive since, by \Cref{lem:estimate_coeff_Sigma} and \eqref{eq:convergence_to_continue_matrix}, there exists $\btZ>0$ such that, for any $t_0\in\ocint{0,\btZ}$, there is $\bar{\gamma}_{t_0}>0$ for which $(\lfloor t_0/\gamma\rfloor,\gamma)\in \msec$ for any $\gamma\in \ocint{0,\bar{\gamma}_{t_0}}$.

\begin{lemma}
  \label{lem:minor_orthogonal_projection_gen}
  For any $(k,\gamma)\in \msec$, and $j\in \{1,..,k+1\}$, consider $\betaminor_j(k+1,\gamma), \alphaminor_j(k+1,\gamma) \in \rset$ the unique solution of 
  \begin{equation}
    \label{eq:defbetaminor_alpha_minor_j_gen}
    \begin{pmatrix}
      \cminorc{k+1}_1  &\cminorc{k+1}_2\\
      \cminorc{k+1}_2 &  \cminorc{k+1}_3
    \end{pmatrix}
    \begin{pmatrix}
      \alphaminor_j(k+1,\gamma) \\
      \betaminor_j(k+1,\gamma) 
    \end{pmatrix} = \begin{pmatrix}
      \gminorbfc{k+1}_{1,j}\\
      \gminorbfc{k+1}_{2,j}
    \end{pmatrix} \eqsp,
  \end{equation}
  and define 
  \begin{equation}
    \label{eq:def_Pminor_gen}
    \Pminorc{k} = \operatorname{I}_{k+1} - \gamma \betaminorbf(k+1,\gamma) \left[\gminorbfc{k+1}_2\right]^{\transpose} - \gamma \alphaminorbf(k+1,\gamma) \left[\gminorbfc{k+1}_1 \right]^{\transpose} \in \rset^{(k+1) \times (k+1)}\eqsp. 
  \end{equation}
  Then, $\Pminorc{k}$ is the orthogonal projection onto $\rmSpan\left(\gminorbfc{k+1}_1,\gminorbfc{k+1}_2\right)^{\perp}$.
\end{lemma}

\begin{proof}
  We compute the action of~$\Pminorc{k}$ on the vector space generated by~$\gminorbfc{k+1}_1$ and~$\gminorbfc{k+1}_2$ and its orthogonal.
  Note first that the definition~\eqref{eq:def_Pminor_gen} ensures that $\Pminorc{k}w = w$ for all $w \in \rmSpan\left(\gminorbfc{k+1}_1,\gminorbfc{k+1}_2\right)^{\perp}$. Moreover, in view of \Cref{lem:covariance_matrix_minorization_conditions_simple_gen} and by the definition~\eqref{eq:defbetaminor_alpha_minor_j_gen} of the coefficients of~$\alphaminorbf(k+1,\gamma),\betaminorbf(k+1,\gamma)$, 
\begin{equation}
\gamma \norm[2]{\gminorbfc{k+1}_1}\alphaminorbf + \gamma \ps{\gminorbfc{k+1}_1}{\gminorbfc{k+1}_2}\betaminorbf=\gminorbfc{k+1}_1\eqsp , \qquad \gamma\ps{\gminorbfc{k+1}_1}{\gminorbfc{k+1}_2}\alphaminorbf+\gamma\norm[2]{\gminorbfc{k+1}_2}\betaminorbf=\gminorbfc{k+1}_2 \eqsp ,
\end{equation}
which implies that $\Pminorc{k}\gminorbfc{k+1}_1 = \Pminorc{k}\gminorbfc{k+1}_2 = 0$, and finally $\Pminorc{k}w = 0$ for~$w \in \rmSpan\left(\gminorbfc{k+1}_1,\gminorbfc{k+1}_2\right)$.
\end{proof}

We are now in position to specify the law of~$\tZbfc{k+1} = \left(\tZc{k+1}_1,\dots,\tZc{k+1}_{k-1}\right)$.

\begin{lemma}
  \label{lem:indep_z_tilde_gen}
 Let~$(k,\gamma) \in \msec$. Then, $\tZbfc{k+1}$ defined by~\eqref{eq:def_tZ_gen} is a $d \times (k-1)$-dimensional zero-mean Gaussian random variable with positive definite covariance matrix $\vartZc{k} \left[\vartZc{k}\right]^{\transpose}$, where
  \begin{equation}
    \label{eq:def_vartZc_gen}
    \vartZc{k} =  \parentheseDeux{ \begin{pmatrix}
        \operatorname{I}_{k-1} & \bfZero_{k-1,2} 
      \end{pmatrix} \Pminorc{k}} \otimes \Idd \in \rset^{(k-1)d \times (k+1)d}\eqsp.
  \end{equation}
In addition, $\tZbfc{k+1}$ is independent of $\Gminorc{k+1}$ defined in~\eqref{eq:def_Gminor_gen}.
\end{lemma}

\begin{proof}
  It is easy to see that $\widetilde{\mathbf{G}}^{(k)} = (\tZbfc{k+1},\Gminorc{k+1}) = \mathbf{A}\Zbfc{k}$ with
  \begin{equation}
    \mathbf{A} =  \parentheseDeux{ \begin{pmatrix}
      \begin{matrix}
        \operatorname{I}_{k-1} & \bfZero_{k-1,2} 
      \end{matrix} \\
            \begin{matrix}
        \bfZero_{2,k-1}& \bfZero_{2,2} 
      \end{matrix}
    \end{pmatrix}
\Pminorc{k} + \sqrt{\gamma}
\begin{pmatrix}
  \bfZero_{k-1,k+1}    \\
  \left[\gminorbfc{k+1}_1\right]^{\transpose}\\
    \left[\gminorbfc{k+1}_2\right]^{\transpose}
\end{pmatrix}
}   \otimes \Idd  \eqsp.
  \end{equation}
  The matrix $\mathbf{A}$ is invertible by Gaussian elimination in view of the definition~\eqref{eq:def_Pminor_gen} of~$\Pminorc{k}$ and the ones \eqref{eq:def_gminorbfc} of $\gminorbfc{k+1}_1,\gminorbfc{k+1}_2$, and 
  \begin{equation}
  \begin{aligned}
    \det(\mathbf{A})^{1/d} & =  \det\parentheseDeux{     \begin{pmatrix}
        \begin{matrix}
          \operatorname{I}_{k-1} & \bfZero_{k-1,2} 
      \end{matrix} \\
        \begin{matrix}
          \bfZero_{2,k-1}& \bfZero_{2,2} 
        \end{matrix}
      \end{pmatrix}\Pminorc{k} + \sqrt{\gamma}
      \begin{pmatrix}
        \bfZero_{k-1,k+1}    \\
               \left[\gminorbfc{k+1}_1\right]^{\transpose}\\
               \left[\gminorbfc{k+1}_2\right]^{\transpose}
      \end{pmatrix}
    } = \det     \begin{pmatrix}
      \begin{matrix}
        \operatorname{I}_{k-1} & \bfZero_{k-1,2} 
      \end{matrix} \\
      \sqrt{\gamma}\left[\gminorbfc{k+1}_1\right]^{\transpose}\\
      \sqrt{\gamma}\left[\gminorbfc{k+1}_2\right]^{\transpose}
    \end{pmatrix} \\
    & = \gamma \det\begin{pmatrix}\gminorbfc{k+1}_{1,k} & \gminorbfc{k+1}_{1,k+1} \\ \gminorbfc{k+1}_{2,k} & \gminorbfc{k+1}_{2,k+1} \end{pmatrix}  >0 \eqsp. 
  \end{aligned}
  \end{equation}
  Therefore, $\widetilde{\mathbf{G}}^{(k)}$ is a zero-mean Gaussian random variable with positive definite covariance matrix~$\mathbf{A}\mathbf{A}^{\transpose}$. As a result, $\tZbfc{k+1}$ defined by \eqref{eq:def_tZ_gen} is a $d \times (k-1)$-dimensional zero-mean Gaussian random variable with positive definite covariance matrix $\vartZc{k} [\vartZc{k}]^{\transpose}$ with $\vartZc{k}$ given by~\eqref{eq:def_vartZc_gen}. Finally, since $\Pminorc{k}\gminorbfc{k+1}_1 = \Pminorc{k}\gminorbfc{k+1}_2 = 0$ in view of~\eqref{eq:def_Pminor_gen}, it follows that for $i \in\{1,\ldots,k-1\}$,  $  \rmCov(\tZ_i^{(k)},\Gminor_1^{(k)})=0$, $ \rmCov(\tZ_i^{(k)},\Gminor_2^{(k)}) = 0$
  and so $\tZbfc{k+1}$ is independent of~$\Gminorc{k+1}$ because all random variables at hand are Gaussian. 
\end{proof}

We conclude this section by providing some bounds of the coefficients of~$\alphaminorbf(k+1,\gamma),\betaminorbf(k+1,\gamma)$ defined in~\eqref{eq:defbetaminor_alpha_minor_j_gen}. We denote by $\Vert\xi\Vert_\infty = \max(|\xi_1|,\dots,|\xi_{k}|)$ the $\ell^\infty$ norm of a vector~$\xi=(\xi_1,\dots,\xi_{k}) \in \rset^{k}$.

\begin{lemma}
  \label{lem:minor_alpha_beta_gen}
  Assume that \Cref{ass:gminorc} holds. Then, there exists $\tbound>0$ and~$K \in \mathbb{R}_+$ such that, for any $t_0\in \ocint{0,\tbound}$,
  \begin{equation}
  \limsup_{\gamma\downarrow 0} \left\Vert\alphaminorbf(\lfloor t_0/\gamma\rfloor +1 ,\gamma)\right\Vert_\infty \leq  K/t_0^2 \eqsp, 
  \qquad
  \limsup_{\gamma\downarrow 0} \left\Vert \betaminorbf(\lfloor t_0/\gamma\rfloor+1,\gamma)\right\Vert_\infty \leq K/t_0 \eqsp.
  \end{equation}
\end{lemma}

\begin{proof}
  We consider~$\tbound>0$ as given by Lemmas~\ref{lem:estimate_coeff_Sigma} and~\ref{lem:matrix_minor_gen_continue}, and such that, for any $t_0\in\ocint{0,\btZ}$, there is~$\bar{\gamma}_{t_0}$ for which~$(\lfloor t_0/\gamma\rfloor,\gamma)\in \msec$ for any $\gamma\in \ocint{0,\bar{\gamma}_{t_0}}$ (see the discussion before \Cref{lem:minor_orthogonal_projection_gen}). For such values of~$\tbound$ and~$\gamma$, the system~\eqref{eq:defbetaminor_alpha_minor_j_gen} for $k_0 = \lfloor t_0/\gamma\rfloor$ can be explicitly solved as
  \begin{equation}
  \alphaminor_j(k_0+1,\gamma) = \frac{\cminorc{k_0+1}_3 \gminorbfc{k_0+1}_{1,j}- \cminorc{k_0+1}_2 \gminorbfc{k_0+1}_{2,j}}{\cminorc{k_0+1}_1 \cminorc{k_0+1}_3-\left(\cminorc{k_0+1}_2\right)^2} \eqsp,
  \qquad
  \betaminor_j(k_0+1,\gamma) = \frac{\cminorc{k_0+1}_1 \gminorbfc{k_0+1}_{2,j}- \cminorc{k_0+1}_2 \gminorbfc{k_0+1}_{1,j}}{\cminorc{k_0+1}_1 \cminorc{k_0+1}_3-\left(\cminorc{k_0+1}_2\right)^2} \eqsp. 
  \end{equation}
  By~\eqref{lem:estimate_gminorbfc} and \Cref{lem:covariance_matrix_minorization_conditions_simple_gen}, it follows that, for any~$t_0\in\ocint{0,\btZ}$,
  \begin{equation}
  \limsup_{\gamma\downarrow 0} \abs{\alphaminor_j(k_0+1,\gamma)} \leq \limsup_{\gamma\downarrow 0} \frac{t_0\abs{\cminorc{k_0+1}_3} + \abs{\cminorc{k_0+1}_2}}{\abs{\cminorc{k_0+1}_1 \cminorc{k_0+1}_3-\left(\cminorc{k_0+1}_2\right)^2}} = \frac{\sigma^2\left(t_0\coeffContinueMatrix{3}{t_0}+\coeffContinueMatrix{2}{t_0}\right)}{\abs{\coeffContinueMatrix{1}{t_0}\coeffContinueMatrix{3}{t_0}-\parenthese{\coeffContinueMatrix{2}{t_0}}^2}} \eqsp ,
  \end{equation}
and 
\begin{equation}
\limsup_{\gamma\downarrow 0} \abs{\betaminor_j(k_0+1,\gamma)} \leq \limsup_{\gamma\downarrow 0} \frac{\abs{\cminorc{k_0+1}_1} +t_0 \abs{\cminorc{k_0+1}_2}}{\abs{\cminorc{k_0+1}_1 \cminorc{k_0+1}_3-\left(\cminorc{k_0+1}_2\right)^2}} = \frac{\sigma^2\left(\coeffContinueMatrix{1}{t_0}+t_0\coeffContinueMatrix{2}{t_0}\right)}{\abs{\coeffContinueMatrix{1}{t_0}\coeffContinueMatrix{3}{t_0}-\parenthese{\coeffContinueMatrix{2}{t_0}}^2}} \eqsp .
\end{equation}
The numerators of the last term in the two previous equalities are upper bounded by \Cref{lem:estimate_coeff_Sigma} as $t_0\coeffContinueMatrix{3}{t_0}+\coeffContinueMatrix{2}{t_0}\leq (\bar{\varrho}_2+\bar{\varrho}_3)\sigma^2 t_0^2$ and $\coeffContinueMatrix{1}{t_0}+t_0\coeffContinueMatrix{2}{t_0}\leq (\bar{\varrho}_1+\bar{\varrho}_2)\sigma^2 t_0^3$; while the denominator is lower bounded with the Minkowski determinant theorem (see \eg~\cite[Section 4.1.8]{marcus:minc:1992}) and \Cref{lem:matrix_minor_gen_continue} as
  \begin{align}
    \abs{\coeffContinueMatrix{1}{t_0}\coeffContinueMatrix{3}{t_0}-\parenthese{\coeffContinueMatrix{2}{t_0}}^2} \geq  \det\parentheseDeux{ \frac{t_0}{\brho_c}\begin{pmatrix} t_0^2 &0\\ 0 &  1\\ \end{pmatrix}} = \frac{t_0^4}{\brho_c^{2}} \eqsp.
  \end{align}
  This finally gives the claimed estimates.
\end{proof}

\subsection{Stability estimates}
\label{sec:stability}

We provide in this section estimates on the mapping $\Gammaminorc{\gamma}[k+1]$ defined in~\eqref{eq:def_Phiminor_k} providing the $(k+1)$-th iterates~$(\XminorTer_{k+1},\VminorTer_{k+1})$ of the Markov chain defined in~\eqref{eq:def_xminor_v_minor_gen} as a function of the initial condition $(x,v)$ and the realizations $\{z_{i+1},w_{i+1}\}_{i=0}^{k}$. We first establish in \Cref{sec:stability_all_noise} that this function is Lipschitz continuous with respect to the initial condition $(x,v)$ and the realizations $\{z_{i+1}\}_{i=0}^{k}$, and provide upper bounds on its Lipschitz constant. Pursuing on the same argument as \Cref{sec:idea_ULA} (see \eqref{eq:ula_G_2_k}), we then show in \Cref{sec:stability_conditioned} that $(X_{k+1},V_{k+1})$ can be written as
\[
(X_{k+1},V_{k+1}) = \bGammaminorc{\gamma,x,v,\{\tZc{k+1}_j\}_{j=1}^{k-1},\{W_j\}_{j=1}^{k+1}}[k+1]\left(\Gminorc{k+1}_1, \Gminorc{k+1}_2\right) \eqsp,
\]
with~$(\Gminorc{k+1}_1, \Gminorc{k+1}_2)$ and $\{\tZc{k+1}_j\}_{j=1}^{k-1}$ respectively defined in~\eqref{eq:def_Gminor_gen} and~\eqref{eq:def_tZ_gen}. We establish Lipschitz properties of the function~$(g_1,g_2) \mapsto \bGammaminorc{\gamma,x,v,\{\widetilde{z}_j\}_{j=1}^{k-1},\{w_j\}_{j=1}^{k+1}}[k+1](g_1,g_2)$. Then, we prove in \Cref{sec:stability_diffeo} that this function is in fact a $\rmC^1$-diffeomorphism. We conclude this section by studying some properties of its inverse seen as a function of~$(g_1,g_2)$ and the initial conditions~$(x,v)$. 
  
\subsubsection{Stability with respect to initial conditions and all noise increments}
\label{sec:stability_all_noise}

The following result provides some Lipschitz bounds for the functions~$\Gammaminorc{\gamma}[k+1]$ defined in~\eqref{eq:def_Phiminor_k}. They are stated in terms of a norm on~$\rset^{2d}$ parametrized by some positive parameter~$\lambda>0$, namely~$\|x\|+\lambda\|v\|$. The proof of Proposition~\ref{propo:inverse_lip_Gamma} below will require~$\lambda$ to be chosen sufficiently small. 

\begin{lemma}
  \label{lem:Lip_Phi_gen}
  Assume that \Cref{ass:lip_minor_gen} holds and $\sup_{\gamma \in\ocint{0,\bgamma}} \tau_{\gamma} \leq 1$ and fix~$\lambda > 0$, a time step~$\gamma > 0$, a maximal number of iterations~$\Nminor \in \nsets$, and realizations $\{w_j\}_{j=1}^{\Nminor+1} \subset \rset^{m}$ of the additional noise. For two initial conditions $(x,v),(x',v') \in \rset^{2d}$ and realizations of the noise $\{z_j\}_{j=1}^{\Nminor+1},\{z_j'\}_{j=1}^{\Nminor+1} \subset \rset^{d}$, define for any $k \in \{0,\ldots,\Nminor-1\}$ the iterates
  \begin{equation}
    \label{eq:lem:Lip_Phi_gen}
    \begin{aligned}
      (x_{k+1},v_{k+1}) & =  \Gammaminorc{\gamma}[k+1]\left(x,v,\left\{z_j\right\}_{j=1}^{k+1},\{w_j\}_{j=1}^{k+1}\right) \eqsp, \qquad       (x'_{k+1},v'_{k+1}) &= \Gammaminorc{\gamma}[k+1]\left(x',v',\left\{z'_j\right\}_{j=1}^{k+1},\{w_j\}_{j=1}^{k+1}\right) \eqsp.
    \end{aligned}
  \end{equation}
  Then, for any $k \in \{0,\ldots,\Nminor\}$, 
  \begin{equation}
    \label{eq:first_stab_estimate}
    \norm{x_{k}-x'_{k}}+\lambda\norm{v_{k}-v'_{k}}\leq \LipGammaminor(\lambda)^{k}\Big( \norm{x-x'}+\lambda\norm{v-v'} \Big) +\LipMGammaminor(\lambda)\sum_{i=1}^{k} \LipGammaminor(\lambda)^{k-i} \norm{z_i-z'_i}  \eqsp,
  \end{equation}
  where
  \begin{equation}
    \label{eq:def_LipGammaminor}
    \LipGammaminor(\lambda) = 1 + \gamma \left[ \frac1\lambda + (1+\lambda) \max\left(1,\frac{\gamma^\delta}{\lambda}\right)\Ltt\right] \eqsp,  
  \qquad 
  \LipMGammaminor(\lambda)= \lambda+\gamma^\delta\mathscr{D}+\gamma^{1+\delta} \parenthese{1+\lambda}\LipT\eqsp. 
  \end{equation}
  Moreover, when the initial conditions coincide (namely when~$(x,v) = (x',v')$),
  \begin{equation}
    \label{eq:second_stab_estimate_x}
    \sum_{i=1}^{k} \norm{x_i-x_i'} \leq \gamma^{\delta} (\mathscr{D}+\gamma\LipT)\norm{z_{k}-z'_{k}} + \{k \LipMGammaminor(\lambda) \LipGammaminor(\lambda)^{k} +\mathscr{L}^x_{k,\gamma,\lambda}\}\sum_{i=1}^{k-1} \norm{z_{i}-z_{i}'} \eqsp,
  \end{equation}
  and
  \begin{equation}
    \label{eq:second_stab_estimate_v}
    \sum_{i=1}^{k} \norm{v_i-v_i'} \leq \left(1+\gamma^{1+\delta}\LipT\right)^k \norm{z_k-z_k'} + k \mathscr{L}^v_{k,\gamma,\lambda} \sum_{i=1}^{k-1} \norm{z_{i}-z_{i}'}\eqsp,
  \end{equation}
  with
  \begin{equation}
    \label{eq:def_mathscr_L}
    \begin{aligned}
      \mathscr{L}^x_{k,\gamma,\lambda} & = \gamma\LipT  \LipMGammaminor(\lambda) \LipGammaminor(\lambda)^{k}   + \gamma\left(1+\gamma^\delta\LipT\right)\left(k\gamma \LipMGammaminor(\lambda)\LipT \left(1+\gamma^{1+\delta}\LipT\right)^k \LipGammaminor\left(\lambda\right)^{k} +  \left(1+\gamma^{1+\delta}\LipT\right)^k \right), \\
      \mathscr{L}^v_{k,\gamma,\lambda} & = (k-1) \gamma \LipMGammaminor(\lambda) \left(1+\gamma^{1+\delta}\LipT\right)^k \LipGammaminor(\lambda)^{k}\LipT + \left(1+\gamma^{1+\delta}\LipT\right)^k.
    \end{aligned}
  \end{equation}
\end{lemma}

For reasons that will appear more clearly in \Cref{propo:Lip_psi_i_k_lip_xi_k_gen} and its proof,  the difference $\norm{z_k-z_k'}$ is isolated on purpose in~\eqref{eq:second_stab_estimate_x}-\eqref{eq:second_stab_estimate_v} and appears with a prefactor at most~1 instead of~$k$ in contrast to the other differences~$\norm{z_i-z_i'}$ for~$i \in\{1,\ldots,k-1\}$.

\begin{proof}
  Since $(x_{k+1},v_{k+1}) = \Gammaminorc{\gamma}(x_k,v_k,z_{k+1},w_{k+1})$ for $k \in\iint{0}{N-1}$, we obtain by \Cref{ass:lip_minor_gen} and \eqref{minor_def_Gamma_b_gamma} that
  \begin{align}
    \norm{x_{k+1}-x'_{k+1}} & \leq \norm{x_{k}-x'_{k}}+\gamma\norm{v_{k}-v'_{k}} + \gamma^\delta\norm{\Dbf_{\gamma} (z_{k+1}- z'_{k+1})} \notag \\
    & \quad + \gamma\norm{f_\gamma\left(x_k,\gamma^{\delta} v_k,\gamma^{\delta}z_{k+1},w_{k+1}\right)-f_\gamma\left(x'_k,\gamma^{\delta} v'_k,\gamma^{\delta}z'_{k+1},w_{k+1}\right)} \notag \\
    & \leq (1+\gamma \LipT) \norm{x_{k}-x'_{k}} + \gamma\left(1+\gamma^\delta\LipT\right)\norm{v_k-v'_k} + \gamma^{\delta} (\mathscr{D}+\gamma\LipT)\norm{z_{k+1}-z'_{k+1}} \eqsp, \label{eq:one_step_x_estimate}
  \end{align}
  and, since $\tau_{\gamma} \leq 1$,
  \begin{align}
    \norm{v_{k+1}-v'_{k+1}} & \leq \gminorc\norm{v_k-v'_k}+\norm{z_{k+1}-z'_{k+1}} \\
    & \quad + \gamma\norm{g_\gamma\left(x_k,\gamma^{\delta} v_k,\gamma^{\delta}z_{k+1},w_{k+1}\right)-g_\gamma\left(x'_k, \gamma^{\delta} v'_k,\gamma^{\delta}z_{k+1}',w_{k+1}\right)} \\
    & \leq \gamma\LipT \norm{x_{k}-x'_{k}} + (1+\gamma^{1+\delta}\LipT)\norm{v_k-v'_k} + (1+\gamma^{1+\delta}\LipT)\norm{z_{k+1}-z'_{k+1}}\eqsp. \label{eq:one_step_v_estimate}
  \end{align}
  Therefore,
  \begin{align}
    & \norm{x_{k+1}-x'_{k+1}} +\lambda\norm{v_{k+1}-v'_{k+1}} \\
    & \leq \left(1+\gamma (1+\lambda)\Ltt\right)\norm{x_{k}-x'_{k}}+\left(\lambda\left[1+\gamma^{1+\delta} \Ltt\right] + \gamma \left[1+ \gamma^{\delta}\Ltt\right] \right)\norm{v_k-v'_k} + \LipMGammaminor(\lambda)\norm{z_{k+1}-z'_{k+1}} \\
    & \leq \left[ 1 + \gamma \max\left((1+\lambda)\Ltt, \frac1\lambda + \left(1+\frac1\lambda\right)\gamma^{\delta} \Ltt\right)\right] \left(\norm{x_{k}-x'_{k}}+\lambda\norm{v_{k}-v'_{k}}\right) + \LipMGammaminor(\lambda)\norm{z_{k+1}-z'_{k+1}}\eqsp.
  \end{align}
  The bound~\eqref{eq:first_stab_estimate} then follows from the inequality $  1 + \gamma \max\left((1+\lambda)\Ltt, \lambda^{-1} + \left(1+\lambda^{-1}\right)\gamma^{\delta} \Ltt\right) \leq \LipGammaminor(\lambda)$,
  and an easy induction on~$k$. 

We now prove~\eqref{eq:second_stab_estimate_x}. For the sum up to indices~$k-1$, we use~\eqref{eq:first_stab_estimate} and $x_0=x_0',v_0=v_0'$ to write
\begin{equation}
  \label{eq:bound_sum_x_until_k-1}
\begin{aligned}
  \sum_{i=1}^{k-1} \norm{x_i-x_i'}  & \leq \sum_{i=1}^{k-1}\{ \norm{x_i-x_i'} + \lambda \norm{v_i-v_i'}\} \\
  & \leq \LipMGammaminor(\lambda) \sum_{i=1}^{k-1} \sum_{j=1}^{i} \LipGammaminor\left(\lambda\right)^{i-j} \norm{z_j-z'_j} \leq \LipMGammaminor(\lambda) (k-1) \LipGammaminor(\lambda)^{k}\sum_{i=1}^{k-1} \norm{z_i-z'_i} \eqsp.
\end{aligned}
\end{equation}
The last term in the sum is bounded with~\eqref{eq:one_step_x_estimate} and~\eqref{eq:first_stab_estimate} as
\begin{equation}
  \label{eq:bound_sum_x_term_k}
\begin{aligned}
& \norm{x_k-x_k'} \leq (1+\gamma \LipT) \norm{x_{k-1}-x'_{k-1}} + \gamma\left(1+\gamma^\delta\LipT\right)\norm{v_{k-1}-v'_{k-1}} + \gamma^{\delta} (\mathscr{D}+\gamma\LipT)\norm{z_{k}-z'_{k}} \\
& \quad \leq \gamma\left(1+\gamma^\delta\LipT\right)\norm{v_{k-1}-v'_{k-1}} + \gamma^{\delta} (\mathscr{D}+\gamma\LipT)\norm{z_{k}-z'_{k}} + (1+\gamma \LipT)\LipMGammaminor(\lambda) \LipGammaminor\left(\lambda\right)^{k} \sum_{i=1}^{k-1} \norm{z_i-z'_i}.
\end{aligned}
\end{equation}
It remains to bound $\norm{v_{k-1}-v'_{k-1}}$ on the right-hand side.
By an easy induction based on~\eqref{eq:one_step_v_estimate}  and since $x_0= x_0'$, $v_0=v_0'$, we obtain, for $\ell \in\{0,\ldots,N-1\}$,
\begin{equation}
  \label{eq:norm_vi_vi'}
  \norm{v_\ell-v_\ell'} \leq \gamma \LipT \sum_{i=1}^{\ell-1} \left(1+\gamma^{1+\delta}\LipT\right)^{\ell-1-i} \norm{x_i-x_i'} + \sum_{i=1}^{\ell} \left(1+\gamma^{1+\delta}\LipT\right)^{\ell-i+1} \norm{z_{i}-z_{i}'}\eqsp.
\end{equation}
In particular, using~the previous inequality for $\ell=k-1$ and \eqref{eq:bound_sum_x_until_k-1}, we get
\begin{equation}
\begin{aligned}
  \norm{v_{k-1}-v_{k-1}'} & \leq \gamma\LipT \left(1+\gamma^{1+\delta}\LipT\right)^k \sum_{i=1}^{k-1} \norm{x_{i}-x_{i}'}+  \left(1+\gamma^{1+\delta}\LipT\right)^k \sum_{i=1}^{k-1} \norm{z_{i}-z_{i}'} \\
  & \leq \left(k\gamma \LipMGammaminor(\lambda)\LipT \left(1+\gamma^{1+\delta}\LipT\right)^k \LipGammaminor\left(\lambda\right)^{k} +  \left(1+\gamma^{1+\delta}\LipT\right)^k \right)\sum_{i=1}^{k-1} \norm{z_{i}-z_{i}'}\eqsp,
\end{aligned}
\end{equation}
which, plugged in the estimate~\eqref{eq:bound_sum_x_term_k} for~$\norm{x_k-x_k'}$ and combined with~\eqref{eq:bound_sum_x_until_k-1}, leads to~\eqref{eq:second_stab_estimate_x}.

To obtain~\eqref{eq:second_stab_estimate_v}, we use first~\eqref{eq:norm_vi_vi'} and isolate the $k$-th term in the sum to write
\begin{equation}
\begin{aligned}
  \sum_{i=1}^{k} \norm{v_i-v_i'} & \leq k\gamma \LipT \left(1+\gamma^{1+\delta}\LipT\right)^k \sum_{i=1}^{k-1} \norm{x_{i}-x_{i}'} + \left(1+\gamma^{1+\delta}\LipT\right)^k \left[ \norm{z_k-z_k'} + k \sum_{i=1}^{k-1} \norm{z_{i}-z_{i}'} \right],
\end{aligned}
\end{equation}
and then conclude with~\eqref{eq:bound_sum_x_until_k-1}. 
\end{proof}

\subsubsection{Stability of maps conditioned by the total noise}
\label{sec:stability_conditioned}

In order to express $(X_{k+1},V_{k+1})$ in  terms of $\Gminorc{k+1}$ and $\tZbfc{k+1}$ defined in \eqref{eq:def_G_minor} and \eqref{eq:def_tZ_gen}, we perform a linear change of variables in the functions~$\Gammaminorc{\gamma}[i]$ defined in~\eqref{eq:def_Phiminor_k}, for $i \in\{0,\ldots,k+1\}$. More precisely, we aim at writing, for $i \in\{0,\ldots,k+1\}$,
\begin{equation}
  \label{eq:21}
  (\XminorTer_i,\VminorTer_i ) = \tGammaminorc{\gamma}[i,k+1]\left(X_0,V_0,\left\{\sqrt{\gamma}\sigma_{\gamma}\tZc{k+1}_j\right\}_{j=1}^{i\wedge (k-1) },\{W_j\}_{j=1}^{i},\sigma_\gamma\Gminorc{k+1}\right) \eqsp.
\end{equation}

To this end, we define, for a given~$(k,\gamma) \in \msec$, the following functions, which take as arguments the initial condition~$(x,v) \in \rset^{2d}$,  the realizations $(\gminor_1,\gminor_2)\in \rset^{2d}$, $\{\widetilde{z}_j\}_{j=1}^{k-1} \in \rset^{d\times(k-1)}$ and $\{w_j\}_{j=1}^{k+1} \in \msw^{k+1}$: for $i \in \{0,\ldots,k+1\}$,
\begin{equation}
  \label{eq:def_ttGammaminorc_0}
\tGammaminorc{\gamma}[i,k+1]\left(x,v,\{\widetilde{z}_j\}_{j=1}^{i\wedge (k-1) },\{w_i\}_{j=1}^{i},(g_1,g_2)\right) = \Gammaminorc{\gamma}[i]\left(x,v,\{z_i \}_{j=1}^{i},\{w_i\}_{j=1}^{i}\right)  \eqsp,
\end{equation}
where 
\begin{align}
  \label{eq:def_z_i_tilde_z_i_g1_g2}
  z_i & = \widetilde{z}_i+ \gamma \alpha_{i}(k,\gamma) g_1 + \gamma \beta_{i}(k,\gamma) g_2, \qquad i\in\iint{1}{k-1} \eqsp,\\
  \label{eq:def_z_k_tilde_z_i_g1_g2}
  z_k & = \frac1\gamma\left( g_1-\sum_{i=0}^{k-2} \gminorbfc{k+1}_{1,i+1} z_{i+1} \right) \eqsp,  \qquad
  z_{k+1}  = g_2-\sum_{i=0}^{k-1} \gminorbfc{k+1}_{2,i+1} z_{i+1} \eqsp. 
\end{align}
Note that \eqref{eq:21} holds in view of~\eqref{eq:def_tZ_gen} and~\eqref{eq:def_Gminor_gen} (mind in particular the change of scaling in~\eqref{eq:def_z_i_tilde_z_i_g1_g2} compared to~\eqref{eq:def_tZ_gen}, which is due to the fact that~$\tilde{z}_j$ correspond to the increments~$\sqrt{\gamma}\widetilde{Z}_j^{(k)}$ in~\eqref{eq:21}). 
The final iterate can then be expressed in terms of the initial condition $(X_0,V_0)$ and the noise~$\Gminorc{k+1}$ by reformulating~\eqref{eq:Xminor_Vminor_decomp_start_gen} as 
\begin{equation}
  \label{eq:def_bGamma_v0}
  (\XminorTer_{k+1},\VminorTer_{k+1}) = \bGammaminorc{\gamma,\sqrt{\gamma}\sigma_{\gamma}\tZbfc{k+1},\bfWc{k+1}}[k+1]\left((X_0,V_0),\sigma_{\gamma}\Gminorc{k+1}\right)  \eqsp,
\end{equation}
where $\bfWc{k+1}=(W_1,\dots,W_{k+1})$, and 
\begin{equation}
    \label{eq:def_bGamma}
    \begin{aligned}
      & \bGammaminorc{\gamma,\{\widetilde{z}_j\}_{j=1}^{k-1},\{w_j\}_{j=1}^{k+1}}[k+1]\Big((x,v),(g_1,g_2)\Big) \\
      & \qquad \qquad = \begin{pmatrix} g_1\\g_2 \end{pmatrix} + \gamma^{\delta} \begin{pmatrix} \Dbf_{\gamma} \tLambda^{(k+1)}_{\{\widetilde{z}_j\}_{j=1}^{k-1}}(g_1,g_2) \\ \Zd \end{pmatrix} + \Ximinorc{\gamma,x,v,\{\widetilde{z}_j\}_{j=1}^{k-1},\{w_j\}_{j=1}^{k+1}}[k+1](g_1,g_2),
    \end{aligned}
\end{equation}
with
\begin{equation}
  \label{eq:def_Ximinor_gen}
  \begin{aligned}
    & \Ximinorc{\gamma,x,v,\{\widetilde{z}_j\}_{j=1}^{k-1},\{w_j\}_{j=1}^{k+1}}[k+1] (g_1,g_2) = \Mminorc{\gamma}[k+1]
    \begin{pmatrix}
     x\\
      v
    \end{pmatrix} \\
    &\qquad\qquad+ \gamma \sum_{i=0}^{k}  \widetilde{\Bminor}^{(i)}\parenthese{\tGammaminorc{\gamma}[i,k+1]\left(x,v,\{\widetilde{z}_j\}_{j=1}^{i \wedge (k-1)},\{w_j\}_{j=1}^{i},(g_1,g_2)\right),\{\widetilde{z}_i\}_{j=1}^{(i+1)\wedge (k-1)},(g_1,g_2),w_{i+1}}
    \eqsp,
  \end{aligned}
\end{equation}
 $\Mminorc{\gamma}[k+1]$ is given by \eqref{eq:def_matrix_Mminor_gen}, and, setting $\{z_i\}_{i=1}^{k+1}$ as in \eqref{eq:def_z_i_tilde_z_i_g1_g2}-\eqref{eq:def_z_k_tilde_z_i_g1_g2},
 \begin{equation}
     \label{eq:def_Ximinor_genb}
\widetilde{\Bminor}^{(i)}\left(x,v,\{\widetilde{z}_j\}_{j=1}^{(i+1)\wedge (k-1)},(g_1,g_2),w_{i+1}\right) = \Bminor^{(i)}(x,v,z_{i+1},w_{i+1})  \eqsp,
\end{equation}
and
\begin{equation}
  \label{eq:def_Lambda}
  \tLambda^{(k+1)}_{\{\widetilde{z}_j\}_{j=1}^{k-1}}\left(g_1,g_2\right) = \sum_{j=1}^{k+1} z_j \eqsp. 
\end{equation}

We give in the next proposition key estimates which allow to prove that $\bGammaminorc{\gamma,x,v,\{\widetilde{z}_j\}_{j=1}^{k-1},\{w_j\}_{j=1}^{k+1}}[k+1]$ is a $\rmC^{1}$-diffeomorphism on $\rset^{2d}$ (see \Cref{propo:inverse_lip_Gamma} below). To state the result, we define, for any~$(k,\gamma) \in \msec$, 
\begin{equation}
  \label{eq:m_k_gamma}
\rmm(k,\gamma)=\max\left(\norm{\alphaminorbf(k+1,\gamma)}_{\infty}, \norm{\betaminorbf(k+1,\gamma)}_{\infty} \right) \eqsp,
\end{equation}
where $\alphaminorbf(k+1,\gamma),\betaminorbf(k+1,\gamma)$ are defined by \eqref{eq:defbetaminor_alpha_minor_j_gen}.

\begin{lemma}
  \label{propo:Lip_psi_i_k_lip_xi_k_gen}
  Assume that \Cref{ass:lip_minor_gen} holds and $\sup_{\gamma \in\ocint{0,\bgamma}} \tau_{\gamma} \leq 1$. For any parameter~$\lambda > 0$, $(k,\gamma) \in \msec$, initial condition~$(x,v) \in \rset^{2d}$, realizations $\tzbf = \{\tz_j\}_{j=1}^{k-1} \subset \rset^{d}$ and $\bfw = \{w_j\}_{j=1}^{k+1} \in \msw^{k+1}$, as well as  $(g_1,g_2),(g'_1,g'_2) \in \rset^{2d}$, it holds 
\begin{equation}
  \label{eq:propo:Lip_psi_i_k_lip_xi_k_gen_1}
  \begin{aligned}
    & \norm{\tLambda^{(k+1)}_{\tzbf}(g_1,g_2) - \tLambda^{(k+1)}_{\tzbf}(g_1',g_2')} \\
    & \qquad \qquad \qquad \leq \left(1+\rmm(k,\gamma) k\gamma + \frac{1-\gminorc}{\gamma}\left[1+ \rmm(k,\gamma)(k\gamma)^2\right] \right)\left( \norm{g_1-g_1'} + \norm{g_2-g_2'} \right),
  \end{aligned}
\end{equation}
and
\begin{equation}
\label{eq:propo:Lip_psi_i_k_lip_xi_k_gen_2}
\begin{aligned}
  & \norm{  \Ximinorc{\gamma,x,v,\tzbf,\bfw}[k+1](g_1,g_2) - \Ximinorc{\gamma,x,v,\tzbf,\bfw}[k+1](g'_1,g'_2)} \leq (2+k\gamma)\LipT\LipXic{k+1,\gamma,\lambda} \left( \norm{g_1-g'_1} + \norm{g_2-g'_2} \right),
\end{aligned}
\end{equation}
with
\begin{equation}
  \label{eq:def_LipXic}
  \begin{aligned}
    \LipXic{k+1,\gamma,\lambda}
    & = \gamma^\delta \left(1+(k\gamma)^2\rmm(k,\gamma)\right)\left(2+\mathscr{D}+\left(1+\gamma^\delta\LipT\right)^k+\gamma \LipT \right) + \gamma^{1+\delta}\left(1+k\gamma\rmm(k,\gamma)\right) \\
    & \qquad + k\gamma^2 \rmm(k,\gamma)\left[k \LipMGammaminor(\lambda) \LipGammaminor(\lambda)^{k}+\mathscr{L}^x_{k,\gamma,\lambda}+\gamma^\delta\left(1+k\mathscr{L}^v_{k,\gamma,\lambda}\right)\right], 
  \end{aligned}
\end{equation}
where $\LipMGammaminor(\lambda), \LipGammaminor(\lambda)$ and~$\mathscr{L}^x_{k,\gamma,\lambda},\mathscr{L}^v_{k,\gamma,\lambda}$ are defined in \eqref{eq:def_LipGammaminor} and~\eqref{eq:def_mathscr_L}, respectively.
\end{lemma}

\begin{proof}
  Fix~$\lambda > 0$, $(k,\gamma) \in \msec$, $(x,v) \in \rset^{2d}$, $\tzbf = \{\tz_j\}_{j=1}^{k-1} \subset \rset^{d}$, $\bfw = \{w_j\}_{j=1}^{k+1} \in \msw^{k+1}$, and  $(g_1,g_2),(g'_1,g'_2) \in \rset^{2d}$. We introduce for $i \in\iint{0}{k}$, $  (x_i,v_i)  =\tGammaminorc{\gamma}[i,k+1](x,v,\{\widetilde{z}_j\}_{j=1}^{i \wedge (k-1)},\{w_j\}_{j=1}^{i},(\gminor_1,\gminor_2))$, and similarly $    (x_i',v_i')  = \tGammaminorc{\gamma}[i,k+1](x,v,\{\widetilde{z}_j\}_{j=1}^{i \wedge (k-1)},\{w_j\}_{j=1}^{i},(\gminor_1',\gminor_2'))$.
 With this notation and recalling~\eqref{eq:def_Phiminor_k}, we obtain by the definition~\eqref{eq:def_ttGammaminorc_0} of $\tGammaminorc{\gamma}[i,k]$ that
  \begin{equation}
    \label{eq:proof_xi_lip_lem:Lip_Phi_gen}
    (x_{i},v_{i}) =  \Gammaminorc{\gamma}[i]\left(x,v,\{z_j\}_{j=1}^{i},\{w_j\}_{j=1}^{i}\right) \eqsp,
    \qquad
    (x'_{i},v'_{i}) = \Gammaminorc{\gamma}[i]\left(x,v,\{z'_j\}_{j=1}^{i},\{w_j\}_{j=1}^{i}\right) \eqsp,
  \end{equation}
    where~$\{z_i\}_{i=1}^{k+1}$ and $\{z_i'\}_{i=1}^{k+1}$ are the sequences defined in~\eqref{eq:def_z_i_tilde_z_i_g1_g2} based on $(g_1,g_2)$ and $(g_1',g_2')$, respectively.
  We start by estimating the difference~$\norm{z_i-z_i'}$ for $i\in\iint{1}{k+1}$. First, by~\eqref{eq:def_z_i_tilde_z_i_g1_g2}, for any $i \in \{1,\dots,k-1\}$,
  \begin{equation}
    \label{eq:proof_Lip_bound_z_1_k-1}
     \norm{z_i-z_i'} \leq \gamma \rmm(k,\gamma)\left( \norm{g_1-g_1'} + \norm{g_2-g_2'} \right).
  \end{equation}
  Next, by~\eqref{eq:def_z_k_tilde_z_i_g1_g2} and~\eqref{lem:estimate_gminorbfc} we get 
  \begin{equation}
    \label{eq:proof_Lip_bound_z_k_z_k_p}
    \begin{aligned}
      \norm{z_k-z_k'} & \leq \frac1\gamma \norm{g_1-g_1'} +  k \sum_{i=1}^{k-1} \norm{z_i-z_i'} \leq \frac{1 + (k\gamma)^2 \rmm(k,\gamma)}{\gamma} \left( \norm{g_1-g_1'} + \norm{g_2-g_2'} \right), \\
      \norm{z_{k+1}-z_{k+1}'} &\leq  \norm{g_2-g_2'} + \sum_{i=1}^{k} \norm{z_i-z_i'} \\
      & \leq \norm{z_{k}-z_k'} + [1 + k\gamma \rmm(k,\gamma)] \left( \norm{g_1-g_1'} + \norm{g_2-g_2'} \right).
    \end{aligned}
  \end{equation}
  We are now in position to prove~\eqref{eq:propo:Lip_psi_i_k_lip_xi_k_gen_1}. By \eqref{eq:def_Lambda} and \eqref{eq:def_gminorbfc} and since $g_2 = \sum_{i=0}^{k}\gminorbfc{k+1}_{2,i+1} z_{i+1}$ by \eqref{eq:def_z_k_tilde_z_i_g1_g2}, 
  \begin{align}
    &  \norm{\tLambda^{(k+1)}_{\{\widetilde{z}_j\}_{j=1}^{k-1}}(g_1,g_2)-\tLambda^{(k+1)}_{\{\widetilde{z}_j\}_{j=1}^{k-1}}(g_1',g_2')} \leq \norm{g_2-g_2'} + \sum_{i=0}^{k-1} \left(1-\gminorbfc{k+1}_{2,i+1}\right) \norm{ z_{i+1} -z_{i+1}'} \\
    & \qquad \qquad \qquad \qquad \leq \norm{g_2-g_2'} + (1-\gminorc)\norm{z_k-z_k'} + \sum_{i=0}^{k-2} \norm{ z_{i+1}  -z_{i+1}'} \\
    & \qquad \qquad \qquad \qquad \leq \norm{g_2-g_2'} + (1-\gminorc)\norm{z_k-z_k'}+ \rmm(k,\gamma) k\gamma \left( \norm{g_1-g_1'}+ \norm{g_2-g_2'}\right) \eqsp.
\end{align}
  The estimate~\eqref{eq:propo:Lip_psi_i_k_lip_xi_k_gen_1} follows by combining the last inequality with~\eqref{eq:proof_Lip_bound_z_k_z_k_p}.

We next turn to the proof of \eqref{eq:propo:Lip_psi_i_k_lip_xi_k_gen_2}. First, by \eqref{eq:def_Ximinor_gen}-\eqref{eq:def_Ximinor_genb} and the notation introduced at the beginning of the proof, 
\begin{equation}
  \begin{aligned}
    \label{eq:proof_Lip_Xi_bound_decomp_1}
    \norm{  \Ximinorc{\gamma,x,v,\tzbf,\bfw}[k+1](g_1,g_2) - \Ximinorc{\gamma,x,v,\tzbf,\bfw}[k+1](g'_1,g'_2)} \leq \gamma \sum_{i=0}^k \norm{\Theta^{(i)}(x_i,v_i,z_{i+1},w_{i+1})-\Theta^{(i)}(x_i',v_i',z_{i+1}',w_{i+1})} & \\
    \leq (2+k\gamma)\LipT \gamma\left( \sum_{i=1}^k \norm{x_i-x_i'} + \gamma^\delta \sum_{i=1}^k \norm{v_i-v_i'}+\gamma^{\delta}\sum_{i=0}^k \norm{z_{i+1}-z'_{i+1}} \right), &
  \end{aligned}
\end{equation}
where we used~\eqref{eq:lip_Bminor} for the last inequality, and the fact that~$(x_0,v_0)=(x_0',v_0')$ to eliminate the term~$i=0$ in the first two sums on the right hand side of the last inequality. Therefore, in view of~\eqref{eq:second_stab_estimate_x} and~\eqref{eq:second_stab_estimate_v} in \Cref{lem:Lip_Phi_gen}, 
\begin{equation}
\begin{aligned}
  & \norm{  \Ximinorc{\gamma,x,v,\tzbf,\bfw}[k+1](g_1,g_2) - \Ximinorc{\gamma,x,v,\tzbf,\bfw}[k+1](g'_1,g'_2)}\\
  & \qquad \qquad \leq (2+k\gamma)\LipT \gamma\left( \gamma^\delta \norm{z_{k+1}-z_{k+1}'} + \gamma^\delta \mathscr{K}_{k,\gamma} \norm{z_k-z_k'} + \mathscr{R}_{k,\gamma,\lambda} \sum_{i=1}^{k-1} \norm{z_{i}-z'_{i}} \right),
  \end{aligned}
\end{equation}
with $\mathscr{K}_{k,\gamma} = 1 + \mathscr{D} + \left(1+\gamma^{1+\delta}\LipT\right)^k + \gamma\LipT$ and $\mathscr{R}_{k,\gamma,\lambda} = k \LipMGammaminor(\lambda) \LipGammaminor(\lambda)^{k}+ \mathscr{L}^x_{k,\gamma,\lambda} + \gamma^\delta (1+k\mathscr{L}^v_{k,\gamma,\lambda})$. The conclusion then follows from~\eqref{eq:proof_Lip_bound_z_1_k-1} and~\eqref{eq:proof_Lip_bound_z_k_z_k_p}.
\end{proof}

\subsubsection{Estimates on diffeomorphisms defined for the total noise}
\label{sec:stability_diffeo}

 We use in this section the stability results obtained in Lemmas~\ref{lem:Lip_Phi_gen} and~\ref{propo:Lip_psi_i_k_lip_xi_k_gen} to prove that the function $\bGammaminorc{\gamma,\{\widetilde{z}_j\}_{j=1}^{k-1},\{w_j\}_{j=1}^{k+1}}[k+1]$ giving the $(k+1)$-th iterate in~\eqref{eq:def_bGamma}, considered as a function of $(g_1,g_2)$ for fixed initial condition $(x,v)$, is a $\rmC^1$-diffeomorphism. In addition, we also establish regularity properties of this function and its inverse with respect to the initial condition $(x,v)$. To state our results, we introduce the following functions for~$(k,\gamma) \in\msec$, and given realizations~$\tzbf = \{\tz_j\}_{j=1}^{k-1} \subset \rset^{d}$, $\bfw=\{w_j\}_{j=1}^{k+1} \subset \rset^{m}$:
\begin{itemize}[wide, labelwidth=!, labelindent=0pt]
\item for a given initial condition~$(x,v) \in \rset^{2d}$, set $\bGammaminorc{\gamma,x,v,\tzbf,\bfw}[k+1](g_1,g_2) = \bGammaminorc{\gamma,\tzbf,\bfw}[k+1]((x,v),(g_1,g_2))$;
\item for a given realization~$(g_1,g_2) \in \rsetdd$, set~$\bGammaminorc{\gamma,\tzbf,\bfw,g_1,g_2}[k+1](x,v) = \bGammaminorc{\gamma,\tzbf,\bfw}[k+1]((x,v),(g_1,g_2))$. 
\end{itemize}
The first result is that~$\bGammaminorc{\gamma,x,v,\tzbf,\bfw}[k+1]$ is a diffeomorphism when~$k\gamma$ is a sufficiently small positive time and~$\gamma$ is not too large.

\begin{proposition}
  \label{propo:inverse_lip_Gamma}
    Assume that \Cref{ass:gminorc} and \Cref{ass:lip_minor_gen} hold. There exists~$\tbound >0$ such that, for any $t_0\in \ocint{0,\tbound}$, there is $\bar{\gamma}_{t_0}>0$ for which, for any $\gamma\in \ocint{0,\bar{\gamma}_{t_0}}$, $(x,v) \in \rset^{2d}$,  $\tzbf \in \rset^{(k-1)\times d}$ and $\bfw \in \rset^{(k+1)\times m}$
  \begin{enumerate}[wide, labelwidth=!, labelindent=0pt, label=(\alph*)]
\item   \label{propo:inverse_lip_Gamma_a} \label{item:3:propo:Lip_gamma_i_k_lip_xi_k}  
 $\bGammaminorc{\gamma,x,v,\tzbf,\bfw}[\lfloor t_0/\gamma\rfloor+1]$  is $3/2$-Lipschitz and a $\rmC^{1}$-diffeomorphism on $\rset^{2d}$ satisfying, for any $(g_1,g_2),(g'_1,g'_2) \in\rset^{2d}$,
  \begin{equation}
    \label{eq:ineq_bGamma_lower_bound}
    \norm{\bGammaminorc{\gamma,x,v,\tzbf,\bfw}[\lfloor t_0/\gamma\rfloor+1](g_1,g_2) - \bGammaminorc{\gamma,x,v,\tzbf,\bfw}[\lfloor t_0/\gamma\rfloor+1](g'_1,g'_2) } \geq \frac12 \norm{\left(g_1-g'_1,g_2-g'_2\right)} \eqsp;
  \end{equation}
  \item   \label{propo:inverse_lip_Gamma_b} \label{item:3:2:propo:Lip_gamma_i_k_lip_xi_k}  
    the inverse of~$\bGammaminorc{\gamma,x,v,\tzbf,\bfw}[\lfloor t_0/\gamma\rfloor]$, denoted by~$\bGammaminorinvc{\gamma,x,v,\tzbf,\bfw}[\lfloor t_0/\gamma\rfloor]$, is $2$-Lipschitz: for any $(u_1,u_2),(u'_1,u'_2) \in\rset^{2d}$,
  \begin{equation}
    \label{eq:ineq_invbGamma_upper_bound}
    \norm{\bGammaminorinvc{\gamma,x,v,\tzbf,\bfw}[\lfloor t_0/\gamma\rfloor+1](u_1,u_2) - \bGammaminorinvc{\gamma,x,v,\tzbf,\bfw}[\lfloor t_0/\gamma\rfloor+1](u'_1,u'_2) } \leq 2\norm{(u_1-u'_1,u_2-u'_2)} \eqsp.
  \end{equation}
  \end{enumerate}
\end{proposition}

\begin{proof}
  We prove that there exists~$\btZ >0$ such that, for any
  $t_0 \in (0,\btZ]$, there is $\bgamma_{t_0} >0$ for which
  $(g_1,g_2) \mapsto \bGammaminorc{\gamma,x,v,\tzbf,\bfw}[\lfloor
  t_0/\gamma\rfloor](g_1,g_2)-(g_1,g_2)$ is
  $1/2$-Lipschitz. Item~\ref{item:3:propo:Lip_gamma_i_k_lip_xi_k} is
  then a straightforward consequence of~\cite[Exercise~3.26]{duistermaat:kolk:2004}.  First, for any $\lambda >0$,
  $(k,\gamma) \in \msec $, $(x,v) \in \rset^{2d}$,
  $\tzbf \in \rset^{(k-1)\times d}$ 
  $\bfw \in \rset^{(k+1)\times m}$, $(g_1,g_2),(g'_1,g'_2)
  \in\rset^{2d}$, by~\eqref{eq:def_bGamma} and \Cref{propo:Lip_psi_i_k_lip_xi_k_gen}, we have 
    \[
    \begin{aligned}
      & \left\| \bGammaminorc{\gamma,x,v,\{\widetilde{z}_j\}_{j=1}^{k-1},\{w_j\}_{j=1}^{k+1}}[k+1](g_1,g_2) - \begin{pmatrix} g_1\\g_2 \end{pmatrix} - \bGammaminorc{\gamma,x,v,\{\widetilde{z}_j\}_{j=1}^{k-1},\{w_j\}_{j=1}^{k+1}}[k+1](g_1',g_2') + \begin{pmatrix} g_1'\\g_2' \end{pmatrix} \right\| \\
      & \qquad \leq \gamma^{\delta}\mathscr{D} \left\| \tLambda^{(k+1)}_{\{\widetilde{z}_j\}_{j=1}^{k-1}}(g_1,g_2)-\tLambda^{(k+1)}_{\{\widetilde{z}_j\}_{j=1}^{k-1}}(g_1',g_2') \right\| \\
      & \qquad \qquad + \left\| \Ximinorc{\gamma,x,v,\{\widetilde{z}_j\}_{j=1}^{k-1},\{w_j\}_{j=1}^{k+1}}[k+1](g_1,g_2) - \Ximinorc{\gamma,x,v,\{\widetilde{z}_j\}_{j=1}^{k-1},\{w_j\}_{j=1}^{k+1}}[k+1](g_1',g_2') \right\| \\
      & \qquad \leq C(k+1,\gamma,\lambda) \left( \left\|g_1-g_1'\right\|+\left\|g_2-g_2'\right\|\right) \eqsp,
    \end{aligned}
  \]
  where we have set
  \begin{equation}
    \label{eq:7}
 C(k+1,\gamma,\lambda)  =  \gamma^{\delta}\mathscr{D}\parenthese{1+\rmm(k,\gamma) k\gamma + \frac{1-\gminorc}{\gamma}[1+ \rmm(k,\gamma)(k\gamma)^2]} +    (2+k\gamma)\LipT\LipXic{k+1,\gamma,\lambda} \eqsp.
\end{equation}

As noted previously, there is $\btZ^{(1)} \geq 0 $ for which, for any $t_0\in\ocintLigne{0,\btZ^{(1)}}$, there exists $\bar{\gamma}_{t_0}^{(1)}$ such that for any $\gamma\in \ocintLigne{0,\bar{\gamma}_{t_0}^{(1)}}$, we have $(\lfloor t_0/\gamma\rfloor,\gamma)\in \msec$.
Then,  \Cref{lem:minor_alpha_beta_gen} implies that there exists $\btZ^{(2)} \in \ocintLigne{0,\btZ^{(1)}}$ such that for any $t_0\in\ocintLigne{0,\btZ^{(2)}}$,  $ \limsup_{\gamma\downarrow 0} \rmm(\floor{t_0/\gamma},\gamma) \leq K[t_0^{-2} \vee t_0^{-1}]$ for some constant $K \geq 0$.
This implies, for any $t_0 \in \ocintLigne{0,\btZ^{(2)}}$ and $\lambda >0$,

    \[
    \begin{aligned}
      \limsup_{\gamma\downarrow 0} C(\lfloor
      t_0/\gamma\rfloor+1,\gamma,\lambda)
      & = \limsup_{\gamma\downarrow 0}(2+\lfloor t_0/\gamma\rfloor\gamma)\LipT \lfloor t_0/\gamma\rfloor^2 \gamma^2 \LipMGammaminor(\lambda) \LipGammaminor(\lambda)^{\lfloor
      t_0/\gamma\rfloor } \rmm(\lfloor t_0/\gamma\rfloor,\gamma)\\
      & = \lambda (2+ t_0)\LipT t_0^2 \exp\left( t_0 \left[ \frac1\lambda + (1+\lambda)\Ltt\right]\right) \limsup_{\gamma\downarrow 0} \rmm(\lfloor t_0/\gamma\rfloor,\gamma) \\
      & \leq \lambda (2+ t_0)\LipT K [t_0\vee 1] \exp\left( t_0 \left[
          \frac1\lambda + (1+\lambda)\Ltt\right]\right) \eqsp,
    \end{aligned}
    \]
    where we have used for the penultimate inequality the expressions of $\LipMGammaminor(\lambda), \LipGammaminor(\lambda)$ provided by \eqref{eq:def_LipGammaminor}. The latter quantity is smaller than~1/4 for any~$t_0\in(0,\btZ^{(2)}]$ upon choosing first~$\lambda = 1/(16K\LipT)$, and then~$\btZ \in \ocintLigne{0,\btZ^{(2)}}$ sufficiently small so that for any $t_0 \in \ocint{0,\btZ}$, 
    \[
      (2+ t_0)[t_0\vee 1] \exp\left( t_0 \left[ 16K\LipT + \left(1+\frac{1}{16K\LipT}\right)\Ltt\right]\right) \leq 4 \eqsp. 
    \]
    This leads to the claimed statement, and therefore proves item~\ref{item:3:propo:Lip_gamma_i_k_lip_xi_k}. Item~\ref{item:3:2:propo:Lip_gamma_i_k_lip_xi_k} next easily follows from~\eqref{eq:ineq_bGamma_lower_bound}. 
\end{proof}

The second result is that $\bGammaminorc{\gamma,\tzbf,\bfw,g_1,g_2}$ is a diffeomorphism when $k\gamma$ is a sufficiently small positive time and~$\gamma$ is not too large.

\begin{proposition}
  \label{propo:inverse_lip_gamma_initial_condition}
  Assume that \Cref{ass:gminorc} and \Cref{ass:lip_minor_gen} hold. There exists $\tbound >0$, such that, for any~$t_0\in \ocint{0,\tbound}$, there is $\bar{\gamma}_{t_0} \in (0,1]$ for which, for any $\gamma\in \ocint{0,\bar{\gamma}_{t_0}}$, $(g_1,g_2),(u_1,u_2) \in \rsetdd$, $\tzbf \in \rset^{(k-1)\times d}$ and $\bfw \in \rset^{(k+1)\times m}$, 
  \begin{enumerate}[wide, labelwidth=!, labelindent=0pt, label=(\alph*)]
  \item \label{item:1:propo:inverse_lip_gamma_initial_condition}
    the mapping~$\bGammaminorc{\gamma,\tzbf,\bfw,g_1,g_2}[\lfloor t_0/\gamma \rfloor+1]$ is Lipschitz on~$\rset^{2d}$: for any $(x,v),(x',v') \in \rset^{2d}$,
    \begin{equation}
       \norm{\bGammaminorc{\gamma,\tzbf,\bfw,g_1,g_2}[\lfloor t_0/\gamma \rfloor+1](x,v) - \bGammaminorc{\gamma,\tzbf,\bfw,g_1,g_2}[\lfloor t_0/\gamma \rfloor+1](x',v')} \leq \rme^{(1+2\Ltt)(t_0+\gamma)}\left( \norm{x-x'}+ \norm{v-v'}\right)\eqsp; 
    \end{equation}
  \item  \label{item:2:propo:inverse_lip_gamma_initial_condition}
    the inverse of~$\bGammaminorc{\gamma,x,v,\tzbf,\bfw}[\lfloor t_0/\gamma \rfloor+1]$ (see \Cref{propo:inverse_lip_Gamma}), namely $(x,v) \mapsto \bGammaminorinvc{\gamma,x,v,\tzbf,\bfw}[\lfloor t_0/\gamma \rfloor+1](u_1,u_2)$, is Lipschitz on $\rset^{2d}$: for any $(x,v),(x',v') \in \rset^{2d}$, 
    \begin{equation}
      \label{eq:item:2:propo:inverse_lip_gamma_initial_condition}
      \norm{\bGammaminorinvc{\gamma,x,v,\tzbf,\bfw}[\lfloor t_0/\gamma \rfloor+1](u_1,u_2) - \bGammaminorinvc{\gamma,x',v',\tzbf,\bfw}[\lfloor t_0/\gamma \rfloor+1](u_1,u_2)} \leq 2 \rme^{(1+2\Ltt)(t_0+\gamma)}\left( \norm{x-x'}+ \norm{v-v'} \right) \eqsp. 
    \end{equation}
  \end{enumerate}
\end{proposition}

\begin{proof}
Let $\btZ \geq 0 $ such that for any $t_0\in\ocintLigne{0,\btZ}$, there exists $\bar{\gamma}_{t_0}$ for which, for any $\gamma\in \ocintLigne{0,\bar{\gamma}_{t_0}}$, $(\lfloor t_0/\gamma\rfloor,\gamma)\in \msec$ and \Cref{propo:inverse_lip_Gamma}-\ref{propo:inverse_lip_Gamma_a}-\ref{propo:inverse_lip_Gamma_b} hold. 
To prove item~\ref{item:1:propo:inverse_lip_gamma_initial_condition}, we note that~\eqref{eq:21}-\eqref{eq:def_ttGammaminorc_0} and~\eqref{eq:first_stab_estimate} in \Cref{lem:Lip_Phi_gen} with $\lambda =1$ imply for any $t_0\in\ocintLigne{0,\btZ}$,  $\gamma\in \ocintLigne{0,\bar{\gamma}_{t_0}}$,
$(g_1,g_2) \in \rsetdd$, $\tzbf \in \rset^{(k-1)\times d}$ and $\bfw \in \rset^{(k+1)\times m}$, $(x,v),(x',v') \in \rset^{2d}$,
  \begin{equation}
    \norm{\bGammaminorc{\gamma,x,v,\tzbf,\bfw}[\lfloor t_0/\gamma \rfloor+1](g_1,g_2) - \bGammaminorc{\gamma,x',v',\tzbf,\bfw}[\lfloor t_0/\gamma \rfloor+1](g_1,g_2)} \leq \LipGammaminor(1)^{\lfloor t_0/\gamma \rfloor+1}\parentheseDeux{\norm{x-x'}+\norm{v-v'}} \eqsp,    
  \end{equation}
  from which the result follows in view of the bound
  \[
  \LipGammaminor(\lambda)^{\lfloor t_0/\gamma \rfloor+1} \leq \exp\left((t_0+\gamma) \left[ \frac1\lambda + (1+\lambda) \max\left(1,\frac{\gamma^\delta}{\lambda}\right)\Ltt\right] \right)\eqsp,
  \]
  and the fact that~$\gamma \leq \bar{\gamma}_{t_0} \leq 1$.

 \sloppy Regarding \ref{item:2:propo:inverse_lip_gamma_initial_condition},
 let
  $(x,v),(x',v') \in \rset^{2d}$ and denote by
  $(g_1,g_2) = \bGammaminorinvc{\gamma,x,v,\tzbf,\bfw}[\lfloor
  t_0/\gamma \rfloor](u_1,u_2)$ and
  $(g'_1,g'_2) = \bGammaminorinvc{\gamma,x',v',\tzbf,\bfw}[\lfloor
  t_0/\gamma \rfloor](u_1,u_2)$. We obtain using
  $(g'_1,g'_2) = \bGammaminorinvc{\gamma,x',v',\tzbf,\bfw}[\lfloor
  t_0/\gamma \rfloor](\bGammaminorc{\gamma,\tzbf,\bfw,g_1,g_2}[\lfloor
  t_0/\gamma \rfloor](x,v))$ and 
  \Cref{propo:inverse_lip_Gamma}-\ref{item:3:2:propo:Lip_gamma_i_k_lip_xi_k}, for any $t_0\in\ocintLigne{0,\btZ}$,  $\gamma\in \ocintLigne{0,\bar{\gamma}_{t_0}}$,
$(g_1,g_2) \in \rsetdd$, $\tzbf \in \rset^{(k-1)\times d}$ and $\bfw \in \rset^{(k+1)\times m}$,
  \begin{align}
    &    \txts \norm{ \bGammaminorinvc{\gamma,x,v,\tzbf,\bfw}[\lfloor t_0/\gamma \rfloor+1](u_1,u_2) 
      - \bGammaminorinvc{\gamma,x',v',\tzbf,\bfw}[\lfloor t_0/\gamma \rfloor+1](u_1,u_2)} = \norm{(g_1,g_2)-(g_1',g_2')}\\
    &\qquad \qquad \qquad\qquad \qquad \qquad \txts = \norm{\bGammaminorinvc{\gamma,x',v',\tzbf,\bfw}[\lfloor t_0/\gamma \rfloor+1](\bGammaminorc{\gamma,\tzbf,\bfw,g_1,g_2}[\lfloor t_0/\gamma \rfloor+1](x',v'))-\bGammaminorinvc{\gamma,x',v',\tzbf,\bfw}[\lfloor t_0/\gamma \rfloor+1](\bGammaminorc{\gamma,\tzbf,\bfw,g_1,g_2}[\lfloor t_0/\gamma \rfloor+1](x,v))}\\
    &\qquad \qquad \qquad\qquad \qquad \qquad \txts \leq 2 \norm{\bGammaminorc{\gamma,\tzbf,\bfw,g_1,g_2}[\lfloor t_0/\gamma \rfloor+1](x',v')-\bGammaminorc{\gamma,\tzbf,\bfw,g_1,g_2}[\lfloor t_0/\gamma \rfloor+1](x,v)} \eqsp,
  \end{align}
  which completes the proof using \ref{item:1:propo:inverse_lip_gamma_initial_condition}.
\end{proof}

\subsection{Proof of \Cref{theo:minor_main_gen}}
\label{sec:proof_main_resukt_gen}

We can now finally provide the proof of \Cref{theo:minor_main_gen}.

\begin{proof}[Proof of \Cref{theo:minor_main_gen}]
  For $(x,v) \in \rset^{2d}$ and $\msb \in \mathcal{B}(\rset^{2d})$, it holds, by the definition \eqref{eq:def_bGamma},
  \begin{equation}
  \Rkerminorc{\gamma}^{\lfloor t_0/\gamma \rfloor+1}((x,v),\msb) =\expe{\1_{\msb}\left\{\bGammaminorc{\gamma,x,v, \sqrt{\gamma}\sigma_{\gamma}\tZbfc{\lfloor t_0/\gamma \rfloor+1},\bfWc{\lfloor t_0/\gamma \rfloor+1}}[\lfloor t_0/\gamma \rfloor+1]\left(\sigma_{\gamma}\Gminorc{\lfloor t_0/\gamma \rfloor+1}\right)\right\}}.
  \end{equation}
  Fix~$M>0$,  and consider~$\tbound > 0$ such that the statements of Propositions~\ref{propo:inverse_lip_Gamma} and~\ref{propo:inverse_lip_gamma_initial_condition} hold true. Introduce next $t_0\in \ocint{0,\tbound}$ and the corresponding stepsize~$\bar{\gamma}_{t_0}$ as given by Propositions~\ref{propo:inverse_lip_Gamma} and~\ref{propo:inverse_lip_gamma_initial_condition}.

  The random variables $\sigma_{\gamma} \Gminorc{\lfloor t_0/\gamma \rfloor+1}$ and $\sqrt{\gamma}\sigma_{\gamma} \tZbfc{\lfloor t_0/\gamma \rfloor+1}$ are independent Gaussian random variables by \Cref{lem:indep_z_tilde_gen}. Denoting their densities with respect to the Lebesgue measure by~$\varphi_{t_0,\gamma}$ and~$\psi_{t_0,\gamma}$ respectively, setting $n_0 = \lfloor t_0/\gamma \rfloor+1$, and using a change of variable, for any  $\gamma \in \ocint{0,\bgamma_{t_0}}$
  \begin{equation}
  \begin{aligned}
    & \Rkerminorc{\gamma}^{\lfloor t_0/\gamma \rfloor+1}((x,v),\msb) = \int_{\rset^{n_0 d}\times \msw^{n_0}} \1_{\msb}\left\{\bGammaminorc{\gamma,x,v,\tzbf,\bfw}[\lfloor t_0/\gamma \rfloor+1](\gminor_1,\gminor_2)\right\} \varphi_{t_0,\gamma}(\gminor_1,\gminor_2)\psi_{t_0,\gamma}(\tzbf) \, \rmd \tzbf \, \wdensity{\otimes \lfloor t_0/\gamma \rfloor+1}(\rmd\bfw) \, \rmd \gminor_1 \, \rmd \gminor_2 \\
    & = \int_{\rset^{n_0 d}\times \msw^{n_0}} \!\!\!\1_{\msb}(u_1,u_2)\JacD_{\bGammaminorinvc{\gamma,x,v,\tzbf,\bfw}[\lfloor t_0/\gamma \rfloor+1]}(u_1,u_2) \varphi_{t_0,\gamma}\!\left(\bGammaminorinvc{\gamma,x,v,\tzbf,\bfw}[\lfloor t_0/\gamma \rfloor+1](u_1,u_2)\right)\!\psi_{t_0,\gamma}(\tzbf) \, \rmd \tzbf \, \wdensity{\otimes \lfloor t_0/\gamma \rfloor+1}(\rmd\bfw) \, \rmd u_1 \, \rmd u_2 \eqsp,
  \end{aligned}
  \end{equation}
  where $\bGammaminorinvc{\gamma,x,v,\tzbf,\bfw}[\lfloor t_0/\gamma \rfloor+1]$ is the inverse of~$\bGammaminorc{\gamma,x,v,\tzbf,\bfw}[\lfloor t_0/\gamma \rfloor+1]$ (well defined by~\Cref{propo:inverse_lip_Gamma}) and~$\JacD_{\bGammaminorinvc{\gamma,x,v,\tzbf,\bfw}[\lfloor t_0/\gamma \rfloor+1]}(u_1,u_2)$ is the absolute value of the determinant of the Jacobian matrix of this mapping. 

  \sloppy We can now introduce the reference point~$(\Zd,\Zd)$,  and relate the transition probability starting from~$(x,v)$ in terms of transitions starting from this $(\Zd,\Zd)$. For reasons that will become clear below and similarly to what is done in \Cref{sec:idea_ULA}, we replace $\bGammaminorinvc{\gamma,x,v,\tzbf,\bfw}[\lfloor t_0/\gamma \rfloor+1](u_1,u_2)$ by~$\sqrt{2}\,\bGammaminorinvc{\gamma,\Zd,\Zd,\tzbf,\bfw}[\lfloor t_0/\gamma \rfloor+1](u_1,u_2)$ and not simply~$\bGammaminorinvc{\gamma,\Zd,\Zd,\tzbf,\bfw}[\lfloor t_0/\gamma \rfloor+1](u_1,u_2)$ in $\varphi_{t_0,\gamma}$, henceforth write for any  $\gamma \in \ocint{0,\bgamma_{t_0}}$
  \begin{equation}
    \begin{aligned}
      \label{eq:proof_small_set_mino_decomp_1_gen}
      \Rkerminorc{\gamma}^{\lfloor t_0/\gamma \rfloor+1}((x,v),\msb) & = \int_{\rset^{n_0 d}\times \msw^{n_0}}\1_{\msb}(u_1,u_2) A_{{\gamma,x,v,\tzbf,\bfw}}(u_1,u_2) \JacD_{\bGammaminorinvc{\gamma,\Zd,\Zd,\tzbf,\bfw}[\lfloor t_0/\gamma \rfloor+1]}(u_1,u_2) \\
      & \qquad \qquad \times \varphi_{t_0,\gamma} \left(\sqrt{2} \, \bGammaminorinvc{\gamma,\Zd,\Zd,\tzbf,\bfw}[\lfloor t_0/\gamma \rfloor+1](u_1,u_2) \right) \psi_{t_0,\gamma}(\tzbf) \, \rmd \tzbf \, \wdensity{\otimes \lfloor t_0/\gamma \rfloor+1}(\rmd\bfw) \, \rmd u_1 \, \rmd u_2 \eqsp,
    \end{aligned}
  \end{equation}
where $A_{{\gamma,x,v,\tzbf,\bfw}}(u_1,u_2)  =   A^{(1)}_{{\gamma,x,v,\tzbf,\bfw}}(u_1,u_2)   A^{(2)}_{{\gamma,x,v,\tzbf,\bfw}}(u_1,u_2)$ with
\begin{align}
  \label{eq:6}
  A^{(1)}_{{\gamma,x,v,\tzbf,\bfw}}(u_1,u_2) & = \frac{\JacD_{\bGammaminorinvc{\gamma,x,v,\tzbf,\bfw}[\lfloor t_0/\gamma \rfloor+1]}(u_1,u_2)}{\JacD_{\bGammaminorinvc{\gamma,\Zd,\Zd,\tzbf,\bfw}[\lfloor t_0/\gamma \rfloor+1]}(u_1,u_2)}\eqsp, \qquad 
  A^{(2)}_{{\gamma,x,v,\tzbf,\bfw}}(u_1,u_2) = \frac{\varphi_{t_0,\gamma}\left(\bGammaminorinvc{\gamma,x,v,\tzbf,\bfw}[\lfloor t_0/\gamma \rfloor+1](u_1,u_2)\right)}{\varphi_{t_0,\gamma}\left(\sqrt{2}\,\bGammaminorinvc{\gamma,\Zd,\Zd,\tzbf,\bfw}[\lfloor t_0/\gamma \rfloor+1](u_1,u_2)\right)} \eqsp. 
\end{align}
We next bound $A^{(1)}_{{\gamma,x,v,\tzbf,\bfw}}(u_1,u_2)$ and~$A^{(2)}_{{\gamma,x,v,\tzbf,\bfw}}(u_1,u_2)$ from below. For the latter term, we assume that~$\norm{x} + \norm{v} \leq M$.

By an application of Hadamard's inequality (see \Cref{propo:bound_det_Lipsc_map} below) and \Cref{propo:inverse_lip_Gamma}-\ref{item:3:propo:Lip_gamma_i_k_lip_xi_k}, we obtain for any $
 (u_1,u_2) \in \rset^{2d}$, $(\tzbf,\bfw) \in \rset^{n_0 d}\times \msw^{n_0}$ and   $\gamma \in \ocint{0,\bgamma_{t_0}}$,
\begin{equation}
\JacD_{\bGammaminorinvc{\gamma,x,v,\tzbf,\bfw}[\lfloor t_0/\gamma \rfloor+1]}(u_1,u_2) = \frac{1}{\JacD_{\bGammaminorc{\gamma,x,v,\tzbf,\bfw}[\lfloor t_0/\gamma \rfloor+1]}\left(\bGammaminorinvc{\gamma,x,v,\tzbf,\bfw}[\lfloor t_0/\gamma \rfloor+1](u_1,u_2)\right)} \geq  \left(\frac23\right)^d.
\end{equation}
Similarly, $\JacD_{\bGammaminorinvc{\gamma,\Zd,\Zd,\tzbf,\bfw}[\lfloor t_0/\gamma \rfloor+1]}(u_1,u_2) \leq  2^d$ by \Cref{propo:inverse_lip_Gamma}-\ref{item:3:2:propo:Lip_gamma_i_k_lip_xi_k}, so that for any $ (u_1,u_2) \in \rset^{2d}$, $(\tzbf,\bfw) \in \rset^{n_0 d}\times \msw^{n_0}$ and  $\gamma \in \ocint{0,\bgamma_{t_0}}$,
\begin{equation}
  \label{eq:bound_A_1_gen}
  A^{(1)}_{{\gamma,x,v,\tzbf,\bfw}}(u_1,u_2)\geq \frac{1}{3^d} \eqsp .
\end{equation}
We next recall that $\varphi_{t_0,\gamma}$ is a Gaussian density with mean~0 and covariance matrix~$\sigma_\gamma^2 \cminorc{\lfloor t_0/\gamma \rfloor+1}$ (see~\eqref{eq:def_C_k}), and use the inequality~\eqref{eq:ineq_Gaussian_A_2} to write for any $ (u_1,u_2) \in \rset^{2d}$, $(\tzbf,\bfw) \in \rset^{n_0 d}\times \msw^{n_0}$,   $\gamma \in \ocint{0,\bgamma_{t_0}}$,
\begin{equation}
A^{(2)}_{{\gamma,x,v,\tzbf,\bfw}}(u_1,u_2) \geq \exp\parenthese{-\frac{1}{\sigma_{\gamma}^2}\norm{\left(\cminorc{\lfloor t_0/\gamma \rfloor+1}\right)^{-\frac{1}{2}} \parenthese{\bGammaminorinvc{\gamma,x,v,\tzbf,\bfw}[\lfloor t_0/\gamma \rfloor+1](u_1,u_2)-\bGammaminorinvc{\gamma,\Zd,\Zd,\tzbf,\bfw}[\lfloor t_0/\gamma \rfloor+1](u_1,u_2)}}^2}.
\end{equation}
Next, by \Cref{propo:inverse_lip_gamma_initial_condition}-\ref{item:2:propo:inverse_lip_gamma_initial_condition} and \Cref{lem:covariance_matrix_minorization_conditions_simple_gen},
\begin{align}
  \liminf_{\widetilde{\gamma}\downarrow 0} A^{(2)}_{\widetilde{\gamma},x,v,\tzbf,\bfw}(u_1,u_2)  & \geq \exp\parenthese{- 4 \rme^{2(1+2\Ltt)t_0}\normop{\parentheseDeux{\continueMatrix{t_0}}^{-1}} \left(\norm{x}+ \norm{v}\right)^2}\\
  &  \geq \exp\parenthese{-4 \rme^{2(1+2\Ltt)t_0} \normop{\parentheseDeux{\continueMatrix{t_0}}^{-1}} M^2} = \eta_{t_0,M} \eqsp,
\end{align}
where $\continueMatrix{t_0}$ is the matrix defined in~\eqref{eq:def_continue_matrix}. Upon reducing~$\bar{\gamma}_{t_0}$, we can therefore assume that the following inequality holds: for any $ (u_1,u_2) \in \rset^{2d}$, $(\tzbf,\bfw) \in \rset^{n_0 d}\times \msw^{n_0}$,   $\gamma \in \ocint{0,\bgamma_{t_0}}$,
\begin{equation}
  \label{eq:bound_A_2_gen} A^{(2)}_{{\gamma},x,v,\tzbf,\bfw}(u_1,u_2) \geq \frac{\eta_{t_0,M}}{2}.
\end{equation}

Now that the factor~$A_{{\gamma,x,v,\tzbf,\bfw}}(u_1,u_2)$ is bounded from below, we can consider the remaining terms in~\eqref{eq:proof_small_set_mino_decomp_1_gen} in order to construct a reference minorization measure. More precisely, using again \Cref{propo:inverse_lip_Gamma}-\ref{item:3:propo:Lip_gamma_i_k_lip_xi_k}, the change of variable $(g_1,g_2)= \sqrt{2} \, \bGammaminorinvc{\gamma,\Zd,\Zd,\tzbf,\bfw}[\lfloor t_0/\gamma \rfloor+1](u_1,u_2)$ implies, for any  $\gamma \in \ocint{0,\bgamma_{t_0}}$,
\begin{align}
  &\int_{\rset^{n_0 d}\times \msw^{n_0}} \hspace{-0.6cm} \1_{\msb}(u_1,u_2) \JacD_{\bGammaminorinvc{\gamma,\Zd,\Zd,\tzbf,\bfw}[\lfloor t_0/\gamma \rfloor+1]}(u_1,u_2)  \varphi_{t_0,\gamma}\left(\sqrt{2}\,\bGammaminorinvc{\gamma,\Zd,\Zd,\tzbf,\bfw}[\lfloor t_0/\gamma \rfloor+1](u_1,u_2)\right)\psi_{t_0,\gamma}(\tzbf) \, \rmd \tzbf \, \wdensity{\otimes \lfloor t_0/\gamma \rfloor+1}(\rmd\bfw) \, \rmd u_1 \, \rmd u_2\\
  &\qquad = 2^{-d} \int_{\rset^{n_0 d}\times \msw^{n_0}} \1_{\msb}\left\{\bGammaminorc{\gamma,\Zd,\Zd,\tzbf,\bfw}[\lfloor t_0/\gamma \rfloor+1]\left(\frac{\gminor_1}{\sqrt{2}},\frac{\gminor_2}{\sqrt{2}}\right)\right\} \varphi_{t_0,\gamma}(\gminor_1,\gminor_2)\psi_{t_0,\gamma}(\tzbf) \, \rmd \tzbf \, \wdensity{\otimes \lfloor t_0/\gamma \rfloor+1}(\rmd\bfw) \, \rmd \gminor_1 \, \rmd \gminor_2\\
&\qquad = 2^{-d}\mu_{t_0,\gamma}(\msb)\eqsp, \label{eq:constructed_minorization_measure}
\end{align}
where $\mu_{t_0,\gamma}$ is a probability measure on~$(\rset^{2d}, \mcb{\rset^{2d}})$. By combining the latter inequality and the lower bounds~\eqref{eq:bound_A_1_gen} and~\eqref{eq:bound_A_2_gen} in~\eqref{eq:proof_small_set_mino_decomp_1_gen}, we obtain the following result: for any  $\gamma \in \ocint{0,\bgamma_{t_0}}$ and $(x,v) \in \rset^{2d}$ such that~$\norm{x}+\norm{v} \leq M$,
\begin{equation}
 \Rkerminorc{\gamma}^{\lfloor t_0/\gamma \rfloor+1}((x,v),\msb) \geq \varepsilon_{t_0,M} \mu_{t_0,\gamma}(\msb) \eqsp, 
\end{equation}
with $\varepsilon_{t_0,M} = 6^{-d} \eta_{t_0,M}/2$. The inequality~\eqref{eq:theo:minor_main_gen} then follows from~\cite[Lemma 18.2.7]{douc:moulines:priouret:2018} for instance. 
\end{proof}

We conclude this section by recalling (and proving for completeness) a well-known result on the Jacobian determinant of~$\rmC^1$ Lipschitz function.

\begin{proposition}
  \label{propo:bound_det_Lipsc_map}
  Let $\Phi : \rset^d \to \rset^d$ be a continuously differentiable $L$-Lipschitz function, \ie~for any $x,y \in \rset^d$,
  \begin{equation}
\norm{\Phi(x)-\Phi(y)} \leq L\norm{x-y} \eqsp.
  \end{equation}
  Then, $\JacD_{\Phi}(x) \leq L^d$, where $\JacD_{\Phi}$ denotes the absolute value of the determinant of the Jacobian matrix of $\Phi$.
\end{proposition}

\begin{proof}
  Denote by $\JacM_{\Phi} : \rset^{d} \to \rset^{d \times d}$ the Jacobian matrix of~$\Phi$. Since $\Phi$ is $L$-Lipschitz, it holds, for any $(x,h) \in \rset^d \times \rset^d$, $   \norm{\JacM_{\Phi}(x)h} \leq L \norm{h}$.
  By choosing $h=\mathbf{e}_i$ for $i \in\{1,\ldots,d\}$ (with $\mathbf{e}_1, \dots, \mathbf{e}_d$ the canonical basis of $\rset^d$), it follows that
  \begin{equation}
  \left\|\nabla \Phi_i \right\|^2 = \sum_{j=1}^d \left(\frac{\partial \Phi_i}{\partial x_j}\right)^2 \leq L^2,
  \end{equation}
  where $\Phi_i$ is the $i$-th component of $\Phi$. In addition, using Hadamard's inequality (see for instance~\cite[Example 4.18]{hiai:petz:2014}), we obtain~$\left|\det(\JacM_{\Phi}(x))\right| \leq \prod_{i=1}^d \left\|\nabla \Phi_i \right\| \leq L^{d}$, which completes the proof.
\end{proof}

\section{Proof of \Cref{theo:drift_exp}}
\label{sec:proof-crefth_drift}

Let us start by recalling the expression of the function~$\Lya:\rset^{2d}\to \rset$, introduced in~\eqref{eq:def_function_lyap_super_quad_init} and defined under \Cref{ass:gminorc} and \Cref{ass:contoleGAndF}(U) for~$x,v\in \rset^d$ and $\gamma\in \ocint{0,\bar{\gamma}}$:
\begin{equation}
  \label{eq:def_function_lyap_super_quad}
  \Lya(x,v)= \frac{\kappa^2}{2} \norm{x}^2 + \norm{v}^2 + \frac{\kappa^2\gamma (1+\gamma^\delta \vartheta_\gamma)}{1-\gminorc} \ps{x}{v} + 2\parametreGradU U(x) \eqsp.
\end{equation}
This function is indeed nonnegative by \Cref{lem:minoration_lya} below. The proof of~\Cref{theo:drift_exp} relies on a conditioned Lyapunov drift inequality for~$\Lya$ (see Section~\ref{sec:conditioned_drift}), which motivates the choice of the various prefactors in~\eqref{eq:def_function_lyap_super_quad} (see the discussion at the end of the proof of Lemma~\ref{lem:big_calc_drift}). This allows to write the proof of \Cref{theo:drift_exp} in \Cref{sec:proof_thm2}, with the help of some technical results postponed to \Cref{sec:supporting_lemmas}.

\subsection{Conditioned drift inequality}
\label{sec:conditioned_drift}

The following Lyapunov inequality is a key result to prove~\Cref{theo:drift_exp}. Recall that the function~$\fonctionControle : \rset^{3d} \times \rset^{m_1+m_2} \to \rset_+$ is defined in~\eqref{eq:5}.

\begin{lemma}
  \label{lem:big_calc_drift}
  Assume that~\Cref{ass:gminorc}, \Cref{ass:lip}(U), \Cref{ass:contoleGAndF}(U) and \Cref{ass:w} hold. For any $x,v \in \rset^{d}$, $w\in \rset^{m_1+m_2}$ and $\gamma \in \ocint{0,\bar{\gamma}}$, set $(X_1^{x,v,w},V_1^{x,v,w}) =  \Gammaminorc{\gamma}\left(x,v,\sqrt{\gamma}\sigma_{\gamma} Z, w\right)$, where $\Gammaminorc{\gamma}$ is defined by \eqref{minor_def_Gamma_b_gamma} and $Z$ is a $d$-dimensional standard Gaussian random variable. Then, there exists $C \geq 0$ such that, for any~$\gamma\in \ocint{0,\bgamma\wedge 1}$, $x,v\in \rset^d$ and $w=(w_1,w_2)\in \rset^{m_1+m_2}$,
  \begin{align}
    \expe{\Lya(\XminorTer_1^{x,v,w},\VminorTer_1^{x,v,w})}\leq  \Lya(x,v)& -\gamma\kappa\norm{v}^2-\gamma\kappa\constasssuperexpone\parentheseDeux{ \frac{\norm{\nabla U(x)}^2}{L^2}+\norm{x}}\\
    & +C\gamma\left[1+\norm{\nabla U(x)}+\norm{v} + \gamma^{\deltau/2}  \fonctionControleBis(x,v,w)+\gamma^{\deltau}\norm{x}\norm{w_1}\right] \eqsp,
  \end{align}
where the expectation is over the realizations of~$Z$, and
$  \fonctionControleBis(x,v,w) = \norm{\nabla U(x)}^2/L^2+\norm{v}^2+\norm{w}^2+\norm{x}$.
\end{lemma}

\begin{proof}
  The finite nonnegative constant~$C$ in this proof may change from line to line but does not depend on~$\gamma$, $x$, $v$ or $w$. For ease of notation, we also simply denote~$(X_1^{x,v,w},V_1^{x,v,w})$ by~$(X_1,V_1)$, and do not explicitly indicate that~$\gamma \in \ocint{0,\bar{\gamma}\wedge 1}$. By definition,
  \begin{equation}
    \label{eq:lem_drift_w_0}
    \expe{\Lya(\XminorTer_1,\VminorTer_1)}=2\parametreGradU\expe{U(\XminorTer_1)}+\frac{\kappa^2}{2}\expe{\norm{\XminorTer_1}^2}+\expe{\norm{\VminorTer_1}^2}+\frac{\kappa^2\gamma(\gamma^\delta \vartheta_\gamma+1)}{1-\gminorc}\expe{\ps{\XminorTer_1}{\VminorTer_1}}\eqsp.
  \end{equation}
  We successively bound each expectation on the right-hand side of the previous equality. To bound~$\expe{U(\XminorTer_1)}$, we make use of the following estimates, which are straightforwardly obtained from~\Cref{ass:contoleGAndF}(U) and the equality~$\expe{\Fpzc(x,v,\sqrt{\gamma}\sigma_{\gamma} Z,w)} = \fonctionControleBis(x,v,w) + \gamma\sigma_\gamma^2 d$:
  \begin{equation}
    \label{eq:drift_proof_lem_f}
    \begin{aligned}
      \expe{\norm{f_\gamma\left(x,\gamma^{\delta}v,\gamma^{\delta+\half} \sigma_{\gamma} Z,w\right)}^2} &\leq C\left[1+\gamma^{\deltau}\fonctionControleBis(x,v,w)\right], \\
      \expe{\norm{f_\gamma\left(x,\gamma^{\delta}v,\gamma^{\delta+\half} \sigma_{\gamma} Z,w\right)}} & \leq C\left[1+ \gamma^{\deltau/2}\sqrt{\fonctionControleBis(x,v,w)} \right],
    \end{aligned}
  \end{equation}
the second inequality being obtained from the first one by a Cauchy--Schwarz inequality and the bound $\sqrt{a+b} \leq \sqrt{a}+\sqrt{b}$ for~$a,b \geq 0$. Using~\cite[Lemma 1.2.3]{nesterov:2004} and~\Cref{ass:lip}(U), we obtain that
  \begin{align}
    \expe{U(\XminorTer_1)}&=\expe{U\parenthese{x + \gamma v + \gamma f_\gamma\left(x,\gamma^{\delta}v,\gamma^{\delta+\half} \sigma_{\gamma} Z,w\right) + \gamma^{\delta+\half} \sigma_{\gamma} \Dbf_{\gamma} Z}}\\
    &\leq U(x)+\expe{\ps{\nabla U(x)}{\gamma v + \gamma f_\gamma\left(x,\gamma^{\delta}v,\gamma^{\delta+\half} \sigma_{\gamma} Z,w\right)}}\\
    &\qquad + \frac{L}{2}\expe{\norm{\gamma v + \gamma f_\gamma\left(x,\gamma^{\delta}v,\gamma^{\delta+\half} \sigma_{\gamma} Z,w\right)+ \gamma^{\delta+\half} \sigma_{\gamma} \Dbf_{\gamma} Z}^2}\\
    &\leq U(x)+\gamma \ps{\nabla U(x)}{v}+\gamma \norm{\nabla U(x)}\expe{\norm{f_\gamma\left(x,\gamma^{\delta}v,\gamma^{\delta+\half} \sigma_{\gamma} Z,w\right)}}\\
    &\qquad + \frac{3 L}{2} \left( \gamma^2 \norm{v}^2 + \gamma^2 \expe{\norm{f_\gamma\left(x,\gamma^{\delta}v,\gamma^{\delta+\half} \sigma_{\gamma} Z,w\right)}^2} + \gamma^{1+2\delta}\sigma_\gamma^2 \expe{\norm{\Dbf_{\gamma} Z}^2} \right).
  \end{align}
  Note first that, in view of~\eqref{eq:drift_proof_lem_f} and~\Cref{ass:gminorc}, and since~$\norm{v}^2 \leq \fonctionControleBis(x,v,w)$,
  \[
  \gamma^2 \norm{v}^2 + \gamma^2 \expe{\norm{f_\gamma\left(x,\gamma^{\delta}v,\gamma^{\delta+\half} \sigma_{\gamma} Z,w\right)}^2} + \gamma^{1+2\delta}\sigma_\gamma^2 \expe{\norm{\Dbf_{\gamma} Z}^2} \leq C\left[\gamma^{1+(2\delta)\wedge 1}+\gamma^2\fonctionControleBis(x,v,w)\right].
  \]
  Moreover, still with~\eqref{eq:drift_proof_lem_f}, and since~$\norm{\nabla U(x)} \leq L \sqrt{\fonctionControleBis(x,v,w)}$,
  \begin{align}
    \norm{\nabla U(x)}\expe{\norm{f_\gamma\left(x,\gamma^{\delta}v,\gamma^{\delta+\half} \sigma_{\gamma} Z,w\right)}} & \leq C\norm{\nabla U(x)}\left[1+ \gamma^{\deltau/2}\sqrt{\fonctionControleBis(x,v,w)} \right], \\
    & \leq C\left[ \norm{\nabla U(x)} + \gamma^{\deltau/2} \fonctionControleBis(x,v,w) \right]. \label{eq:bound_difficult_term_drift_U}
  \end{align}
  Since~$\deltau \leq 1$, this leads finally to
  \begin{equation}
    \label{eq:lem_drift_w_1}
    \expe{U(\XminorTer_1)} \leq U(x)+\gamma \ps{\nabla U(x)}{v} +C\gamma\left[\gamma^{(2\delta) \wedge 1}+\norm{\nabla U(x)}+\gamma^{\deltau/2}\fonctionControleBis(x,v,w)\right]\eqsp .
  \end{equation}
  
  Let us next bound~$\expe{\norm{\XminorTer_1}^2}$. Note first that~\Cref{ass:contoleGAndF}(U) implies that 
  \begin{align}
  \label{eq:8}
  \expe{    \ps{x}{f_\gamma\left(x,\gamma^{\delta}v, \gamma^{\delta+\half} \sigma_{\gamma} Z,w\right)}}&\leq \gamma^\delta\vartheta_\gamma\ps{x}{v}\\
  &\quad+\constgGradU\parentheseDeux{1+\gamma^{\deltau}\fonctionControleBis\left(x,v,w\right)+ \gamma^{1+2\delta+\deltau} \sigma_{\gamma}^2 d+\gamma^{\deltau}\norm{x}\norm{w_1}} \eqsp,
  \end{align}
  so that, with \Cref{ass:gminorc} and \Cref{ass:contoleGAndF}(U),  
  \begin{align}
    \expe{\norm{\XminorTer_1}^2}&\leq \norm{x}^2+\gamma^2 \norm{v}^2+\gamma^2\expe{\norm{f_\gamma\left(x,\gamma^{\delta}v,\gamma^{\delta+\half} \sigma_{\gamma} Z,w\right)}^2}+\gamma^{1+2\delta}\sigma_\gamma^2\mathscr{D}^2d+2\gamma\ps{x}{v}\\
    &\qquad+2\gamma\expe{\ps{x}{f_\gamma\left(x,\gamma^{\delta}v,\gamma^{\delta+\half} \sigma_{\gamma} Z,w\right)}}+2\gamma^2\expe{\ps{v}{f_\gamma\left(x,\gamma^{\delta}v,\gamma^{\delta+\half} \sigma_{\gamma} Z,w\right)}}\\
    &\qquad+ 2\gamma^{3/2+\delta}\sigma_\gamma\expe{\ps{f_\gamma\left(x,\gamma^{\delta}v,\gamma^{\delta+\half} \sigma_{\gamma} Z,w\right)}{\Dbf_{\gamma} Z}}\\
    &\leq \norm{x}^2+2\gamma\left(1+\gamma^\delta\vartheta_\gamma\right)\ps{x}{v}+C\gamma\left[1+ \gamma^{\deltau}\fonctionControleBis(x,v,w) +\gamma^{\deltau}\norm{x}\norm{w_1}\right]\eqsp.
    \label{eq:lem_drift_w_2}
\end{align}
  
  To bound~$\expe{\norm{\VminorTer_1}^2}$, we rely on the following estimates, obtained from~\Cref{ass:contoleGAndF}(U):  
  \begin{equation}
  \label{eq:drift_proof_lem_g}
\begin{aligned}
  & \expe{     \norm{g_\gamma\left(x,\gamma^{\delta}v, \gamma^{\delta+\half} \sigma_{\gamma} Z,w\right)+\parametreGradU\nabla U(x)}^2} \leq C\parenthese{1+\gamma^{\deltau}\fonctionControleBis\left(x,v,w\right)}, \\
  & \expe{     \norm{g_\gamma\left(x,\gamma^{\delta}v, \gamma^{\delta+\half} \sigma_{\gamma} Z,w\right)+\parametreGradU\nabla U(x)}} \leq C\parenthese{1+\gamma^{\deltau/2}\sqrt{\fonctionControleBis\left(x,v,w\right)}}, \\
& \expe{    \ps{x}{g_\gamma\left(x,\gamma^{\delta}v, \gamma^{\delta+\half} \sigma_{\gamma} z,w\right)}}\leq -\constasssuperexpone\parentheseDeux{ \frac{\norm{\nabla U(x)}^2}{L^2}+\norm{x}}+ C\gamma^{\deltau}\parenthese{1+\fonctionControleBis\left(x,v,w\right)}.  
  \end{aligned}
\end{equation}
Therefore, 
\begin{align}
  \expe{\norm{\VminorTer_1}^2}& = \gminorc^2\norm{v}^2+\gamma^2\parametreGradU^2\norm{\nabla U(x)}^2+\gamma^2\expe{\norm{g_\gamma\left(x,\gamma^{\delta}v, \gamma^{\delta+\half} \sigma_{\gamma} Z,w\right)+\parametreGradU\nabla U(x)}^2}\\
  &\qquad+\gamma\sigma_\gamma^2 d-2\gminorc\gamma\parametreGradU\ps{v}{\nabla U(x)}\\
  &\qquad+2\gminorc\gamma\ps{v}{\expe{g_\gamma\left(x,\gamma^{\delta}v, \gamma^{\delta+\half} \sigma_{\gamma} Z,w\right)}+\parametreGradU\nabla U(x)}\\
  &\qquad -2\gamma^2\parametreGradU\ps{\nabla U(x)}{\expe{g_\gamma\left(x,\gamma^{\delta}v, \gamma^{\delta+\half} \sigma_{\gamma} Z,w\right)}+\parametreGradU\nabla U(x)}\\
  &\qquad +2\gamma^{3/2}\sigma_\gamma\expe{\ps{g_\gamma\left(x,\gamma^{\delta}v, \gamma^{\delta+\half} \sigma_{\gamma} Z,w\right)+\parametreGradU\nabla U(x)}{Z}}.
\end{align}
The terms on the third and fourth lines of the above series of inequalities can be bounded as in~\eqref{eq:bound_difficult_term_drift_U} using~\Cref{ass:contoleGAndF}(U). For instance, for the one in the third line:
\[
\begin{aligned}
  \left| \ps{v}{\expe{g_\gamma\left(x,\gamma^{\delta}v, \gamma^{\delta+\half} \sigma_{\gamma} Z,w\right)}+\parametreGradU\nabla U(x)}\right|
  & \leq \norm{v} \expe{ \norm{g_\gamma\left(x,\gamma^{\delta}v, \gamma^{\delta+\half} \sigma_{\gamma} Z,w\right)+\parametreGradU\nabla U(x)}} \\
  & \leq C\parentheseDeux{\norm{v} +\gamma^{\deltau/2} \norm{v}\sqrt{\fonctionControleBis\left(x,v,w\right)}}, \\
  & \leq C\parentheseDeux{\norm{v} +\gamma^{\deltau/2} \fonctionControleBis\left(x,v,w\right)}.
\end{aligned}
\]
We thus obtain with \Cref{ass:gminorc} that
\begin{align}
\nonumber
\expe{\norm{\VminorTer_1}^2} &\leq \gminorc^2\norm{v}^2-2\gminorc\gamma\parametreGradU\ps{v}{\nabla U(x)}+ C\gamma\defEns{1+\norm{v} +\gamma^{\deltau/2}\fonctionControleBis(x,v,w)}\\
\label{eq:lem_drift_w_3}
& \leq (1-\kappa \gamma)\norm{v}^2-2\gamma\parametreGradU\ps{v}{\nabla U(x)}+ C\gamma\defEns{1+\norm{v} +\gamma^{\deltau/2}\fonctionControleBis(x,v,w)}
\eqsp,
\end{align}
where we have used for the last inequality that~$|\gminorc-1| \leq $ and $\gminorc^2 = \rme^{-2\kappa\gamma} +\gminorc^2-\rme^{-2\kappa\gamma} \leq 1 - 2\gamma \kappa + 2(\kappa^2 + C_{\kappa})\gamma^2$ by \Cref{lem:lemma_exp_gminorc} (the term of order~$\gamma^2$ going into the remainder term in the inequality). Finally, using \Cref{ass:gminorc} and~\Cref{ass:contoleGAndF}(U), as well as the estimates~\eqref{eq:drift_proof_lem_f} and~\eqref{eq:drift_proof_lem_g},
\begin{align}
  &\expe{\ps{\XminorTer_1}{\VminorTer_1}} = \gminorc\ps{x}{v}+\gamma\expe{\ps{x}{g_\gamma\left(x,\gamma^{\delta}v, \gamma^{\delta+\half} \sigma_{\gamma} Z,w\right)}}+\gamma\gminorc\norm{v}^2\\
  &\qquad +\gamma\gminorc\expe{\ps{f_\gamma\left(x,\gamma^{\delta}v,\gamma^{\delta+\half} \sigma_{\gamma} Z,w\right)}{v}}\\
  &\qquad-\gamma^2\parametreGradU\expe{\ps{f_\gamma\left(x,\gamma^{\delta}v,\gamma^{\delta+\half} \sigma_{\gamma} Z,w\right)}{\nabla U(x)}} \\
  &\qquad +\gamma^2\expe{\ps{f_\gamma\left(x,\gamma^{\delta}v,\gamma^{\delta+\half} \sigma_{\gamma} Z,w\right)}{g_\gamma\left(x,\gamma^{\delta}v, \gamma^{\delta+\half} \sigma_{\gamma} Z,w\right)+\parametreGradU\nabla U(x)}}\\
  &\qquad +\gamma^2\expe{\ps{v}{g_\gamma\left(x,\gamma^{\delta}v, \gamma^{\delta+\half} \sigma_{\gamma} Z,w\right)+\parametreGradU \nabla U(x)}} -\gamma^2\parametreGradU\ps{v}{\nabla U(x)} \\ 
  &\qquad +\gamma^{3/2}\sigma_\gamma \expe{\ps{f_\gamma\left(x,\gamma^{\delta}v,\gamma^{\delta+\half} \sigma_{\gamma} Z,w\right)}{Z}}\\
  &\qquad+ \gamma^{\delta+3/2}\sigma_\gamma \expe{\ps{\Dbf_{\gamma} Z}{g_\gamma\left(x,\gamma^{\delta}v, \gamma^{\delta+\half} \sigma_{\gamma} Z,w\right)+\parametreGradU\nabla U(x)}}\\
  &\qquad + \gamma^{1+\delta}\sigma_\gamma^2\expe{\ps{\Dbf_{\gamma} Z}{Z}}\\
  &\leq \gminorc\ps{x}{v}-\gamma\constasssuperexpone\parentheseDeux{ \frac{\norm{\nabla U(x)}^2}{L^2}+\norm{x}} + \gamma\gminorc\norm{v}^2+C\gamma\defEns{1+\norm{v}+\gamma^{\deltau/2}\fonctionControleBis(x,v,w)}\eqsp. \label{eq:lem_drift_w_4}
\end{align}

We can now come back to~\eqref{eq:lem_drift_w_0}. In order to make apparent our choices of prefactors for the Lyapunov function~\eqref{eq:def_function_lyap_super_quad}, we introduce real numbers~$c_1,c_2,c_3>0$ and~$c_4 \in \rset$, and define
\[
\mathscr{W}_{c}(x,v) = c_1 U(x) + c_2 \norm{x}^2 + c_3 \norm{v}^2 + c_4 \ps{x}{v} \eqsp.
\]
By multiplying~\eqref{eq:lem_drift_w_1} by~$c_1$, \eqref{eq:lem_drift_w_2} by~$c_2$, \eqref{eq:lem_drift_w_3} by~$c_3$ and~\eqref{eq:lem_drift_w_4} by~$c_4$, we obtain  
  \begin{align}
    \expe{\mathscr{W}_{c}(\XminorTer_1,\VminorTer_1)}&\leq \mathscr{W}_{c}(x,v) -c_4\gamma \constasssuperexpone\parentheseDeux{\frac{\norm{\nabla U(x)}^2}{L^2}+\norm{x}} - \gamma \left[2c_3  \kappa - c_4\gminorc\right] \norm{v}^2 \\
    & \qquad + \left[c_4(\gminorc-1)+2c_2\gamma\left(1+\gamma^\delta\vartheta_{\gamma}\right)\right]\ps{x}{v} + \gamma\left[c_1-2c_3\parametreGradU\right] \ps{v}{\nabla U(x)}\\
                                                           & \qquad + C\gamma \left[ 1 + \norm{\nabla U(x)} + \norm{v}+\gamma^{\deltau}\norm{x}\norm{w_1} + \gamma^{\deltau/2}  \fonctionControleBis(x,v,w) \right].
  \end{align}
  We choose $c_1,c_2,c_3,c_4$ to cancel the prefactors of the scalar products in the second line and to ensure that the term $-\left[2 \kappa c_3  - c_4\gminorc\right] \norm{v}^2 \leq - \tilde{c} \norm{v}^2 + C \gamma^2 \fonctionControleBis(x,v,w)$ for some $\tilde{c} >0$. Such requirements can be obtained with $\tilde{c} = \kappa$ by setting $c_1 = 2 \parametreGradU$, $c_2= \kappa^2/2$, $c_3= 1$ and $c_4 = \kappa^2 \gamma(1+\gamma^{\delta}\vartheta_{\gamma})/(1-\gminorc)$ and using $\abs{c_4 - \kappa} \leq C \gamma$ by \Cref{lem:lemma_exp_gminorc}.
\end{proof}

\subsection{Proof of \Cref{theo:drift_exp}}
\label{sec:proof_thm2}

We are now in position to prove \Cref{theo:drift_exp}. The finite nonnegative constant~$C$ in this proof may change from line to line but does not depend on $\gamma$, $x$, $v$ or $w$. A first idea in the proof is to rewrite the Lyapunov condition to be shown for~$\Lyaexp$ as a Lyapunov condition for the function
\begin{equation}
  \label{eq:def_Lyaexp}
  \LSfonction = \sqrt{1+\Lya} \eqsp,
\end{equation}
thanks to \Cref{ass:w}. 
In all this proof, the timestep~$\gamma$ belongs to~$\ocintLigne{0,\cbgamma}$ with~$\cbgamma >0$ defined below in~\eqref{eq:def_c}. We also denote by 
$
(X_1,V_1) = \Gammaminorc{\gamma}(x,v,\sqrt{\gamma}\sigma_{\gamma} Z, W)
$,
the output of one step of the Markov chain starting from a given configuration~$x,v\in \rset^d$, with~$Z$ a $d$-dimensional standard Gaussian random variable and $W=(W_1,W_2)$ a random variable independent of $Z$ with distribution~$\muw=\muwOne\otimes\muwTwo$ (recall that~$\Gammaminorc{\gamma}$ is defined by \eqref{minor_def_Gamma_b_gamma}).

\sloppy We start by bounding~$\Rkerminorc{\gamma}\Lyaexp$ in terms of~$\CPE{\LSfonction(X_1,V_1)}{W_2}$. The first observation is that, in view of Lemmas~\ref{lem:phi_lip} and~\ref{lem:Gammaminorc_lip}, there exists a constant~$\mathscr{L} \in \rset_+$ such that the function $(z,w_1) \mapsto \LSfonction(\Gammaminorc{\gamma}(x,v,\sqrt{\gamma}\sigma_{\gamma} z, (w_1,w_2)))$ is Lipschitz with Lipschitz constant~$\sqrt{\mathscr{L}\gamma}$, uniformy in~$x,v$ and~$w_2$. Therefore, by \Cref{ass:w}-\ref{ass:w_logsob_spec_1}-\ref{ass:w_logsob_spec}, for any $\varpi\in \rset_+^*$,
\begin{align}
  \Rkerminorc{\gamma}\Lyaexp(x,v) &=\expe{\CPE{\Lyaexp(X_1,V_1)}{W_2}} \leq \expe{\exp\parenthese{\varpi\CPE{\LSfonction(X_1,V_1)}{W_2}+\frac{\varpi^2 \mathscr{L}\gamma}{2} }}\eqsp .
  \label{eq:first_inequality_proof_theorem_drift}
\end{align}

The next step of the proof is to obtain bounds on~$\CPE{\LSfonction(X_1,V_1)}{W_2}$. We first relate this quantity to~$\CPE{\Lya(X_1,V_1)}{W_2}$ and then rely on \Cref{lem:big_calc_drift}. More precisely, using the bound~$\sqrt{1+t}\leq 1+t/2$ for~$t \geq -1$,
\begin{align}
  \frac{\CPE{\LSfonction(X_1,V_1)}{W_2}}{\LSfonction(x,v)} &= \CPE{\sqrt{1+\frac{\Lya(X_1,V_1)-\Lya(x,v)}{\LSfonction^2(x,v)}}}{W_2}\\
  \label{eq:sqrt_phi_drift_1}
  & \leq \CPE{1+\frac{\Lya(X_1,V_1)-\Lya(x,v)}{2\LSfonction^2(x,v)}}{W_2} = 1 + \frac{\CPE{\Lya(X_1,V_1)}{W_2}-\Lya(x,v)}{2\LSfonction(x,v)^2}. 
\end{align}
Since $W_1$ admits a first moment by \Cref{ass:w}, using that $a \leq (2\varepsilon)^{-1} + \varepsilon a^2 /2$ for any $a,\varepsilon >0$, there exists $\bar{\gamma}_U^{(1)}\in \ocintLigne{0,\cbgamma}$ such that for any $\gamma\in \ocintLigne{0,\bar{\gamma}_U^{(1)}}$,
\begin{align}
  &C\CPE{1+\norm{\nabla U(x)}+\norm{v} + \gamma^{\deltau/2}  \fonctionControleBis(x,v,W)+\gamma^{\deltau}\norm{x}\norm{W_1}}{W_2} \\
  &\qquad\qquad\qquad\qquad\qquad\qquad\qquad\qquad\qquad\leq C(1+\gamma^{\deltau/2}\norm{W_2}^2) + \frac{\kappa}{2}\left[ \norm{v}^2 + \constasssuperexpone \left(\frac{\norm{\nabla U(x)}^2}{L^2} + \norm{x} \right)\right]  \eqsp.
\end{align}
Therefore, by \Cref{lem:big_calc_drift} and since~$-\norm{v}^2\leq -\norm{v}+1$,
\begin{equation}
  \begin{aligned}
  \CPE{\Lya(X_1,V_1)}{W_2} & \leq \Lya(x,v)-\frac{\gamma\kappa}{2} \left( \norm{v}^2 + \constasssuperexpone\norm{x}\right)+ C\gamma (1+\norm{W_2}^2) \\
  & \leq \Lya(x,v)-\frac{\gamma\kappa\min(1, \constasssuperexpone)}{2} \left( \norm{v} +\norm{x}\right)+ C\gamma (1+\norm{W_2}^2)\eqsp.
  \end{aligned}
\end{equation}
Plugging this estimate in \eqref{eq:sqrt_phi_drift_1} and using Lemmas~\ref{lem:majoration_lsfonction} and~\ref{lem:minoration_lya}, as well as the inequality $\sqrt{a+b} \geq 2^{-1}(\sqrt{a} +\sqrt{b})$ for~$a,b \geq 0$, we obtain
\begin{align}
  \CPE{\LSfonction(X_1,V_1)}{W_2}
  & \leq \LSfonction(x,v)+\gamma\frac{-\kappa\min(1,\constasssuperexpone)(\norm{x}+\norm{v}) + C(1+\norm{W_2}^2)}{4\LSfonction(x,v)} \\
  & \leq \LSfonction(x,v)-\gamma\frac{\kappa\min(1,\constasssuperexpone)(\norm{x}+\norm{v})}{4[1+\cPhi\parenthese{\norm{x}+\norm{v}}] } + C\gamma \frac{1+\norm{W_2}^2}{1 + \sqrt{\cLya/2}(\norm{x} +\norm{v})} \eqsp. \label{eq:second_inequality_proof_theorem_drift}
\end{align}
Then, by~\eqref{eq:first_inequality_proof_theorem_drift} and~\eqref{eq:second_inequality_proof_theorem_drift}, for any $\gamma\in \ocintLigne{0,\bar{\gamma}_U^{(1)}}$, $\varpi >0$, and $x,v\in\rset^d$,
\begin{equation}
  \label{eq:10}
  \frac{\Rkerminorc{\gamma}\Lyaexp(x,v)}{\Lyaexp(x,v)}
  \leq \expe{\exp\left(\frac{\varpi^2 \mathscr{L}\gamma}{2} + \varpi C\gamma \frac{1+\norm{W_2}^2}{1 + \sqrt{\cLya/2}(\norm{x} +\norm{v})}\right)}\exp\left(-\gamma\varpi\frac{\kappa\min(1,\constasssuperexpone)(\norm{x}+\norm{v})}{ 4[1+\cPhi\parenthese{\norm{x}+\norm{v}}] }\right) .
\end{equation} 
We now  choose successively~$K_U>0$ sufficiently large and~$\varpi>0$ sufficiently small so that the Markov chain induces a contraction for the Lyapunov function~$\Lyaexp$ on the set~$\{(x,v)\in\rset^{2d}, \norm{x} + \norm{v} \geq K_U \}$. We first need to this end a bound on exponential moments of~$W_2$. By Jensen's inequality and~\Cref{ass:w}, for any~$c>0$ and~$\gamma\in\ocintLigne{0,\bgamma_W/c}$, setting $C_W = \expe{\rme^{\bgamma_W\norm{W_2}^2}}$,
\begin{equation}
  \label{eq:controle_moment_proof_theorem_drift}
  \expe{\rme^{c\gamma\norm{W_2}^2}}=\expe{\left(\rme^{\bgamma_W\norm{W_2}^2}\right)^{c\gamma/\bgamma_W}}\leq    C_W^{c\gamma/\bgamma_{W}} \eqsp.
\end{equation}

For any  $\varpi >0$, $\gamma\in \ocintLigne{0,\bar{\gamma}_U^{(1)}}$ such that $\varpi \gamma \leq \bgamma_{W}(1 + \sqrt{\cLya/2}K_U)/C$ and  $K_U > 0$, $x,v\in\rset^d$ with~$\norm{x} + \norm{v} \geq K_U$, it therefore holds by \eqref{eq:10},
\begin{equation}
  \frac{1}{\varpi\gamma}\log\left( \frac{\Rkerminorc{\gamma}\Lyaexp(x,v)}{\Lyaexp(x,v)} \right)
 \leq \frac{\varpi \mathscr{L}}{2} + \frac{C}{1 + \sqrt{\cLya/2}K_U}\left(1+\frac{\log C_W}{\cbgamma}\right) - \frac{\kappa\min(1,\constasssuperexpone)K_U}{ 4[1+\cPhi K_U] } \eqsp.
\end{equation}
Setting 
\[
  K_U = 1 \vee \parentheseDeux{\frac{ 8[1+\cPhi ]
    }{\kappa\min(1,\constasssuperexpone)}\left(1+\frac{\log
        C_W}{\cbgamma}\right) \sqrt{2/\cLya}} \eqsp, \qquad \varpi_U =
  \frac{\kappa\min(1,\constasssuperexpone)K_U}{ 8\mathscr{L}[1+\cPhi
    K_U] } \eqsp, 
\]
and $\bar{\gamma}_U^{(2)} = \cbgamma[1\wedge\{(1 + \sqrt{\cLya/2}K_U)/C\varpi_U\}]$,
it follows that, for any~$\gamma \leq \min(\bar{\gamma}_U^{(1)},\bar{\gamma}_U^{(2)})$ and $(x,v)\in\rset^{2d}$ with~$\norm{x} + \norm{v} \geq K_U$,
\begin{equation}
  \label{eq:11}
\Rkerminorc{\gamma}\LyaexpU(x,v) \leq \lambda_U^\gamma \LyaexpU(x,v) \eqsp, \qquad \lambda_U = \exp\left(-\frac{\kappa\min(1,\constasssuperexpone)K_U}{ 16[1+\cPhi K_U] }\right) < 1 \eqsp.
\end{equation}

We finally consider the case $\norm{x}+\norm{v}< K_U$.
First, we note that~\eqref{eq:second_inequality_proof_theorem_drift} implies that for any $\gamma \leq \gamma_U^{(1)}$, $  \CPE{\LSfonction(X_1,V_1)}{W_2}\leq \LSfonction(x,v) + C\gamma(1+\norm{W_2}^2)$.
By plugging this result in~\eqref{eq:first_inequality_proof_theorem_drift}, we obtain
\[
\Rkerminorc{\gamma}\LyaexpU(x,v) \leq \expe{\exp\left(\frac{\gamma\varpi_U}{2}\left[\varpi_U\mathscr{L}+2C(1+\norm{W_2}^2)\right]\right)} \LyaexpU(x,v) \eqsp .
\]
 When~$\gamma \leq \bar{\gamma}_U= \min(\bar{\gamma}_U^{(1)},\bar{\gamma}_U^{(2)},\bar{\gamma}_U^{(3)})$, with $\bar{\gamma}_U^{(3)} = \cbgamma/(\varpi_U C)$,
 the right-hand side can be bounded by~\eqref{eq:controle_moment_proof_theorem_drift} as
\[
\Rkerminorc{\gamma}\LyaexpU(x,v) \leq \mathrm{e}^{\mathscr{K}\gamma}\LyaexpU(x,v), \qquad \mathscr{K} = \frac{\varpi_U}{2}\left[\varpi_U\mathscr{L}+2C\left(1+\frac{\log C_W}{\cbgamma}\right)\right] \eqsp .
\]
We can therefore write, in view of the inequality $\rme^{t}-\rme^{s} \leq (t-s) \rme^{t}$ for~$t,s \in \rset$, $s \leq t$, for any $x,v \in\rset^d$ with $\norm{x} + \norm{v} < K_U$, and $\gamma \leq \bar{\gamma}_U$,
\[
\begin{aligned}
  \Rkerminorc{\gamma}\Lyaexp(x,v)
  & \leq \lambda_U^{\gamma}\Lyaexp(x,v) + \parenthese{\mathrm{e}^{\mathscr{K}\gamma}-\lambda_U^{\gamma}} \Lyaexp(x,v) \\
  & \leq \lambda_U^{\gamma}\Lyaexp(x,v) + \gamma(\mathscr{K}-\log \lambda_U)\mathrm{e}^{\mathscr{K}\gamma} \Lyaexp(x,v)\eqsp.
\end{aligned}
\]
We finally obtain the following bound for any $x,v \in\rset^d$, $\norm{x} + \norm{v} < K_U$, $\gamma \leq \bar{\gamma}_U$, using \Cref{lem:majoration_lsfonction}:
\[
\Rkerminorc{\gamma}\Lyaexp(x,v) \leq \lambda_U^{\gamma}\Lyaexp(x,v) + \gamma b_U,
\qquad
b_U = (\mathscr{K}-\log \lambda_U)\mathrm{e}^{\mathscr{K} \bar{\gamma}_U + \varpi_U(1+\cPhi K_U)} \eqsp.
\]
Combining this bound with \eqref{eq:11} completes the proof.

\subsection{Supporting lemmas}
\label{sec:supporting_lemmas}

The proofs of the following technical lemmas are postponed to the \supplementname. The first technical result ensures that~$\Lya$ is non-negative.

\begin{lemma}
  \label{lem:minoration_lya}
  Assume \Cref{ass:gminorc} and \Cref{ass:contoleGAndF}(U). Then, for any $x,v \in \rset^d$ and $\gamma \in \ocintLigne{0,\cbgamma}$ with
  \begin{equation}
    \label{eq:def_c}
    \cbgamma = \min\left( 1,\bgamma, \parenthese{\frac{\cLya}{\kappa\bvartheta + (1+\bvartheta)(2 C_{\kappa} +\kappa^2)}}^{(\delta\wedge 1)^{-1}} \right),
    \qquad
    \cLya = \frac12\min\left(\frac{\kappa^2}{6},\frac{1}{4}\right),
  \end{equation}
  it holds $    \Lya(x,v) \geq \cLya ( \norm{x}^2 +\norm{v}^2) + 2 \parametreGradU U(x)$,
  where $\Lya$ is defined in~\eqref{eq:def_function_lyap_super_quad}.
\end{lemma}

The next result provides an upper bound on~$\LSfonction$, itself obtained from an upper bound on~$\Lya$.

\begin{lemma}
  \label{lem:majoration_lsfonction}
  Assume \Cref{ass:gminorc}, \Cref{ass:lip}(U) and \Cref{ass:contoleGAndF}(U). Then, for any $\gamma\in \ocintLigne{0,\min(1,\bgamma)}$ and $x,v\in \rset^d$,
    \begin{equation}
    \label{eq:upper_bound_QF_Lya}
    \Lya(x,v)\leq \cPhi^2(\norm{x}^2+\norm{v}^2) \eqsp, \qquad     \LSfonction(x,v)\leq 1+\cPhi\parenthese{\norm{x}+\norm{v}}\eqsp,
  \end{equation}
  where
  \begin{equation}
    \cPhi = \sqrt{\max\left(1,\frac{\kappa^2}{2} + \parametreGradU L\right) + \frac12(1+\bvartheta)(\kappa^2+\kappa+2C_\kappa)}\eqsp.  
  \end{equation}
\end{lemma}

The last two results provide Lipschitz bounds which allow to obtain~\eqref{eq:first_inequality_proof_theorem_drift}.

\begin{lemma}
  \label{lem:phi_lip}
  Assume that \Cref{ass:gminorc}, \Cref{ass:lip}(U) and \Cref{ass:contoleGAndF}(U) hold. For any $\gamma\in \ocintLigne{0,\cbgamma}$, the fonction $\LSfonction\in \rmC^1(\rset^{2d},\rset)$ defined by~\eqref{eq:def_Lyaexp} is Lipschitz continuous, and its Lipschitz constant is uniformly bounded by
  \begin{equation}
    \begin{aligned}
      \mathscr{L}_\phi = \frac{1}{\sqrt{\cLya}} \max\left(2, 2\parametreGradU L + \kappa^2, (1+\bvartheta)(\kappa^2+\kappa+C_\kappa) \right) \eqsp ,
    \end{aligned}
  \end{equation}
  where $\cbgamma,\cLya$ are defined in~\eqref{eq:def_c}.
\end{lemma}

\begin{lemma}
  \label{lem:Gammaminorc_lip}
  Assume that~\Cref{ass:gminorc} and \Cref{ass:w} hold, and recall that~$\Gammaminorc{\gamma}$ is defined in~\eqref{minor_def_Gamma_b_gamma}. Then, for any $\gamma\in \ocintLigne{0,\min(1,\bgamma)}$, $x,v\in\rset^d$ and $w_2\in \rset^m$, the function $(z,w_1) \mapsto \Gammaminorc{\gamma}\left(x,v,\sqrt{\gamma}\sigma_{\gamma} z, (w_1,w_2)\right)$ is Lipschitz continuous, and its Lipschitz constant is bounded by $\sqrt{2\gamma}[2\LLS \max(1,\bsigma)+\mathscr{D}\bsigma+\bsigma]$.
\end{lemma}



\paragraph{Acknowledgments.}
G.S. thanks Matthias Sachs for pointing out a mistake in the proof of~\cite[Lemma~2.8]{leimkuhler:matthews:stoltz:2016}, which triggered this work. The authors also thank the referees for their suggestions to improve the presentation of the manuscript. The work of G.S. is funded in part by the European Research Council (ERC) under the European Union's Horizon 2020 research and innovation programme (grant agreement No~810367), and by the Agence Nationale de la Recherche, under grants ANR-19-CE40-0010-01 (QuAMProcs) and ANR-21-CE40-0006 (SINEQ). A.D. acknowledges support  of the Lagrange Mathematical and Computing Research Center.


\bibliographystyle{abbrv}
\bibliography{../../Bibliography/bibliography}

\tableofcontents
\appendix

\section{Postponed proofs}

\subsection{Proof of \Cref{lem:estimate_coeff_Sigma}}
\label{sec:proof-crefl_coeef_sigma}
  Note that by \eqref{eq:def_continue_matrix}
  and using that for any $t\in \rset_+ $, $t-t^2/2+t^3/6-t^4/24\leq 1-\rme^{-t}\leq t-t^2/2+t^3/6$, we have for any $t_0\in \rset_+ $,
  \begin{equation}
    \sigma^2t_0^3/3-\sigma^2\kappa t_0^4/3\leq \coeffContinueMatrix{1}{t_0}\leq \sigma^2t_0^3/3 +\sigma^2\kappa t_0^4/12 \eqsp .
  \end{equation}
Similarly using that for any $t\in \rset_+ $, $t-t^2/2\leq 1-\rme^{-t}\leq t $, we have for any $t_0\in \rset_+ $,
  \begin{equation}
    \sigma^2t_0^2/2-\sigma^2\kappa t_0^3/2 \leq \coeffContinueMatrix{2}{t_0}\leq \sigma^2t_0^2/2 \eqsp ,
\qquad  
\sigma^2t_0-\sigma^2\kappa t_0^2 \leq \coeffContinueMatrix{3}{t_0}\leq \sigma^2t_0 \eqsp .
  \end{equation}  
  Then, since $1/3-1/4>0$, taking $\btZ$ sufficiently small completes the proof.

  \subsection{Proof of \Cref{lem:lemma_exp_gminorc}}
  \label{sec:proof-lem:lemma_exp_gminorc}
  
  The first estimate is a direct consequence of the following inequality for~$0 \leq a \leq b$ and~$\ell \geq 1$:
\begin{equation}
 0 \leq b^\ell-a^\ell = \ell \int_a^b x^{\ell-1} \, \rmd x \leq (b-a)\ell b^{\ell-1},
\end{equation}
together with the bound $\max(\gminorc,\rme^{-\kappa \gamma}) \leq 1$. The bound on~$|\gminorc - 1|$  follows from the fact that this quantity is bounded by~$1-\rme^{-\kappa \gamma} + C_\kappa \gamma^2$ in view of~\Cref{ass:gminorc}, together with the inequality~$1-\rme^{-\kappa \gamma} \leq \kappa \gamma$. For the final estimate, we write 
\begin{equation}
\left| \frac{\gamma}{1-\gminorc} - \frac1\kappa\right| = \frac{|\gminorc - 1 + \kappa\gamma|}{\kappa(1-\gminorc)} \leq \frac{|\gminorc - \rme^{-\kappa \gamma}| + |\rme^{-\kappa\gamma}- 1 + \kappa\gamma|}{\kappa(1-\gminorc)} \eqsp.
\end{equation}
The first term in the last numerator is bounded by~$C_\kappa \gamma^2$ in view of \Cref{ass:gminorc}. For the second one and the denominator, we use the inequality $-t^2/2 \leq 1-t-\rme^{-t} \leq 0$ for any~$t \geq 0$ to write $|\rme^{-\kappa\gamma}- 1 + \kappa\gamma| \leq \kappa^2\gamma^2/2$ and 
\begin{equation}
\frac{1-\gminorc}{\gamma} \geq \frac{1-\rme^{-\kappa\gamma}-C_\kappa\gamma^2}{\gamma} \geq \kappa - \left(\frac{\kappa^2}{2}+C_{\kappa}\right) \gamma \geq \frac{\kappa}{2}\eqsp,
\end{equation}
where the last inequality follows from the bound~$\gamma \leq \bgamma \leq (\kappa + 2C_{\kappa}/\kappa)^{-1}$ in~\Cref{ass:gminorc}. This finally leads to~\eqref{eq:lem:lemma_exp_gminorc_2}.

  \subsection{Proof of \Cref{lem:minoration_lya}}
  Let $\gamma \in \ocintLigne{0,\cbgamma}$.  Consider~$V_0(x,v) = \kappa^2 \norm{x}^2/2+\norm{v}^2+\kappa\ps{x}{v}+2\parametreGradU U(x)$. The Cauchy--Schwarz inequality and \Cref{lem:lemma_exp_gminorc} give
  \begin{align}
    \label{eq:9}
    \abs{\Lya(x,v) - V_0(x,v) } &\leq \kappa \abs{ \frac{\kappa\gamma(1+\gamma^{\delta}\vartheta_{\gamma})}{1-\tau_{\gamma}} -1} \frac{\norm{x}^2+\norm{v}^2}{2}  \\
    &\leq \kappa\left( \gamma^\delta \vartheta_{\gamma} + (1+\gamma^{\delta}\vartheta_{\gamma})\left|\frac{\kappa\gamma}{1-\tau_{\gamma}} -1\right| \right)\frac{\norm{x}^2+\norm{v}^2}{2} \\
    &  \leq \kappa\left( \gamma^\delta \bvartheta + (1+\gamma^{\delta}\bvartheta) \left[\frac{2C_\kappa}{\kappa} +\kappa\right]\gamma \right)\frac{\norm{x}^2+\norm{v}^2}{2}  \leq \cLya (\norm{x}^2+\norm{v}^2) \eqsp,
  \end{align}
  where the last inequality follows from the definition of~$\cbgamma$. In addition, using the Cauchy--Schwarz inequality again, we get for any~$\eta \in (1/2,1)$, 
  \begin{equation}
      V_0(x,v)  \geq \frac{\kappa^2}{2} (1-\eta)\norm{x}^2 + \left(1-\frac{1}{2\eta}\right) \norm{v}^2 +2 \parametreGradU U(x)  \geq 2\cLya (\norm{x}^2+\norm{v}^2) +2 \parametreGradU U(x) \eqsp,
  \end{equation}
  where the last inequality is obtained with~$\eta=2/3$. The combination of the two previous inequalities finally gives the claimed result.

  \subsection{Proof of \Cref{lem:majoration_lsfonction}}

  By~\cite[Lemma 1.2.3]{nesterov:2004} and~\Cref{ass:lip}(U), it holds $U(x) \leq L\norm{x}^2/2$. Moreover, the Cauchy--Schwarz inequality, \Cref{lem:lemma_exp_gminorc} and the last condition in~\Cref{ass:contoleGAndF}(U) lead to 
  \begin{align}
    \frac{\kappa^2\gamma (1+\gamma^\delta \vartheta_\gamma)}{1-\gminorc} \ps{x}{v}
    & \leq \frac{\kappa^2}{2}(1+\bvartheta)\left(\frac1\kappa+\left[\frac{2C_\kappa}{\kappa^2}+1\right]\gamma\right)\left(\norm{x}^2+\norm{v}^2\right) \\
    & \leq \frac12 (1+\bvartheta)(\kappa^2+\kappa+2C_\kappa)(\norm{x}^2+\norm{v}^2), \label{eq:bound_cross_term_Wgamma}
  \end{align}
  where we used~$\gamma \leq 1$ in the last inequality. This finally implies the first inequality in \eqref{eq:upper_bound_QF_Lya} by the definition~\eqref{eq:def_function_lyap_super_quad} of~$\Lya$. 
  The proof of the second one is concluded with the inequality~$\sqrt{a+b}\leq \sqrt{a}+\sqrt{b}$ for~$a,b\in \rset_+$.

  \subsection{Proof of \Cref{lem:phi_lip}}

    In view of the definitions~\eqref{eq:def_function_lyap_super_quad} and~\eqref{eq:def_Lyaexp}, it holds
  \[
  \nabla_x \LSfonction(x,v) = \frac{1}{2\LSfonction(x,v)}\left[\kappa^2x + \frac{\kappa^2\gamma (1+\gamma^\delta \vartheta_\gamma)}{1-\gminorc} v + 2\parametreGradU\nabla U(x)\right],
  \]
  so that, by a triangle inequality and upon bounding the prefactor of~$v$ as in~\eqref{eq:bound_cross_term_Wgamma}, and using also \Cref{lem:minoration_lya} and the inequality~$\norm{\nabla U(x)} \leq L\norm{x}$, we obtain
  \begin{align}
    \norm{\nabla_x \LSfonction(x,v)}
    & \leq \frac{\kappa^2\norm{x}+(1+\bvartheta)(\kappa^2+\kappa+2C_\kappa)\norm{v}+2\parametreGradU\norm{\nabla U(x)}}{2\sqrt{1+\cLya \{\norm[2]{x} + \norm[2]{v}\}}}\\
    &\leq \frac{1}{\sqrt{\cLya}}\max\left(\parametreGradU L + \frac{\kappa^2}{2}, \frac12(1+\bvartheta)(\kappa^2+\kappa+2C_\kappa) \right).
  \end{align}
  Similarly, for any $x,v \in \rset^d$,
  \begin{equation}
    \norm{\nabla_v \LSfonction(x,v)}\leq \frac{1}{\sqrt{\cLya}}\max \parenthese{1,\frac12(1+\bvartheta)(\kappa^2+\kappa+2C_\kappa) } \eqsp .
  \end{equation}
  The conclusion then follows from the inequality~$\norm{\nabla \LSfonction(x,v)} \leq 2 \max(\norm{\nabla_x \LSfonction(x,v)},\norm{\nabla_v \LSfonction(x,v)})$ and the mean value theorem.

  \subsection{Proof of \Cref{lem:Gammaminorc_lip}}

    In view of~\eqref{minor_def_Gamma_b_gamma}, \Cref{ass:gminorc} and \Cref{ass:w}-\ref{ass:w_logsob_spec_1}-\ref{ass:w_lip}, we can write, for any $x,v,z,z'\in\rset^d$, $w_1,w'_1\in \rset^{m_1}$ and $w_2\in\rset^{m_2}$,
  \begin{align}
    & \norm{\Gammaminorc{\gamma}\left(x,v,\sqrt{\gamma}\sigma_{\gamma} z, (w_1,w_2)\right)-\Gammaminorc{\gamma}\left(x,v,\sqrt{\gamma}\sigma_{\gamma} z', (w'_1,w_2)\right)} \\
    & \qquad \leq \gamma \norm{f_\gamma\left(x, \gamma^{\delta} v,\gamma^{1/2+\delta} \sigma_{\gamma} z,(w_1,w_2)\right)-f_\gamma\left(x, \gamma^{\delta} v,\gamma^{1/2+\delta} \sigma_{\gamma} z',(w_1',w_2)\right)} + \gamma^{\delta+1/2}\sigma_\gamma\norm{\Dbf_{\gamma}(z-z')} \\
    & \qquad \quad + \gamma \norm{g_\gamma\left(x, \gamma^{\delta}v,\gamma^{\delta+1/2}\sigma_\gamma z,(w_1,w_2)\right)-g_\gamma\left(x, \gamma^{\delta}v,\gamma^{\delta+1/2}\sigma_\gamma z',(w_1',w_2)\right)} + \sigma_\gamma \sqrt{\gamma}\norm{z-z'} \\
    & \qquad \leq \left(2\LLS\gamma \max(1,\gamma^{\delta+1/2}\bsigma)+\gamma^{\delta+1/2}\mathscr{D}\bsigma+\bsigma\sqrt{\gamma}\right) \left( \norm{z-z'} + \norm{w_1-w'_1} \right) \eqsp , \\
    & \qquad \leq \sqrt{2}\left(2\LLS\gamma \max(1,\gamma^{\delta+1/2}\bsigma)+\gamma^{\delta+1/2}\mathscr{D}\bsigma+\bsigma\sqrt{\gamma}\right) \norm{(z,w_1)-(z',w_1')} \eqsp ,
  \end{align}
  which completes the proof.

  \subsection{Proof of \Cref{propo:verif_drift_condition_splitting_scheme}}
  \label{sec:check_D3}

  In view of~\eqref{eq:second_order_splitting_complicated} and since $b = - \nabla U$,
\[
\begin{aligned}
  f_\gamma(x,v,z,w)
  & = \mathscr{C}_{1,\gamma} v - \frac{\gamma}{2} \nabla U\left(x+ \mathscr{C}_{2,\gamma} v + \mathscr{C}_{3,\gamma} z + \gamma^{3/2}\mathscr{C}_{4,\gamma} w\right) + 2\gamma^{3/2} \mathscr{C}_{4,\gamma} w\eqsp,
\end{aligned}
\]
with
\[
\mathscr{C}_{1,\gamma} = \frac{\rme^{-\kappa \gamma/2}-1}{\gamma},
\qquad
\mathscr{C}_{2,\gamma} = \frac{\rme^{-\kappa \gamma/2}}{2},
\qquad
\mathscr{C}_{3,\gamma} = \frac{\rme^{-\kappa\gamma/2}}{2(1+\rme^{-\kappa\gamma})},
\qquad
\mathscr{C}_{4,\gamma} = \sqrt{\frac{\tsigma_{\gamma/2}^2}{8(1+\rme^{-\kappa\gamma})}}.
\]
The coefficients~$\mathscr{C}_{i,\gamma}$ (for~$1 \leq i \leq 4$) are uniformly bounded in~$\gamma$ for~$\gamma \in (0,\bgamma]$, and we denote by~$\overline{\mathscr{C}}$ their maximal value:
\begin{equation}
  \label{eq:s1}
  \overline{\mathscr{C}} = \sup_{1 \leq i \leq 4} \sup_{\gamma \in (0,\bgamma]} \mathscr{C}_{i,\gamma} < +\infty \eqsp.
  \end{equation}
A Cauchy--Schwarz inequality gives 
  \[
  \begin{aligned}
  \norm{f_\gamma(x,\gamma^\delta v,\gamma^{\delta+\half} \sigma_\gamma z,w)}^2 & \leq 3 \gamma^{2\delta} \overline{\mathscr{C}}^2 \norm{v}^2 + 12 \gamma^3 \overline{\mathscr{C}}^2 \norm{w}^2 \\
  & \qquad + \frac{3\gamma^2}{4} \norm{\nabla U\left(x+ \gamma^{\delta} \mathscr{C}_{2,\gamma} v + \gamma^{\delta + \half} \sigma_\gamma \mathscr{C}_{3,\gamma} z + \gamma^{3/2}\mathscr{C}_{4,\gamma} w\right)}^2\eqsp.
  \end{aligned}
  \]
  It therefore suffices to bound the term on the second line of the previous inequality. To this end, we note that~\Cref{ass:lip}(U) and a Cauchy--Schwarz inequality imply that, for any~$h \in \rset^d$,
  \begin{equation}
    \label{eq:bound_nablaU_x+h}
    \norm{\nabla U(x+h)}^2 \leq 2 \norm{\nabla U(x)}^2 + 2L^2 \norm{h}^2 \eqsp.
  \end{equation}
  The first condition in~\Cref{ass:contoleGAndF}(U) is then easily seen to hold upon setting~$h = \gamma^{\delta} \mathscr{C}_{2,\gamma} v + \gamma^{\delta + \half} \sigma_\gamma \mathscr{C}_{3,\gamma} z + \gamma^{3/2}\mathscr{C}_{4,\gamma} w$.

  To prove that the second condition in~\Cref{ass:contoleGAndF}(U) holds, we need to be careful about the dependence of our estimates on~$\norm{x}$. We rely on~\eqref{eq:19}, which implies that there exist~$a>0$ and~$b \in \rset $ such that
  \[
  - \ps{x}{\nabla U(x)}\leq -a \left( \norm{x}+\norm[2]{\nabla U(x)}\right) + b.
  \]
  An inequality similar to~\eqref{eq:bound_nablaU_x+h} can also be written for any~$h \in \rset^d$: 
  \[
  -2\norm{\nabla U(x+h)}^2 \leq -\norm{\nabla U(x)}^2 + 2L^2 \norm{h}^2 \eqsp.
  \]
  We therefore obtain, using a Cauchy--Schwarz inequality, for any~$h \in \rset^d$,
  \begin{align}
    - \ps{x}{\nabla U(x+h)} & = - \ps{x+h}{\nabla U(x+h)} + \ps{h}{\nabla U(x+h)} \\
    & \leq -a \left( \norm{x+h}+\norm[2]{\nabla U(x+h)}\right) + b + \frac{a}{2} \left( \frac{\norm{h}^2}{a^2} + \norm{\nabla U(x+h)}^2\right) \\
    & \leq -a \left( \norm{x}+\frac14 \norm[2]{\nabla U(x)}\right) + \widetilde{b}\left(1+\norm{h}^2\right) \eqsp, \label{eq:ineq_diss_D3}
  \end{align}
  for some constant $\widetilde{b} \in\rset$.
  The second condition in~\Cref{ass:contoleGAndF}(U) then follows from the above inequality, the fact that 
  \[
  \begin{aligned}
    \ps{x}{f_\gamma\left(x,\gamma^{\delta}v, \gamma^{\delta+\half} \sigma_{\gamma} z,w\right)}
    & = \gamma^\delta\mathscr{C}_{1,\gamma} \ps{x}{v} + 2\gamma^{3/2} \mathscr{C}_{4,\gamma} \ps{x}{w} - \frac{\gamma}{2} \ps{x}{\nabla U\left(x+h\right)} 
  \end{aligned}
  \]
  with~$h= \gamma^{\delta}\mathscr{C}_{2,\gamma} v + \gamma^{\delta+\half} \sigma_\gamma \mathscr{C}_{3,\gamma} z + \gamma^{3/2}\mathscr{C}_{4,\gamma} w$ and using~\eqref{eq:s1}.

  Let us next check that the conditions in \Cref{ass:contoleGAndF}(U) involving~$g_\gamma$ are satisfied with $\parametreGradU = 1$. In view of~\eqref{eq:second_order_splitting_complicated} and since $b = - \nabla U$,
  \[
  g_\gamma(x,v,z,w) = -\mathscr{G}_{1,\gamma} \nabla U\left(x+\mathscr{G}_{2,\gamma} v + \mathscr{G}_{3,\gamma} z + \gamma^{3/2}\mathscr{G}_{4,\gamma} w\right),
  \]
  with
  \[
  \mathscr{G}_{1,\gamma} = \rme^{-\kappa \gamma/2},
  \qquad
  \mathscr{G}_{2,\gamma} = \frac{\rme^{-\kappa \gamma/2}}{2},
  \qquad
  \mathscr{G}_{3,\gamma} = \frac{\rme^{-\kappa\gamma/2}}{2(1+\rme^{-\kappa\gamma})},
  \qquad
  \mathscr{G}_{4,\gamma} = \sqrt{\frac{\tsigma_{\gamma/2}^2}{8(1+\rme^{-\kappa\gamma})}}.
  \]
The coefficients~$\mathscr{G}_{i,\gamma}$ (for~$1 \leq i \leq 4$) are uniformly bounded in~$\gamma$ for~$\gamma \in (0,\bgamma]$, and we denote by~$\overline{\mathscr{G}}$ their maximal value:
\begin{equation}
  \label{eq:s2}
  \overline{\mathscr{G}} = \sup_{1 \leq i \leq 4} \sup_{\gamma \in (0,\bgamma]} \mathscr{G}_{i,\gamma} < +\infty \eqsp.
\end{equation}
Note also that, there exists $K \geq 0$ such that for any~$\gamma \in (0,\bgamma]$,
      \begin{equation}
        \label{eq:G0_estimate}
        |\mathscr{G}_{1,\gamma} - 1| \leq K\gamma \eqsp,
      \end{equation}
      so that we bound using \Cref{ass:lip}(U) the term involving~$g_\gamma$ in the first condition as
    \[
      \begin{aligned}
        \norm{g_\gamma\left(x,\gamma^{\delta}v, \gamma^{\delta+\half} \sigma_{\gamma} z,w\right)+\nabla U(x)}^2 & \leq 2 \left(1-\mathscr{G}_{1,\gamma}\right)^2 \norm{\nabla U(x)}^2 + 2\overline{\mathscr{G}}^2 \norm{\nabla U(x+h)-\nabla U(x)}^2 \\
        & \leq 2 K^2 \gamma^2 \norm{\nabla U(x)}^2 + 2\overline{\mathscr{G}}^2 L^2 \norm{h}^2 \eqsp,
        \end{aligned}
    \]
    with $h = \gamma^{\delta} \mathscr{G}_{2,\gamma} v + \gamma^{\delta+\half} \sigma_{\gamma} \mathscr{G}_{3,\gamma} z + \gamma^{3/2}\mathscr{G}_{4,\gamma} w$, which easily implies that the second condition holds by \eqref{eq:s2}. Moreover, with the same definition of~$h$,
    \[
    \ps{x}{g_\gamma\left(x,\gamma^{\delta}v, \gamma^{\delta+\half} \sigma_{\gamma} z,w\right)} = -\mathscr{G}_{1,\gamma} \ps{x}{\nabla U(x+h)}, 
    \]
    from which the third condition easily follows in view of~\eqref{eq:ineq_diss_D3}, \eqref{eq:s2} and~\eqref{eq:G0_estimate}.

\section{Complementary and technical results}
\label{sec:compl-techn-results}

\begin{lemma}
\label{prop:inv}
For any $\kappa, \sigma, \gamma>0$, $\coeffContinueMatrix{}{\gamma} \otimes \Idd$ in \eqref{eq:def_continue_matrix} is invertible.
\end{lemma}

\begin{proof}
Note first that we only need to consider the case $\sigma=\kappa=1$. Let $\gamma>0$. Set $\mathrm{M} = \coeffContinueMatrix{}{\gamma} \otimes \Idd$. We show that $\det(\mathrm{M}) >0$.
Since by  \cite[Exercise 45, Chapter 1]{hiai:petz:2014}, $\det(\mathrm{M})= \det( \coeffContinueMatrix{}{\gamma})^d$ where $\coeffContinueMatrix{}{\gamma}$ is given by \eqref{eq:def_continue_matrix}, it suffices to show that $\det( \coeffContinueMatrix{}{\gamma})>0$.
Denote for any $t>0$ and square-integrable functions $h_1,h_2:\ccint{0,t}\to \rset$,
\[
\ps{h_1}{h_2}_{\rml^2(\ccint{0,t})}=\int_0^t h_1(s)h_2(s) \, \rmd s,
\qquad
\norm{h_1}_{\rml^2(\ccint{0,t})}=\sqrt{\ps{h_1}{h_1}_2}.
\]
By \eqref{eq:def_continue_matrix},
\begin{equation}
  \det( \coeffContinueMatrix{}{\gamma})=\det \begin{pmatrix}
    \norm{h_1}_{\rml^2(\ccint{0,\gamma})}^2&\ps{h_1}{h_2}_{\rml^2(\ccint{0,\gamma})}\\
    \ps{h_1}{h_2}_{\rml^2(\ccint{0,\gamma})}&\norm{h_2}_{\rml^2(\ccint{0,\gamma})}^2
  \end{pmatrix}=\norm{h_1}_{\rml^2(\ccint{0,\gamma})}^2\norm{h_2}_{\rml^2(\ccint{0,\gamma})}^2-\ps{h_1}{h_2}_{\rml^2(\ccint{0,\gamma})}^2 \eqsp ,
\end{equation}
where for any $s\in \ccint{0,\gamma}$, $h_1(s)=1-\rme^{- (\gamma-s)}$ and $h_2(s)=\rme^{- (\gamma-s)}$. The result follows by a Cauchy–Schwarz inequality since $h_1, h_2$ are linearly independent. 
\end{proof}

\begin{lemma}
  \label{lem:OU_equiv}
Let $x,v\in\rset^d$ and $\kappa>0$. Consider for any $t\in \rset_+$, 
\begin{equation}
  \label{eq:def_discretization}
  \begin{aligned}
  \widetilde{\bfX}_{t}&=x+\frac{1-\rme^{-\kappa t}}{\kappa} v+ \frac{\kappa t+\rme^{-\kappa t}-1}{\kappa^2}   b (x) +\sigma \int_{0}^{t} \frac{1-\rme^{-\kappa (t-s)}}{\kappa} \rmd \BM_{s} \eqsp ,\\
\widetilde{\bfV}_{t} &= \rme^{-\kappa t} v+\frac{1-\rme^{-\kappa t}}{\kappa}  b (x) +\sigma\int_{0}^{t} \rme^{-\kappa (t-s)} \rmd \BM_{s} \eqsp ,
\end{aligned}
\end{equation}
where $(\BM_t)_{t \geq 0}$ is a standard $d$-dimensional Brownian motion.
 The process $(\widetilde{\bfX}_{t},\widetilde{\bfV}_{t})_{t\geq 0}$ is the unique solution of the SDE,
  \begin{equation}
    \label{eq:sde_discretization}
    \widetilde{\bfX}_t =  x + \int_{0}^t \widetilde{\bfV}_s \, \rmd s    \eqsp,
    \qquad
    \widetilde{\bfV}_t =  v + \int_{0}^t \defEns{-\kappa \widetilde{\bfV}_s +  b(x) } \, \rmd s  + \sigma\BM_t  \eqsp .
  \end{equation}
\end{lemma}

\begin{proof}
  For any $ t \in \rset_+$, by \eqref{eq:def_discretization}, \Cref{lem:fubini_sto} and linearity, 
  \begin{align}
    \int_{0}^{t} \widetilde{\bfV}_s \, \rmd s &=\int_{0}^{t}\rme^{-\kappa s} v+\frac{1-\rme^{-\kappa s}}{\kappa}  b (x) \, \rmd s +\sigma\int_{0}^{t}\int_{0}^{s} \rme^{-\kappa (s-r)} \, \rmd \BM_{r} \, \rmd s\\
    \label{eq:prop_solve_1}
    &=\frac{1-\rme^{-\kappa t}}{\kappa} v+ \frac{\kappa t+\rme^{-\kappa t}-1}{\kappa^2}   b (x)+\sigma\int_{0}^{t}\int_{r}^{t} \rme^{-\kappa (s-r)} \, \rmd s \, \rmd \BM_{r} \\
    &=\widetilde{\bfX}_t-x\eqsp .
  \end{align}
  In addition, using \eqref{eq:prop_solve_1}, we obtain 
  \begin{align}
    \int_{0}^t -\kappa\widetilde{\bfV}_s +b(x) \rmd s&=(\rme^{-\kappa t}-1) v-\frac{\rme^{-\kappa t}-1}{\kappa}   b (x) +\sigma \int_{0}^{t} (\rme^{-\kappa (t-s)}-1) \, \rmd \BM_{s} = \widetilde{\bfV}_t-v-\sigma \int_{0}^{t} \rmd \BM_{s} \eqsp,
  \end{align} 
  which completes the proof. 
\end{proof}

The following Fubini-type result for stochastic integrals is established in \cite[Theorem 1]{kailath1978fubini} (see also~\cite[Chapter~IV, Exercise~(5.17)]{revuz:yor:1999}), but an alternative proof is given here for completeness.

\begin{lemma} 
  \label{lem:fubini_sto}
  For any $f \in \rmC^1(\rset)$, $g \in \rmC^0(\rset)$, $u,v\geq 0$, 
  \begin{equation}
    \label{eq:fubini_sto}
    \int_{u}^{v} \int_{u}^{v}  \1_{ \rset_+}(t-s) f(s)g(t) \, \rmd B_s \, \rmd t = \int_{u}^{v} \int_{u}^{v}  \1_{ \rset_+}(t-s) f(s)g(t) \, \rmd t \, \rmd B_s.
  \end{equation}
\end{lemma}

\begin{proof}
  Consider $f \in \rmC^1(\rset)$, $g \in \rmC^0(\rset)$, $v\geq 0$. Without loss of generality, it is sufficient to show \eqref{eq:fubini_sto} for  $u=0$. 
  Introduce, for any $w\geq 0$, $G(w)=\int_0^w g(t) \,\rmd t$ and $M_w = f(w)B_w$. By integration by parts \cite[Chapter IV, Proposition~(3.1)]{revuz:yor:1999}, we have, for any $w \geq 0$, 
  \begin{align}
  \label{eq:ibp}
  M_w&=\int_0^w f'(s)B_s \, \rmd s +\int_0^w f(s) \, \rmd B_s \eqsp ,\\
  \label{eq:ibpbis}
  G(w)M_w  & =\int_0^w g(s)f(s)B_s \, \rmd s +\int_0^w G(s) \, \rmd M_s \eqsp .
  \end{align}
  Then, by~\eqref{eq:ibp},
  \begin{equation}
    \label{eq:ibp_2}
    \int_{0}^{v} \int_{0}^{v}  \1_{\rset_+}(t-s) f(s)g(t) \,\rmd B_s \,\rmd t = \int_{0}^{v} \left(\int_{0}^{t}  f(s) \,\rmd B_s \right) g(t) \,\rmd t = \int_{0}^{v} \left(M_t  -\int_{0}^{t} f'(s)B_s \,\rmd s\right) g(t) \,\rmd t \eqsp ,
  \end{equation}
  and 
  \begin{align}
    \int_{0}^{v} \int_{0}^{v}  \1_{\rset_+}(t-s) f(s)g(t) \,\rmd t \,\rmd B_s &=\int_{0}^{v} \left(G(v) - G(s) \right) f(s) \,\rmd B_s\\
    \label{eq:ibp_3}
  &=\int_{0}^{v}\left(G(v) - G(s) \right) \,\rmd M_s-\int_{0}^{v}  \left(G(v) - G(s) \right)  f'(s)B_s \,\rmd s\eqsp. 
  \end{align}
  By Fubini's theorem, almost surely we have $  \int_{0}^{v} \int_{0}^{t} f'(s)B_s \,\rmd s g(t) \,\rmd t =  \int_{0}^{v}  \left(G(v) - G(s) \right)  f'(s)B_s \,\rmd s$.
  Therefore using this result and \eqref{eq:ibp_2}-\eqref{eq:ibp_3},  \eqref{eq:fubini_sto} holds if $\int_{0}^{v} M_t  g(t) \,\rmd t = \int_{0}^{v}\left(G(v) - G(t) \right) \,\rmd M_t$.
  which follows from~\eqref{eq:ibpbis} as
  \begin{align}
    \int_{0}^{v}\left(G(v) - G(t) \right) \rmd M_t&= G(v) M_v-\int_{0}^{v} G(t) \,\rmd M_t =\int_{0}^{v} g(t)f(t)B_t \,\rmd  t  = \int_{0}^{v} M_t  g(t) \,\rmd t  \eqsp.
  \end{align}
  This allows to conclude the proof.
\end{proof}

\end{document}